\documentclass[reqno]{amsart}

\usepackage{amssymb,amsmath,amsthm,amsfonts,color} 

\usepackage[left=3cm,right=3cm,top=3cm,bottom=3cm]{geometry} 
\usepackage{enumitem}
\usepackage{comment,graphicx,setspace}
\usepackage{wrapfig}

\usepackage[colorlinks=true, pdfstartview=FitV, linkcolor=blue, citecolor=blue, urlcolor=blue]{hyperref}

\allowdisplaybreaks
\numberwithin{equation}{section}

\theoremstyle{plain}
\newtheorem{theorem}{Theorem}[section]
\newtheorem{proposition}[theorem]{Proposition}
\newtheorem{lemma}[theorem]{Lemma}
\newtheorem{corollary}[theorem]{Corollary}

\theoremstyle{definition}
\newtheorem{definition}[theorem]{Definition}
\newtheorem{remark}[theorem]{Remark}

\newcommand{\C}{\mathbb{C}}

\newcommand{\R}{\mathbb{R}}
\newcommand{\N}{\mathbb{N}}
\renewcommand{\S}{\mathbb{S}}

\newcommand{\bN}{\mathbf{N}}

\newcommand{\bM}{\mathbf{M}}

\newcommand{\fm}{\mathtt{v}}

\newcommand{\cH}{\mathcal{H}}
\newcommand{\cL}{\mathcal{L}}

\newcommand{\cS}{\mathcal{S}}
\newcommand{\cF}{\mathcal{F}}
\newcommand{\cG}{\mathcal{G}}
\newcommand{\cW}{\mathcal{W}}
\newcommand{\cU}{\mathcal{U}}

\newcommand{\p}{\partial}
\newcommand{\wk}{\rightharpoonup}
\newcommand{\cl}[1]{\overline{#1}}
\newcommand{\bound}{b}
\renewcommand{\d}{\mathrm{d}}
\newcommand{\e}{{\bf e}}
\renewcommand{\hat}{\widehat}
\newcommand{\strictlyincluded}{\subset\subset}
\newcommand{\res}{\mathop{\hbox{\vrule height 9pt width 0.5pt depth 0pt
\vrule height 0.5pt width 5pt depth 0pt}}\nolimits}

\renewcommand{\tilde}{\widetilde}
\newcommand{\supp}{\mathrm{supp}\,}
\newcommand{\dist}{\mathrm{dist}}

\newcommand{\loc}{\mathrm{loc}}
\newcommand{\Lip}{\mathrm{Lip}}
\newcommand{\Int}[1]{\mathrm{Int}(#1)}
\newcommand{\tr}{\mathrm{tr}}

\newcommand{\aplim}{\mathop{\mathrm{ap\,\, lim}}}

\newcommand{\openset}{O}

\newcommand{\str}[1]{\mathcal{E}#1}

\newcommand{\mtwo}{\mathbb{M}^{n\times n}_{\rm sym}}
\newcommand{\admissible}{\mathcal{C}}
\newcommand{\substrate}{S}
\newcommand{\Ins}[1]{{\rm Int}{(#1\cup \substrate\cup \Sigma)}}

\title[Existence for the SDRI model in $\R^n$]{Existence of minimizers for the SDRI model in $\R^n$: Wetting and dewetting regimes with mismatch strain}

\author[Sh. Kholmatov] {Shokhrukh Yu. Kholmatov} 
\address[Sh. Kholmatov]{University of Vienna, Oskar-Morgenstern Platz 1, 1090 Vienna (Austria) \& Samarkand State University, Univ. Boulevard 15, 140104 Samarkand (Uzbekistan)}
\email{shokhrukh.kholmatov@univie.ac.at}

\author[P. Piovano] {Paolo Piovano} 
\address[P. Piovano]{Politecnico di Milano, P.zza Leonardo da Vinci 32, 20133 Milano (Italy)}
\email{paolo.piovano@polimi.it}

\subjclass[2020]{49J45, 35R35, 74G65}

\keywords{Minimal configurations, elastic energy, surface energy,  mismatch strain, Dirichlet boundary conditions, existence, regularity, compactness, lower semicontinuity, density estimates,  SDRI, interface instabilities, thin films, crystal cavities, fractures.}

\date{\today}
\begin{document}

\begin{abstract} 
The existence and the regularity results obtained in  \cite{HP:2021_arxiv} for the variational model introduced in \cite{HP:2020_arma} to study the optimal shape of crystalline materials in the setting of stress-driven rearrangement instabilities (SDRI) are extended from two dimensions to any dimensions $n\geq2$. The energy is the sum of the elastic and the surface energy contributions, which cannot be decoupled, and depend on configurational pairs consisting of a set and a function that model the region occupied by the crystal and  the bulk displacement field, respectively. By following the physical  literature,  the  ``driving stress'' due to the mismatch  between the ideal free-standing equilibrium lattice of the crystal with respect to adjacent materials is included in the model by considering a discontinuous mismatch strain in the elastic energy. Since two-dimensional methods and the methods used in the previous literature where Dirichlet boundary conditions instead of the mismatch strain and only the wetting regime were considered, cannot be employed in this setting, we proceed differently, by including in the analysis the dewetting regime and carefully analyzing the fine properties of energy-equibounded sequences. This analysis allows to establish both a compactness property in the family of admissible configurations and the lower semicontinuity of the energy with respect to the topology induced by the $L^1$-convergence of sets and a.e.\ convergence of displacement fields, so that the direct method  can be applied. We also prove that our arguments work as well in the setting with Dirichlet boundary conditions.
\end{abstract}

\maketitle

\section{Introduction}

Elastic effects can strongly affect the structure of crystalline materials by inducing morphological destabilizations  from the optimal free-standing crystalline equilibrium, that are often referred to as the family of \emph{stress-driven rearrangement instabilities} (SDRI)  \cite{AT:1972,D:2001,FG:2004,Gr:1993,Srolovitz:1989}. In order to relieve the strain, atoms move from their crystalline order possibly inducing  both bulk deformations and interface irregularities. The latter can be originated in various forms, such as the roughness of the exposed crystalline boundaries, the formation of internal cracks in the bulk, the nucleation of dislocations in the crystalline lattice, and the delamination at contact edges with adjacent materials. However, such corrugations and extra boundary interfaces are not favorable with respect to the surface energy, which would instead prescribe regular specific Wulff/Winterbottom-type shapes \cite{PV:2022,PV:2023,Winterbottom,Wulff01}. Therefore, the surface energy competes against the destabilizing effect of the elastic energy with a regularizing effect: a delicate microscopical compromise between such opposite mechanisms  must then be reached strongly affecting in a variety of ways the original crystalline-material macroscopical properties.  

In the strive of capturing such interplay between elastic and (anisotropic) surface energy described by the physical literature  \cite{Deng:1995,HS:1991,Baldelli:2013,SMV:2004,Spencer:1999,WL:2003,Xia:2000}, various mathematical  models with a variational nature have been introduced in relation to the different  settings relevant for the applications. A non-exhaustive list includes \cite{BGZ:2015,BCh:2002,DP:2018_1,DP:2018_2,FFLM:2007,GZ:2014,KP:2021} for epitaxially-strained thin films deposited on supporting materials, \cite{BFM:2008,CCF:2016,FGL:2019} for fractures, \cite{Babadjian:2016,Baldelli:2014} for delamination, and, e.g.,  \cite{FFLM:2011} for crystalline cavities. Establishing the existence of minimizers for such models even in dimension $n=2$ is a challenging task especially due to compactness issues. Such issues were first solved in simplified settings, by working under the antiplane-shear assumption \cite{Bonaci:2015,ChS:2007}, or by distinguishing the applications with \emph{adhoc} geometric assumptions on the morphology of the crystalline materials, such as adopting graph-type  and star-shapedness constraints on film profiles and crystal cavities, respectively.
However, numerous experiments in dewetting and crystal growth exhibit morphologies that cannot be represented as graphs, due to the formation of holes, rims with overhangs, merging fronts, pinch-off events, and other topological changes. For instance, the dewetting experiments of \cite{R:1992,R:1994} clearly illustrate the nucleation of holes, rim growth, coalescence into cellular patterns, and eventual breakup into droplets. Similarly, solid-state dewetting experiments \cite{KGT:2009,RG:2021} reveal faceted rims and islands, sharp corners, and pinch-off events, while related phenomena are also documented in \cite{SS:1986}.
These observations motivate the need for a geometric formulation of the SDRI model that does not rely on a graph representation.

More recently, the development of several techniques  related to GSBD-functions, a specific subclass of functions of bounded deformation \cite{D:2013_jems}, have been sucessfully applied to models related to the  Griffith energy \cite{CCF:2016,CCI:2017,ChC:2019_arxiv,ChC:2020_jems,D:2013_jems,FGL:2019}. Following this progress, there has been a growing effort \cite{ChC:2020_arxiv,CF:2020_arma,HP:2020_arma,HP:2021_arxiv} to develop  mathematical frameworks enabling the simultaneous treatment of the various mechanisms of mass rearrangement and boundary instabilities, which is of crucial importance, as often such phenomena concomitantly occur in applications.

The aim of  this paper is to extend to dimension  $n\ge2$, and hence including the physical relevant case of $n=3$, the existence and the regularity results  obtained in \cite{HP:2021_arxiv} for $n=2$ for the SDRI model introduced in  \cite{HP:2020_arma}.  In regard of the existence, such an extension was previously achieved in \cite{CF:2020_arma}  for the  \emph{wetting regime}, i.e., the case  for which it is more convenient for the crystal material to always cover the surface of a (supporting) adjacent material rather than letting it exposed,  and the setting in which the stress driving effect characterizing SDRI is mathematically prescribed by introducing  \emph{boundary Dirichlet conditions}. Here we address also the dewetting regime and,  as previously done by the authors in \cite{HP:2020_arma,HP:2021_arxiv} for $n=2$, by following  the physical literature \cite{AT:1972,D:2001,FG:2004,Gr:1993,Spencer:1999,Srolovitz:1989,Xia:2000} we avoid the use of any Dirichlet boundary conditions and we directly introduce a \emph{mismatch strain}  in the elastic energy. As suggested by its name, such strain is induced in the \emph{free crystal}, i.e., the crystal of which we are studying the morphology,  by the mismatch between its ideal free-standing equilibrium lattice and the lattice of adjacent materials. Since the approach used in \cite{CF:2020_arma} cannot be applied to this setting without boundary conditions as it is described below (see also \cite{HP:2021_arxiv}), we have developed an alternative strategy that allows us to tackle both the case with mismatch strain and the one with Dirichlet conditions (see Remark \ref{rem:literature_models} for more details).   Finally, the method of this paper extends (also to both the settings with and without Dirichlet conditions) the regularity results for the bulk displacements and the morphologies of the energy minimizing configurations obtained by the authors in \cite{HP:2021_arxiv} for $n=2$  (besides extending the existence results of \cite{HP:2021_arxiv} to the presence of different adjacent materials and to Griffith-type models with mismatch strain and delamination).

To facilitate this generalization, we adopt the terminology introduced in \cite{HP:2020_arma,HP:2021_arxiv}, by referring to the  bounded region $\Omega$ in the space $\R^n$ where the free crystal is located as the \emph{container} in analogy to capillarity problems,  and to the region $S$  occupied by adjacent materials outside the container, i.e., $S\subset\R^n\setminus \Omega$,  as the  \emph{substrate} in analogy to the thin-film setting where $S$ is the supporting material on which the film is being deposited. We notice that the  \emph{contact region} between the container and the substrate  $\Sigma:=\p\Omega\cap\p \substrate$ is assumed to be a Lipschitz $(n-1)$-manifold and that $S$ can be given by a finite number of different connected components possibly modeling different adjacent materials. The free crystals are then represented by configurational pairs of set-function type $(A,u)$, where $A\subset\Omega$ is a set of finite perimeter denoting the region occupied by the free crystal and subject to the volume constraint $|A|=\fm$ with $\fm\in(0,|\Omega|]$,  and $u$ is a vector valued faction in $GSBD^2({\rm Int}(A\cup\Sigma\cup S)) \cap H_\loc^1(\substrate)$ denoting the displacement field of the free-crystal and substrate bulk materials with respect to their optimal equilibrium arrangements. The family of all such admissible configurational pairs $(A,u)$ is denoted by $\mathcal{C}$.

The configurational energy of any free-crystal pair $(A,u)\in\mathcal{C}$ is defined by
\begin{equation}\label{ebergy:sauri}
\cF(A,u) =  \cW(A,u) +\cS(A,u),
\end{equation}
where $\cS$ and $\cW$ represent the  elastic  and the surface energy, respectively. The elastic energy $\cW$ in \eqref{ebergy:sauri} is defined as in \cite{FFLM:2007} by
$$
\cW(A,u) = \int_{A\cup S} \C(x)[\str{u}-\bM_0]:[\str{u}-\bM_0]\,\d x,
$$
where $\C$ is a  bounded measurable tensor-valued map $\C$ in $\Omega\cup S$ satisfying the coercivity assumption $\C\ge c\,\mathbb{I}>0$ (in the sense of linear operators), where $\mathbb{I}$ is the identity tensor, $\str{u}$ is the approximate symmetric gradient of $u$ (see \eqref{approx_sym_gradents}) and 
$
\bM_0
$ 
is the (discontinuous) \emph{mismatch strain} defined as 
\begin{equation}\label{mismatchstrain}
\bM_0 = 
\begin{cases}
\str{u_0} & \text{in $\Omega,$}\\
0 & \text{in $S$}
\end{cases}
\end{equation}
for some fixed $u_0\in H^1(\R^n)$. In the special case, when the equilibrium lattice of the free crystal and of the substrate matches at $\Sigma$, we take $u_0\equiv0$. The surface energy $\cS$ in \eqref{ebergy:sauri} is defined as
$$
\cS(A,u):=\int_{\p^*A \cup J_u} \psi(x,\nu(x))\d\cH^{n-1},
$$
where $\p^*A$ is the reduced boundary of $A$, $J_u$ is the jump set of $u$, and the surface energy density $\psi(\cdot,\nu(\cdot))$  is given by 
\begin{equation}\label{surergy}
\psi(x,\nu(x)):=
\begin{cases}
\varphi(x,\nu_A(x)) & \text{if $x\in \Omega\cap \p^*A,$}\\
2\varphi(x,\nu_{J_u}(x)) & \text{if $x\in A^{(1)}\cap J_u,$}\\
\beta(x) & \text{if $x\in [\Sigma\cap \p^*A]\setminus J_u,$}\\
\varphi(x,\nu_\Sigma(x)) & \text{if $x\in \Sigma\cap \p^*A \cap J_u$,}
\end{cases}
\end{equation}
where $\nu_U(x)$  denotes the outward-pointing normal vector to $U$ at $x\in\p^*U$ for any set of finte perimeter $U\subset\R^n$, $\nu_\Sigma:=\nu_S$, $\nu_{J_u}$ is the  normal on $J_u$,  $A^{(1)}$   is the set of points of density $1$ for $A$, $\varphi\in C(\cl{\Omega}\times \R^n)$ is a Finsler norm denoting the  \emph{anisotropic surface tension}   of the free-crystal material, and  $\beta\in L^\infty(\Sigma)$ represents the  \emph{relative adhesion coefficient} of $\Sigma$ for which we assumed, as in capillarity theory (see, e.g., \cite{DPhM:2015}), that
\begin{equation*}
|\beta(x)| \le \varphi(x,\nu_\Sigma)\quad\text{for a.e.\ $x\in\Sigma$.} 
\end{equation*}

Notice that the weights in \eqref{surergy}, which forbid to  decouple  the surface energy from the elastic energy making the energy $\cF$ highly nonlocal, are consistent with the ones chosen in  \cite{CF:2020_arma,FFLM:2007,FFLM:2011,HP:2020_arma,HP:2021_arxiv},  where they were crucial to prove energy lower semicontinuity-type properties. In particular, the anisotropy on  \emph{internal cracks} $A^{(1)}\cap J_u$ is weighted twice as much as the  \emph{free boundary} $\Omega\cap \p^*A$ of the exposed boundary of the free crystal, because cracks can be approximated by ``closing voids'' as in \cite{CF:2020_arma,FFLM:2007, HP:2020_arma}.  The presence of the surface energy over $\Sigma\cap \p^*A\cap J_u$ allows to consider a more general framework for thin films depositing on a substrate, in which cracks  are allowed to appear not only inside the film material, but also along the surface of the substrate characterizing the \emph{delamination region}, where debonding between the atoms of the two materials occurs, and as such,  the corresponding surface tension in \eqref{surergy}  is regarded as the same of the one on the  free-crystal exposed  boundary. Finally, on the complementary  region to the delamination in $\Sigma\cap\p^*A$ where the bulk displacement is continuous,  the relative adhesion coefficient $\beta$ is considered.

We also observe that  in the case of total wetting case, i.e., if $\beta(x) = -\varphi(x,\nu_\Sigma(x))$ for a.e.\ $x\in\Sigma,$ we reduce to the setting of material voids considered in \cite{CF:2020_arma} (with the mismatch strain $\bM_0$ replaced by a Dirichlet boundary condition). On the contrary, in the total dewetting case, i.e., if $\beta(x) = \varphi(x,\nu_\Sigma(x))$ for a.e.\ $x\in\Sigma,$ the energy $\cF$ is minimized by configurational pairs with displacement $u\equiv u_0$ in $\Omega$ and null otherwise, and so characterized by having a zero elastic energy: the model reduces to the dewetted capillarity setting, or in other words, to the anisotropic isoperimetric problem in a container. Finally, in the case with $\fm=|\Omega|$, we reduce to the Griffith model with the inclusion of possible delamination at the substrate boundary, which generalize also for $n=2$ the setting considered by the authors in \cite{HP:2020_arma,HP:2021_arxiv} together with $S\neq\emptyset$.

We now present the two main results of the paper (see Section \ref{subsec:main_results} for more detailed statements) and comment their proofs. We begin by observing that, since the values of the admissible displacement fields $u$ in the void regions $\Omega\setminus A$ do not play any role in the energy of $(A,u)$, as only a formal difference with respect to the previous presentation of the SDRI models introduced in \cite{HP:2020_arma,HP:2021_arxiv}, for every $(A,u)\in\mathcal{C}$ we can redefine $u$ in $\Omega\setminus A$ with a properly chosen constant such that $\Omega\cap \p^*A \subset J_u$ (see Remark \ref{rem:ext_u_out_A}), and so without changing the value of $\cF(A,u)$. We make use of this observation in the following.

\begin{theorem}[\textbf{Existence of minimizing configurations}]\label{teo:intro_existos}
The minimum problem 
\begin{equation}\label{shohdoahoda}
\min\limits_{(A,u)\in\admissible,\,|A|=\fm}\,\,\cF(A,u) 
\end{equation} 
admits a solution.
\end{theorem}

\noindent We refer the Reader  to Theorem \ref{teo:global_existence} for a more detailed and comprehensive statement.

Theorem \ref{teo:intro_existos}  is established by means of the \emph{direct method of the calculus of variations} with respect to a properly chosen  topology $\tau_\admissible$ in $\admissible,$ characterized by the convergence:
$$
(A_k,u_k) \overset{\tau_\admissible}{\longrightarrow} (A,u) \qquad \Longleftrightarrow\qquad 
\begin{cases}
A_k\to A \quad\text{in $L^1(\R^n),$}\\
u_k\to u \quad\text{a.e.\ in $\Omega\cup S.$}
\end{cases}
$$

In order to establish the $\tau_\admissible$-lower semicontinuity of $\cF$ in Theorem \ref{teo:lower_semicontinuity} we consider the positive Radon measures  $\mu_k$ and $\mu$  in $\R^n$ associated to the localized versions of $\cF(A_k,u_k)$ and $\cF(A,u)$, respectively, for which it holds that
\begin{equation}\label{lsc_F_saa}
\liminf\limits_{k\to+\infty} \cF(A_k,u_k) \ge \cF(A,u)\qquad\Longleftrightarrow\qquad \liminf\limits_{k\to+\infty} \mu_k(\R^n) \ge \mu(\R^n).
\end{equation}
 Then, we observe that, up to a subsequence, $\mu_k$ weakly* converges to some positive Radon measure $\mu_0$, and that $\mu$ is absolutely continuous with respect to $\cH^{n-1}\res (\p^*A\cup J_u\cup \Sigma) + \cL^n\res(\Omega\cup S),$  and we establish the following estimates for the Radon-Nikodym derivatives:
\begin{align}
& \frac{\d\mu_0}{\d\cH^{n-1}\res(\p^*A\cup J_u\cup \Sigma)} \ge \frac{\d\mu}{\d\cH^{n-1}\res(\p^*A\cup J_u\cup \Sigma)} \quad  \text{$\cH^{n-1}$ -a.e. on $\p^*A\cup J_u\cup \Sigma,$}\label{surface_part_mm}\\
&\frac{\d\mu_0}{\d\cL^n \res(\Omega\cup S)} \ge \frac{\d\mu}{\d\cL^n \res(\Omega\cup S)} \quad  \text{$\cL^n$-a.e. in $\Omega\cup S,$} \label{volume_part_mm}
\end{align}
which imply that
$
\lim \mu_k(\R^n) = \mu_0(\R^n) \ge \mu(\R^n)  
$
and, in view of \eqref{lsc_F_saa}, conclude the proof of the lower semicontinuity. For the estimate \eqref{surface_part_mm} we need to distinguish between  the estimate at the reduced boundary of $A$ and at $\Sigma\setminus J_u$, where we can implement techniques developed in capillarity theory \cite{ADT:2017,DPhM:2015}, from the  estimate at the (approximate) jump points of $u$, where we employ arguments based on the slicing properties of GSBD-functions as in the Griffith model \cite{ChC:2019_arxiv,ChC:2020_jems,ChC:2020_arxiv}, for  which though extra care is needed: unless $\fm=|\Omega|,$ we cannot directly apply those arguments because at jump points we need to obtain different weights with respect to the ones at the reduced boundary of $A.$ Rather, we replace $J_{u_k}$ in small ``holes'' up to some error by means of Corollaries \ref{cor:functions_with_good_cracks} and \ref{cor:functions_with_cracks_Sigma}  in such a way that each slice   intersects  the boundary of those holes at least in  two points (see the proof of Proposition \ref{prop:estimate_inner_jump}),  which in turns yields the desired estimate with weight $2$ at such  jump points (see Corollary \ref{cor:further_estimates}). Finally, we prove  \eqref{volume_part_mm} by using the convexity of $\cW(A,\cdot)$ and by observing that the condition $u_k\to u$ a.e.\ in $\Omega\cup S,$ together with the compactness result \cite[Theorem 1.1]{ChC:2020_jems}, allows us to conclude  that $\str{u_k} \wk \str{u}$ in $L^2(\Omega\cup S)$.
We recall that in \cite{CF:2020_arma} the authors prove the lower semicontinuity of an energy for crystalline voids via relaxation arguments. Namely, the authors start in the regular family of pair configurations given by voids with a Lipschitz boundary and Sobolev displacement fields, and then in the relaxation, the jump set appears as the void boundaries collapse, resulting in a coefficient $2$ in front of the jump energy of $\cS$. We are here actually arguing in the reverse direction:  first we start in $\mathcal{C}$ with admissible pairs allowing displacements with jump sets, and then we carefully create an at most countable family of voids around them.

The $\tau_\admissible$-compactness of an energy-equibounded sequence $\{(A_k,u_k)\}\subset \admissible$ is established in Theorem \ref{teo:compactness}. We easily get the uniform bounds on the perimeters of $A_k$, the $\cH^{n-1}$- measure of the jumps $J_{u_k}$, and the $L^2$-norm of $\str{u_k}$ by the assumptions on the  anisotropic surface tensions and the elasticity tensor (see Remark \ref{rem:apriori_bounds_seq}). Thus, we can directly deduce the convergence in $L^1(\R^n)$ up to a non-relabelled subsequence of $A_k$ to some set $A\subset\Omega$ of finite perimeter. However,  establishing the $\mathcal{L}^n$ a.e.\ convergence of the displacements $u_k$ is delicate: by \cite[Theorem 1.1]{ChC:2020_jems} there could be an exceptional set $E$ with $\mathcal{L}^n$ positive measure, in which $|u_k|\to+\infty$. The presence of such an exceptional set has been  previously treated by prescribing Dirichlet boundary conditions  \cite{ChC:2019_arxiv,ChC:2020_jems, CF:2020_arma}. For instance, in \cite{CF:2020_arma} the compactness issue is solved by considering in the proof an auxiliary  more general class $GSBD_\infty^p$, $p>1$, of displacements (which are allowed to attain the infinite value on a subset of their domain of also $\mathcal{L}^n$ positive measure) and then, by using the  Dirichlet condition imposed on the displacements at the boundary, the authors are able to prove that the minimizing displacements belong to the original space $GSBD^p.$ However, as in the setting with the mismatch strain \eqref{mismatchstrain}, we cannot rely on any  fixed  boundary condition, one cannot  even exclude the situation with $E=\Omega\cup S$   and hence, this issue unfortunately forbids the implementation of the strategy of \cite{CF:2020_arma} to our SDRI setting. The other option of excluding the presence of the exceptional set is based on the employment of Poincar\'e-Korn inequality for GSBD-functions \cite{CCF:2016} with small jump: the set $\Omega$ is partitioned into a Caccioppoli family $\{P_j\}$ of sets $P_j$ in which a sequence $\{a_k^j\}$ of rigid displacements are defined in such a way that $u_k - a_k^j$ is convergent pointwise a.e.\ in $P_j$, so that one can conclude that  the sequence
\begin{equation}\label{vk_defIntroa}
v_k:=u_k - \sum\limits_j a_k^j\chi_{P_j} 
\end{equation}
converges to some $u\in GSBD^p(\Omega)$ a.e.\ in $\Omega$, $\str{v_k}\wk \str{u}$ in $L^p(\Omega)$, and 
$$
\lim\limits_{k\to+\infty} \cH^{n-1}(J_{u_k}) \ge \cH^{n-1}\Big(J_u\cup(\Omega\cap \bigcup_j \p^*P_j)\Big)
$$
(see \cite[Theorem 1.1]{ChC:2020_arxiv}). However, also this approach seems not implementable in our SDRI setting, since the functions $v_k$ defined in \eqref{vk_defIntroa} may admit extra jumps along the boundary of the partition phases $P_j$ that should be counted with different weights in our setting with different surface tensions. 

In view of these  issues, in order to prove compactness we use a different strategy in this  paper by directly partitioning the sets $A$ and $A_k$ (not only $A$!) into Caccioppoli families (that need to be created by starting from the connected components of the substrate) up to a controllable error (see Figures \ref{fig:construction_Gkdelta} and \ref{fig:set_Gdelta}). Such strategy is a reminiscence of the ideas already used by the authors  in \cite[Theorem 2.7]{HP:2020_arma}, of partitioning $A_k$  by means of introducing extra circles closing the shrinking ``necks'', which though works only for $n=2$ and under the constraint assumed in  \cite{HP:2020_arma} on the number of boundary components for the admissible free-crystal regions. More precisely, we proceed here arguing as follows: First,  by  the classical Poincar\'e-Korn inequality we partition $S$ in a family $\{S^i\}_{i\geq1}$ of sets $S^i$ such that for each $i\geq1$ the set $S^i$ is a union of connected components of $S$ and  there exists a sequence of rigid displacements $\{a_k^i\}$ such that, up to a subsequence, $u_k - a_k^i $ converges a.e.\ in $S^i$ and  $|a_k^i - a_k^j|\to+\infty$ a.e.\ in $\R^n$ for every $j\neq i$. Second, by applying \cite[Theorem 1.1]{ChC:2020_jems} with $u_k-a_k^i$ we construct a family $\{F^i\}_{i\geq0}$ of pairwise disjoint Caccioppoli subsets of $A,$ such that for $i\geq1$ the sequence $u_k-a_k^i$ converges a.e.\ in $F^i\cup S^i$ and diverges to infinity otherwise, and $F^0:=A\setminus\bigcup_{i\ge1} F^i$. Furthermore, since $F^0$ is the portion of the free crystal, so-called in the following ``hanging phase'' (see Figure \ref{fig:parto_intro}), that does not ``interact'' with any substrate component, we can redefine the displacements  in $F^0$ as $u_0$ (see \eqref{mismatchstrain}), which corresponds to providing a zero contribution to the overall elastic energy. Third, by using the $\cH^{n-1}$-rectifiability of $\p^*F^i$ and Propositions \ref{prop:estimate_inner_jump} and \ref{prop:estimate_red_boundary}, we construct for any $\delta>0$ a union $G_k^\delta\subset\Omega$ of open sets covering $\bigcup \p^*F^i$ up to some error of order $O(\sqrt{\delta})$ and whose perimeter and volume are controlled, and we set 
\begin{equation}\label{intro_approx_conf}
B_k^\delta:=A_k\setminus G_k^\delta\quad\text{and}
\quad 
v_k^\delta:=u_0\chi_{F^0} + \sum\limits_{i\ge1}(u_k-a_k^i)\chi_{S^i\cup (F^i\setminus G_k^\delta)} + u_0 \chi_{F^0}.
\end{equation}
We notice that actually the definition of the $v_k^\delta$ in \eqref{intro_approx_conf} is more involved  (see \eqref{def_vvkdelta}), as we need also to control the possible large jumps created along $\Sigma$, that though in the limit disappear (becoming wetting layer), by creating artificial small jumps in $A_k^\delta\setminus A$ and redefining $v_k^\delta$ in that set near $\Sigma$. The obtained configurations satisfy
\begin{equation*}
\cF(A_k,u_k) \ge \cF(B_k^\delta, v_k^\delta) - c\sqrt\delta \Big(\cH^{n-1}(\p^*A_k) + \cH^{n-1}(J_{u_k}) + \sum\limits_{h=0}^mP(F^h)\Big) 
\end{equation*}
for some constant $c>0$ (see Proposition \ref{prop:pass_to_good_seq_compacte}), from which Theorem \ref{teo:compactness} follows by a diagonal argument. 

\begin{figure}[htp!]
\includegraphics[width=0.8\textwidth]{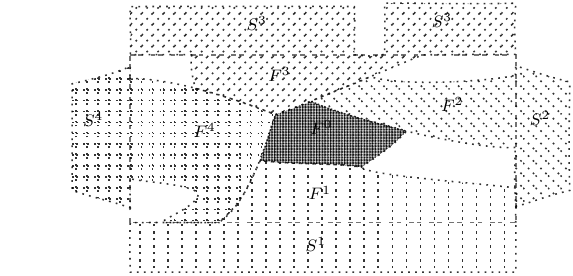}
\caption{\small The partitions of the substrate and the free crystal respectively into the families of Caccioppoli sets $\{S^i\}_{i\geq1}$ and $\{F^i\}_{i\geq0},$  used in the proof of the $\tau_\admissible$-compactness, are illustrated by assigning different line patterns to the crystal phases that interact with the
substrate, and a dotted pattern to the remaining  ``hanging phase''  $F^0.$}
\label{fig:parto_intro}
\end{figure}

We also notice that in the  case with Dirichlet boundary conditions, one see at most 2 elements in the partition, the  hanging phase $F^0$ and a phase $F^1$ interacting with the substrate, since in this case we do not need to add any rigid displacements. Apart from this simplification, the methods used in the proof of Theorems \ref{teo:lower_semicontinuity} and \ref{teo:compactness} still work, even by relaxing the assumptions on the   convex elastic energy densities, i.e., by allowing for a $p$-growth  with respect to strains (see Section \ref{subsec:generalization}). This allows us in particular to recover in Remark \ref{rem:literature_models} the  existence results for the model representing material voids in the  framework with Dirichlet boundary conditions of \cite{CF:2020_arma} and the existence and regularity results  for the Griffith fracture model with Dirichlet boundary conditions of \cite{ChC:2019_arxiv}.

The second main result of the paper relates to properties of partial regularity satisfied by the minimizers $(A,u)$ of $\cF$, such as the essential closedness of $J_u$ and $\p^*A.$ 

\begin{theorem}[\textbf{Regularity results for minimizing configurations}]\label{teo:reguaksaisa}
Let $(\tilde A,\tilde u)$  be solution of \eqref{shohdoahoda}. Then the pair $(A,u)$ defined by 
$$
A:= \Int{A^{(1)}}\qquad\text{and}\qquad u:= \tilde u\chi_{A\cup S} + \xi \chi_{\Omega\setminus A},
$$
where $\xi\in\R^n$ is chosen  such that $\Omega\cap\p^*A\subset J_u$ (see Remark \ref{rem:ext_u_out_A}), is also a solution of \eqref{shohdoahoda}. Furthermore, we have that
$$
\cH^{n-1}(\tilde A^{(1)} \setminus A)<+\infty,\quad 
\cH^{n-1}(J_u\setminus J_u^*)=0,\qquad\text{and}\qquad
\cH^{n-1}(\cl{J_u^*}\setminus J_u^*)=0,
$$
where 
$$
J_u^*:=\{x\in J_u:\,\, \theta(J_u,x)=1\}
$$
with $\theta(J_u,x)$ denoting the $(n-1)$-dimensional density of $J_u$ at $x.$ Finally, there exists a constant $c>0$ such that if $E\subset A$ is a ``hanging'' component of $A,$  i.e., if $\cH^{n-1}([\p^*E\cap\Sigma]\setminus J_u) =0,$ then $|E| \ge c.$ 
\end{theorem}

\noindent We refer the Reader  to Theorem \ref{teo:regularity_of_minimizers} for a more detailed  statement.

The proof of Theorem \ref{teo:reguaksaisa}  is carried out by implementing in the SDRI setting the methods for the partial regularity of the minimizers of the Griffith model by means of the ideas already employed by the authors in \cite{HP:2021_arxiv} for $n=2$: we introduce a localized version of $\cF$ and establish uniform lower and upper $\mathcal{H}^{n-1}$ density estimates for the jump sets (see Section \ref{sec:decay_estimates}) by paying extra care to treat the presence of voids and of the different weights for the surface tension in the surface energy, which is a crucial difference from the Griffith model. We overcome such difficulties by means of the strategy employed in  \cite{P:2012} and based on the \emph{relative isoperimetric inequality} \cite{AFP:2000} to distinguish in the \emph{Decay Lemma} the blows up ``inside the free crystal'' from the ones ``in the voids'', and by applying the approximation result of \cite[Theorem 3]{CCI:2017}.

In this work, we have established existence and partial regularity results for SDRI models in the two-phase setting. Several natural directions remain open for further investigation:
\begin{itemize}[left=5pt]
\item[] {\it Multiphase SDRI models).}
Our analysis considers only two phases, but many applications involve
multiple interacting phases, see e.g. \cite{FPS:2021,LP:2025} and the references therein. Extending existence and regularity results to the multiphase case is challenging even without elasticity \cite[Chapter 30]{Maggi:2012}, due to complex interfaces, junctions, and topological changes, and a general theory for SDRI setting is still lacking. An elastic model for multilayer films presenting both coherent and incoherent interlayer interfaces is introduced and studied in \cite{LP:2025}.
\medskip

\item[] {\it Contact-angle conditions).}
Thin films and crystalline surfaces often meet substrates or other phases at nontrivial contact angles. Incorporating such boundary effects into SDRI models is largely unexplored; see e.g. \cite{DP:2018_1,DP:2018_2} for the 2D graph setting.
\medskip

\item[] {\it Regularity and singularities).}
Even in the static case, detailed regularity beyond partial results is largely unknown. Extending classical results on minimal surfaces to SDRI energies could provide deeper insight into interface geometry and stability.
\medskip

\item[] {\it Quasistatic and dynamic evolution).} Beyond static configurations, one may consider quasistatic gradient-flow evolutions or fully dynamic models including inertia, surface diffusion, or elasticity. While some results exist for quasistatic crack growth \cite{DFT:2005,DT:2002,FL:2003,FS:2018}, including elastic stresses in a general evolution framework remains an open problem.

\end{itemize}
These directions illustrate that while the mathematical theory of SDRI models is developing, there is a rich landscape of questions linking variational analysis, geometric measure theory, and materials science that remains to be explored.

The paper is organized as follows: In Section \ref{sec:main_results} we introduce the SDRI model, some preliminary results related to sets of finite perimeter and GSBD-functions, and state the main results. In Section \ref{sec:prelim_results} we
provide some technical results which allows to replace a part of jump set with an open set without modifying too much the corresponding SDRI energy. Section \ref{sec:lower_semicontos} is devoted to the proof of the lower semicontinuity of $\cF. $ Section \ref{sec:compactoser} contains the proof of the compactness for energy-equibounded sequences. In Section \ref{sec:decay_estimates} we prove the decay estimates for $\cF$ and the regularity results of Theorem \ref{teo:regularity_of_minimizers}. Finally, we conclude the paper with  the  Appendix containing the results related to the equivalence of  the volume-constrained minimum problem with the volume-uncontrained penalized minimum problem, and to some properties of GSBD-functions.

\section{Mathematical setting and formulation of the main results}\label{sec:main_results}

\subsection*{Notation} 

Unless otherwise stated, all sets we consider are subsets of $\R^n,$ in which the coordinates $(x_1,\ldots,x_n)$ of $x\in\R^n$ are given with respect to the standard basis $\{\e_1,\ldots, \e_n\}.$  The symbol $B_r(x)$  stands for the open ball in  $\R^n$  centered at $x$  and of radius $r>0.$  
The symbol $Q_r(x):=x+[-\frac{r}{2},\frac{r}{2}]^n$ stands for the standard $n$-dimensional (hyper) cube in $\R^n$ of sidelength $r$ centered at $x.$ We write $Q_r:=[-\frac{r}{2},\frac{r}{2}]^n.$ Given $r>0,$ $\nu\in\S^{n-1}$ and $x\in\R^n,$ we denote  by $ Q_{r,\nu}(x)$ the cube of sidelength $r$ centered at $x$ whose sides are either parallel or perpendicular to $\nu$.
The characteristic function of a Lebesgue measurable set $F$ is denoted by $\chi_F$ and its Lebesgue measure by $|F|;$ we set also $\omega_n:=|B_1(0)|.$  We denote by $E^c$ the complement of $E$ in $\R^n.$ 
By $\cH^{n-1}$ we denote by $(n-1)$-dimensional Hausdorff measure in $\R^n$ and we write $K=_{\cH^{n-1}} L$ and $K\subset_{\cH^{n-1}} L$ to mean $\cH^{n-1}(K\Delta L)=0$ and $\cH^{n-1}(K\setminus L)=0.$ 

Given an open set $U\subset\R^n,$ the set of $L^1(U)$-functions having bounded total variation in $U$ is denoted by  $BV(U)$ and the elements of 
$$
BV(U;\{0,1\}):=\{E\subseteq U:\,\, \chi_E\in BV(U)\}
$$
are called sets of finite perimeter in $U$. The standard references for $BV$-functions and sets of finite perimeter are for instance \cite{AFP:2000,Gi:1984,Maggi:2012}.

Given $E\in BV(U,\{0,1\}),$ we denote
\begin{itemize}
\item[--] by $P(E,U):=\int_U|D\chi_E|$ the perimeter of $E$ in $U;$

\item[--] by $\p E$ the measure-theoretic boundary of $E,$ i.e., 
$$
\p E:=\{x\in \R^n:\,\, 0<|B_\rho\cap E|<|B_\rho|\quad\forall \rho>0\};
$$
\item[--] by $\p^*E$ the reduced boundary of $E,$ i.e.,  
$$
\p^*E:=\Big\{x\in\R^n:\,\, \exists \nu_E(x):=-\lim\limits_{r\to0}
\frac{D\chi_E(B_r(x))}{|D\chi_E|(B_r(x))}\quad \text{and} \quad |\nu_E(x)|=1\Big\}. 
$$ 
\item[--] by $\nu_E$   the measure-theoretic outer  unit normal to $\p^* E.$
\end{itemize}
Given a Lebesgue measurable set $E\subseteq\R^n$ and $\alpha\in [0,1],$ we define
$$
E^{(\alpha)}:=\left\{x\in\R^n:\,\,\lim\limits_{\rho\to0^+}
\frac{|B_\rho(x)\cap E|}{|B_\rho(x)|} =\alpha\right\}.
$$
Given a set $K\subset\R^n$ and a point $x_0\in \R^n,$ we denote by 
$$
\theta_*(K,x_0):= \liminf\limits_{r\to0} \frac{\cH^{n-1}(B_r(x_0) \cap K)}{ \omega_{n-1}r^{n-1}}
$$
and
$$
\theta^*(K,x_0):= \limsup\limits_{r\to0} \frac{\cH^{n-1}(B_r(x_0) \cap K)}{ \omega_{n-1} r^{n-1}}
$$ 
the $(n-1)$-dimensional lower and upper density of $K$ at $x_0$,  respectively (see e.g., \cite[page 78]{AFP:2000}).  When these  densities coincide, we denote their common value by $\theta(K,x_0)$. Recall that by \cite[Theorem 2.63]{AFP:2000}, $K$ is $\cH^{n-1}$-rectifiable if and only if $\theta(K,x)=1$ for $\cH^{n-1}$-a.e.\ $x\in K.$

Given $x\in\R^n$ and $r>0,$ the blow-up map $\sigma_{x,r}$ is defined as
\begin{equation}\label{blow_ups}
\sigma_{x,r} (y)= \frac{y-x}{r}. 
\end{equation}

Given an open set $U\subset\R^n$ and a metric space $X,$ we denote by $\Lip(U;X)$ the family of all Lipschitz functions $\psi:U\to X.$ We denote by $\Lip(\psi)$ the Lipschitz constant of $\psi\in \Lip(U;X).$

By $GSBD(U;\R^n)$ we denote the collection of all \emph{generalized special functions of bounded deformation} (see \cite{ChC:2020_jems,D:2013_jems} for their definition and properties). Given $u\in GSBD(U;\R^n)$ we denote by $\str{u}\in\mtwo$ the
\emph{approximate symmetric gradient} and by $J_u$ the jump set of $u;$ we recall that  by \cite[Theorem 9.1]{D:2013_jems},
\begin{equation}\label{approx_sym_gradents}
\aplim\limits_{y\to x} \frac{[u(y) - u(x) - \str{u}(x)(y-x)] \cdot (y-x)}{|y-x|^2}=0\qquad\text{for a.e.\  $x\in U$} 
\end{equation} 
and by \cite[Theorem 6.2]{D:2013_jems}, $J_u$ is $\cH^{n-1}$-rectifiable.  Let us also define
$$
GSBD^2(U):=\{u\in GSBD(U;\R^n):\,\,\str{u}\in L^2(U;\mtwo)\}.
$$ 
Given a $\cH^{n-1}$-rectifiable set $K\subset \overline{U},$ we
consider a normal vector $\nu_K$ to its approximate tangent space and we denote by $u_K^+$ and $u_K^-$ the approximate limits of $u\in GSBD(U;\R^n)$ with respect to $\nu_K,$  i.e., 
\begin{equation*} 
u_K^+(x):=\aplim\limits_{\substack{(y-x)\cdot \nu_K>0,\\y\in U}} \,\,u(y)\quad \text{and}\quad u_K^-(x):= \aplim\limits_{\substack{(y-x)\cdot \nu_K<0\\y\in U}}\,\, u(y)  
\end{equation*}
for every $x\in K$ whenever they exist \cite[Definition 2.4]{D:2013_jems}.
We refer to $u_K^+$ and $u_K^-$ as the \emph{one-sided traces} of $u$ at $K$ and we notice that reversing the sign of $\nu_K$ simply  interchanges them.

Let us recall some   notation from \cite{ChC:2020_jems} related to GSBD-functions. For $\xi\in\S^{n-1},$ $y\in\R^n,$  $B\subset\R^n$ and $v:B\to\R^n$ let 
$$
\Pi_\xi:=\{x\in\R^n:\,x\cdot\xi=0\},\qquad B_y^\xi:=\{t\in\R:\,\,y+t\xi\in B\},
$$
and
$$
v_y^\xi(t):=v(y+t\xi),\qquad \hat v_y^\xi(t):=v_y^\xi(t)\cdot \xi.
$$ 
We denote by $\pi_\xi$ the projection of $\R^n$ onto $\Pi_\xi,$ i.e., 
$$
\pi_\xi:=x - (x\cdot\xi) \xi .
$$
Recall that if $v\in GSBD^2(U)$ for an open set $U\subset\R^n,$  then $\hat v_y^\xi \in SBV^2_\loc(U_y^\xi)$ for every $\xi\in\S^{n-1}$ and $\cH^{n-1}$-a.e.\  $y\in\Pi_\xi.$  We denote by $\dot u_y^\xi$ the the absolutely continuous part of $D u_y^\xi$ w.r.t. $\cL^1.$ Let us introduce 
$$
I_{y,\xi}^U (v):=\int_{U_y^\xi} |\dot v_y^\xi|^2\d t 
$$
and
$$
II_{y,\xi}^U(v):=|D[\tau(v\cdot \xi)]_y^\xi|(U_y^\xi),
$$
where $\tau\in  C^1(\R,(-\frac12,\frac12))$ and satisfies $0\le \tau'\le1.$
By \cite[Eq. 3.8]{ChC:2020_jems},
\begin{equation}\label{estimate_I}
\int_{\Pi_\xi} I_{y,\xi}^U(v)\d\cH^{n-1}(y) = \int_U |\str{v}(x)\xi\cdot \xi|^2\d x \le \int_U |\str{v}|^2\d x   
\end{equation}
and by \cite[Eq. 3.9]{ChC:2020_jems} and obvious estimate $a\le 1+a^2$
\begin{align}
\int_{\Pi_\xi} II_{y,\xi}^U(v)\d\cH^{n-1}(y) =  |D_\xi[\tau(v\cdot\xi)]|(U) \le & \int_U |\str{v}|\d x  + \cH^{n-1}(U\cap J_v)\nonumber \\
\le & |U| + \int_U |\str{v}|^2\d x  + \cH^{n-1}(U\cap J_v). \label{estimate_II}
\end{align}
By the Fubini Theorem and the equality 
$$
\int_{\S^{n-1}} |\nu \cdot \xi|\,\d\cH^{n-1}(\xi)=2\omega_{n-1},\qquad \nu\in\S^{n-1},
$$
for any $\cH^{n-1}$-rectifiable Borel set $L\subset\R^n$ and an open set $U\subset\R^n$ we have  
\begin{align}
\cH^{n-1}(U\cap L) = & \frac{1}{2\omega_{n-1}}\,\int_{\S^{n-1}} \d\cH^{n-1}(\xi) \int_{U\cap L} |\nu_L\cdot \xi|\,\d\cH^{n-1}(y)\nonumber \\
=& \frac{1}{2\omega_{n-1}}\,\int_{\S^{n-1}} \d\cH^{n-1}(\xi) \int_{\Pi_\xi} \cH^0(U_y^\xi\cap L_y^\xi)\,\d\cH^{n-1}(y), \label{jump_estimate}
\end{align}
where we applied the area formula with $\pi_\xi$ in the second equality.

A linear function $a:\R^n\to\R^n$ satisfying $\nabla a= - (\nabla a)^T$ is called an (infinitesimal) \emph{rigid displacement}.

\subsection{The SDRI model} 

Given nonempty open sets $\Omega \subset\R^n$ and $\substrate\subset\R^n\setminus\Omega,$ we define  
the space of {\it admissible  configurations} by 
$$
\begin{aligned}
\admissible:=\Big\{(A,u):\,\,& A\in BV(\Omega;\{0,1\}),\,\,u\in GSBD^2(\Ins{\Omega})\cap H_\loc^1(\substrate) \Big\} 
\end{aligned}
$$
where $\Sigma:=\p \substrate\cap \p \Omega.$

The {\it energy} of admissible configurations is given by
\begin{equation*} 
\cF:\admissible\to (-\infty,+\infty], \qquad \cF:=\cS + \cW,
\end{equation*}
where $\cS$ and $\cW$ are the surface and elastic energies of the configuration, respectively. The surface energy of $(A,u)\in\admissible$ is defined as 
\begin{align*}
\cS(A,u):=& \int_{\Omega \cap\p^*A} \varphi(x,\nu_A(x))\,\d \cH^{n-1}(x) \nonumber \\
&+\int_{A^{(1)}\cap J_u} \big[\varphi(x,\nu_{J_u}(x)) + \varphi(x,-\nu_{J_u}(x))\big]\,\d\cH^{n-1}(x)\nonumber\\  
& + \int_{\Sigma\cap \p^*A\setminus J_u} \beta(x) \,\d\cH^{n-1}(x)  + \int_{\Sigma\cap \p^*A\cap J_u} \varphi(x,-\nu_\Sigma(x))\,\d\cH^{n-1}(x), 
\end{align*}
where
$\varphi:\overline\Omega\times\S^{n-1}\to(0,+\infty)$ and $\beta:\Sigma\to\R$  are  Borel functions  denoting the {\it anisotropy} of crystal and the {\it relative adhesion} coefficient of the substrate boundary, respectively, and  $\nu_\Sigma:=\nu_\substrate.$
In applications  it is often convenient to use  the positively one-homogeneous extension $|\xi|\varphi(x,\xi/|\xi|),$ which we also denote by $\varphi.$

The elastic energy of $(A,u)\in\admissible$ is defined as
$$
{\mathcal W}(A,u):= \int_{A\cup  \substrate } W(x,\str{u}  - \bM_0) \d x,
$$
where the elastic energy density $W$ is a quadratic form 
$$
W(x,\bM): =  \C(x)\bM:\bM,
$$
determined by a tensor-valued measurable map $x\in\Omega\cup\substrate\to\C(x),$ the so-called {\it stress-tensor}, in the Hilbert space  $\mtwo$ of all $n\times n$-symmetric matrices with the natural inner product 
$$
\bM:\bN=\sum\limits_{i,j=1}^n M_{ij}N_{ij}.
$$

The {\it mismatch strain}  $x\in\Omega\cup\substrate\mapsto \bM_0(x)\in\mtwo$ is given by  
$$
\bM_0: = 
\begin{cases}
\str{u_0} & \text{in $\Omega ,$}\\
0 &  \text{in $\substrate,$}
\end{cases}
$$ 
for a fixed $u_0\in H^1(\R^n)$. 

\begin{remark}[\textbf{Values of displacements outside a set}]\label{rem:ext_u_out_A}
$\,$
\begin{itemize}[left = 7pt]
\item[(i)] The functional $\cF(A,u)$ does not ``see'' the values of $u$ in $\Omega\setminus A,$ i.e., 
$$
\cF(A,u) = \cF(A,u\chi_{A\cup S}+v\chi_{\Omega\setminus A})\quad \text{for any $v\in GSBD^2(\Omega).$}
$$
Thus, we can redefine $u$ in $\Omega\setminus A$ arbitrarily without changing the energy of the configuration $(A,u)$.

\item[(ii)] For any $(A,u)\in\admissible$ there exists an at most countable set $\Xi^{(A,u)}\subset \R^n$  such that for any $\xi\in\R^n\setminus \Xi^{(A,u)}$ the function 
\begin{equation}\label{replace_uk}
u^\xi:=u\chi_{A\cup S} + \xi\chi_{\Omega\setminus A} 
\end{equation}
satisfies 
\begin{equation}\label{jump_vk}
J_{u^\xi} =_{\cH^{n-1}} (\Omega\cap\p^*A)\cup (\Sigma\cap J_u)\cup (A^{(1)}\cap J_u)\cup (\Sigma\setminus \p^*A).
\end{equation}
Indeed, for $\xi\in\R^n$ let 
$
E_\xi^{(A,u)}:=\{x\in \p^*A\cup \Sigma:\,\, \tr_{A\cup S}^{}u(x) =\xi\} \subset \Sigma\cup \p^*A
$
and let
$$
\Xi^{(A,u)}: = \{\xi\in\R^n:\,\,\cH^{n-1}(E_\xi^{(A,u)})>0\}.
$$
Since $\cH^{n-1}(\p^*A\cup \Sigma)<+\infty$ and $E_\xi^{(A,u)}\cap E_\eta^{(A,u)}=\emptyset$ for $\xi\ne\eta,$  by slicing arguments (see e.g. \cite[Proposition A.2]{HP:2021_arxiv}) the set $\Xi^{(A,u)}$ is at most countable. By the definition of jump, for any $\xi\in\R^n\setminus \Xi^{(A,u)}$ the function $u^\xi$ satisfies \eqref{jump_vk}.

\item[(iii)] For any countable set $\cU\subset\admissible$ there exists an at most countable set $\Xi_{\cU}\subset (0,1)^n$ such that for any $\xi\in(0,1)^n\setminus \Xi_\cU$ and $(A,u)\in \mathcal{U}$ the function $\tilde u^\xi$, defined as in \eqref{replace_uk}, 
satisfies 
\eqref{jump_vk}.  Indeed, it is enough to set
$$
\Xi_\cU:=\bigcup_{(A,u)\in\cU}\Xi^{(A,u)}.
$$ 
\end{itemize}

\end{remark}

\noindent
We introduce a topology in $\admissible$ as follows.

\begin{definition}
We say that a sequence $\{(A_k,u_k)\}$ converges to $(A,u)\in\admissible$ in the $\tau_\admissible$-topology (or shortly $\tau_\admissible$-converges) and denote as $(A_k,u_k)\overset{\tau_\admissible}{\to}(A,u)$  if
\begin{itemize}
 \item $A_k\to A$ in $L^1(\R^n),$ 
 
 \item $u_k\to u$ a.e.\  in $\Omega\cup\substrate.$
\end{itemize}
\end{definition}

\subsection{Main results} \label{subsec:main_results}

Unless otherwise stated, throughout the paper the parameters  $\Omega,$ $\substrate,$ $\varphi,$ $\beta,$ $\C$ of the SDRI energy and volume constant $\fm$  are assumed to satisfy the following hypotheses:

\begin{itemize}
\item[(H0)] $\Omega$ and $\substrate$ are bounded Lipschitz open sets, $\substrate$ has finitely many connected components, $\Sigma:=\p\Omega\cap\p \substrate$ is a Lipschitz $(n-1)$-manifold; 

\item[(H1)] $\varphi\in C^0(\cl{\Omega}\times \R^n)$ and is a 
Finsler norm, i.e., there exist $\bound_2\ge \bound_1>0$ such that 
for every $x\in \cl{\Omega },$ $\varphi(x,\cdot)$ is a norm in $\R^n$ satisfying  
\begin{equation}\label{finsler_norm}
\bound_1|\xi| \le \varphi(x,\xi) \le \bound_2|\xi|,\qquad x\in\cl{\Omega},\quad \xi\in\R^n;
\end{equation}

\item[(H2)] $\beta\in L^\infty(\Sigma)$ and 
satisfies
\begin{equation}\label{hyp:bound_anis}
-\varphi(x,\nu_\Sigma(x))\le \beta(x) \le \varphi(x,\nu_\Sigma(x)) \qquad \text{$\cH^{n-1}$-a.e.\ $x\in \Sigma$};
\end{equation}

\item[(H3)] $\C\in L^\infty(\Omega\cup\substrate)\cap C^{0}(\cl{\Omega})$ and there exists $\bound_4\ge \bound_3>0$ such that 
\begin{equation}\label{hyp:elastic}
2\bound_3\,\bM:\bM \le \C(x)\bM:\bM \le 2\bound_4\,\bM:\bM, \qquad x\in\Omega\cup\substrate,\quad \bM\in\mtwo;
\end{equation}

\item[(H4)] $\fm\in(0,|\Omega|].$
\end{itemize}

\begin{remark}[\textbf{A priori bounds}]\label{rem:apriori_bounds_seq}
Hypotheses (H1)-(H3) are important to get a priori estimates for energy-equibounded countable families. 
Indeed, let $\mathcal{U}\subset \admissible$ be any at most countable family of $\admissible$ such that  
$$
M:=\sup\limits_{(A,u)\in\mathcal{U}}\, \cF(A,u)<+\infty. 
$$
Then by \eqref{finsler_norm} and \eqref{hyp:bound_anis},
\begin{equation*}
\cS(A,u) \le M\quad\text{and}\quad \cW(A,u) \le M + \int_\Sigma |\beta|\d\cH^{n-1} \le M + \bound_2\cH^{n-1}(\Sigma). 
\end{equation*}
Moreover:
\begin{itemize}
\item[(i)]  for any $(A,u)\in\mathcal{U}$
$$
P(A) + \cH^{n-1}(A^{(1)}\cap J_u) \le \frac{M+\bound_2\cH^{n-1}(\Sigma)}{\bound_1} +P(\Omega)   
$$
and
$$
\int_{A\cup\substrate} |\str{u}|^2\d x \le  \frac{2M+2\bound_2\cH^{n-1}(\Sigma)}{\bound_3} + 3\int_{\Omega} |\str{u_0}|^2 \d x ;
$$

\item[(ii)] if $\cU\ni (A_k,u_k)\overset{\tau_\admissible}{\to}(A,u)$ for some $(A,u)\in\admissible,$ then\footnote{Indeed, let $\Xi_\cU\subset\R^n$ be the countable set,   given by Remark \ref{rem:ext_u_out_A} (c), and let $\xi\in (0,1)^n\setminus \Xi_\cU.$ Since the values of $u_k$ are not important in $\Omega\setminus A_k,$ we may assume  $u = u^\xi,$ where $  u^\xi$ is given as \eqref{replace_uk}. Then $u_k\to u\chi_{A\cup S}+\xi\chi_{\Omega\setminus A}$ a.e.\   and hence, \eqref{strain_converges} follows from \cite[Theorem 1.1]{ChC:2020_jems}.} 
\begin{equation}\label{strain_converges}
\chi_{A_k\cup S}\str{u_k} \wk \chi_{A\cup S}\str{u}\quad\text{in $L^2(\Ins{\Omega})$.}
\end{equation} 

\end{itemize}
 
\end{remark}

Now we formulate main results of the paper. First we deal with the existence of admissible configurations with minimal energy.

\begin{theorem}[\textbf{Existence  of minimizing configurations}]\label{teo:global_existence}
The minimum problem 
\begin{equation}\label{min_prob_globals}
\inf\limits_{(A,u)\in \admissible,\,\,|A| = \fm} \cF(A,u)  
\end{equation}
has a solution. Moreover, there exists $\lambda_0>0$ such that $(A,u)\in\admissible$ is a solution of \eqref{min_prob_globals} if and only if it solves  
\begin{equation}\label{min_prob_globals_uncons}
\inf\limits_{(A,u)\in \admissible} \cF^\lambda(A,u)
\end{equation}
for any $\lambda\ge \lambda_0,$ where  
$$
\cF^\lambda(A,u):=\cF(A,u) +\lambda\big||A| - \fm\big|.
$$
\end{theorem}

To prove Theorem \ref{teo:global_existence} we will apply direct methods of Calculus of Variations. To this aim, we establish the $\tau_\admissible$-lower semicontinuity of $\cF$ and the $\tau$-compactness of energy-equibounded sequences in $\admissible.$

\begin{theorem}[\textbf{Lower semicontinuity}]\label{teo:lower_semicontinuity}
Assume that the sequence $\{(A_k,u_k)\}\subset\admissible$ $\tau_\admissible$-converges to $(A,u)\in\admissible.$ Then 
\begin{equation}\label{eq:lsc_functional}
\liminf\limits_{k\to+\infty}\,\cF(A_k,u_k) \ge \cF(A,u). 
\end{equation}
\end{theorem}

\begin{theorem}[\textbf{Compactness}]\label{teo:compactness}
Let $\{(A_k,u_k)\}\subset  \admissible$ be such that
$$
M:=\sup_k\,\cF(A_k,u_k)<+\infty.
$$
Then there exists a subsequence $\{(A_{k_l},u_{k_l})\},$ a sequence $\{(B_l,v_l)\}\subset \admissible$ and $(A,u)\in\admissible$ such that $(B_l,v_l)\overset{\tau_\admissible}{\to} (A,u),$ $|A_{k_l}\Delta B_l|\to0$ and 
\begin{equation*}
\liminf\limits_{l\to+\infty}\, \cF(A_{k_l},u_{k_l}) \ge \liminf\limits_{l\to+\infty}\, \cF(B_l,v_l) \ge \cF(A,u).  
\end{equation*}
\end{theorem}

\noindent
Notice that our compactness result is analogous to those in \cite{FFLM:2007,HP:2020_arma}. According to the proof, in general we have $|B_l|\le |A_{k_l}|,$ i.e., the volume constraint may not be preserved. Rather, Theorems \ref{teo:lower_semicontinuity} and \ref{teo:compactness} allow to solve the unconstrained minimum problem \eqref{min_prob_globals_uncons}, and then, as in \cite[Theorem 1]{EF:2011}, using the equivalence of the minimum problems \eqref{min_prob_globals} and \eqref{min_prob_globals_uncons} (see Proposition \ref{prop:fusco}), we establish the existence of a volume-constraint minimizer.

It is worth to remark that in both Theorems \ref{teo:lower_semicontinuity} and \ref{teo:compactness} (and hence, in the existence) the assumption $\C\in C(\cl{\Omega})$ can be relaxed to $\C\in L^\infty(\Omega)$. The continuity of $\C$ is important in the (partial) regularity of minimizers of $\cF$.

\begin{theorem}[\textbf{Properties of minimizing configurations)}]\label{teo:regularity_of_minimizers}  
Let $(\tilde A,\tilde u)\in \admissible$ be a solution of \eqref{min_prob_globals},
$$
A=\Int{\tilde A^{(1)}}\quad\text{and}\quad u = \tilde u\chi_{A\cup S} + \xi \chi_{\Omega\setminus A},
$$
where $\xi\in(0,1)^n$ is chosen such that $\Omega\cap\p^* A\subset_{\cH^{n-1}} J_u$ (see Remark \ref{rem:ext_u_out_A}), and let
$$
J_u^*=\{x\in J_u:\,\, \theta(J_u,x)=1\}.
$$
Then: 

\begin{itemize}[left=12pt]
\item[\rm(i)]  $(A,u)$ is a minimizer of $\cF$ and 
$$
\cH^{n-1}(\tilde A^{(1)} \setminus A)<+\infty,\quad 
\cH^{n-1}(J_u\setminus J_u^*)=0,\quad
\cH^{n-1}(\cl{J_u^*}\setminus J_u^*)=0; 
$$

\item[\rm(ii)]  for any $x\in\Omega$ and $r\in(0,\min\{1,\dist(x,\p \Omega)\})$
$$
\frac{\cH^{n-1}(Q_r(x)\cap J_u)}{r^{n-1}} \le \frac{4n\bound_2 + \lambda_0}{\bound_1},
$$
where $\lambda_0$ is given by Theorem \ref{teo:global_existence}; 

\item[\rm(iii)]  there exist $\varsigma_0=\varsigma_0(\bound_1,\bound_2,\bound_3,\bound_4)\in(0,1)$ and $R_0=R_0(\bound_1,\bound_2,\bound_3,\bound_4)>0$  such that 
$$
\frac{\cH^{n-1}(Q_r(x)\cap J_u)}{r^{n-1}} \ge \varsigma_0
$$
for all cubes $Q_r(x)\subset\Omega$ centered at $x\in \Omega\cap \cl{J_u^*}$ with sidelength $r\in(0,R_0);$ 

\item[\rm(iv)]   if $E\subset A$ is any connected component of $A$ with $\cH^{n-1}([\p^* E\cap\Sigma] \setminus J_u)=0,$ then $|E|\ge\omega_n \big(\frac{\bound_1n}{\lambda_0}\big)^n$  and $u=u_0+a$ in $E$   for some rigid displacement $a.$
\end{itemize}
\end{theorem}
 
\subsection{Generalization and extra results related to literature models}\label{subsec:generalization}

In this section we discuss some SDRI-type models to which the arguments used for the main results extend with only minor modifications. In doing so, we also recover a number of results previously established in the literature.

First we consider more general elastic energy densities.

\begin{theorem}[\textbf{Elastic density with $p$-growth}]\label{teo:elastic_plow}
For $p>1$ let an elastic energy density $W_p:\Ins{\Omega}\times \mtwo\to\R$ be such that
\begin{itemize}[left=15pt]
 \item[\rm(a1)] for any $x\in \Ins{\Omega},$ the map $\bM\mapsto W_p(x,\bM)$ is convex and lower semicontinuous in $\mtwo$,

 \item[\rm(a2)] for any $\bM\in\mtwo,$ the map $x\mapsto W(x,\bM)$ is measurable,

 \item[\rm(a3)] there exist $c>0$ and $f\in L^1(\Ins{\Omega})$ such that
 \begin{equation}\label{lower_boundas_papap}
 W_p(x,\bM) \ge c|\bM|^p + f(x)\quad \text{for a.e.\ $x\in \Ins{\Omega}$ and for all $\bM\in\mtwo.$}
 \end{equation}
\end{itemize}
Let 
$$
\admissible_p:=\{(A,u):\,\, A\in BV(\Omega;\{0,1\}),\,\, u\in GSBD^p(\Ins{\Omega})\}
$$
be a class of admissible configurations and let
$$
\cF_p = \cS + \cW_p\quad \text{in $\admissible_p$},
$$
where 
$$
\cW_p(A,u) = \int_{A\cup S} W_p(x,\str{u} - \bM_0)\,\d x.
$$
Then for any $\fm\in (0,|\Omega|]$ the minimum problem 
\begin{equation}\label{min_prob_plow}
\min\limits_{(A,u)\in\admissible_p,\,\,|A| = \fm} \,\cF_p(A,u) 
\end{equation}
admits a solution. Moreover, there exists $\lambda_0>0$ such that for any $\lambda>\lambda_0$ a configuration $(A,u)$ is a solution to \eqref{min_prob_plow} if and only if it is a minimizer of 
$$
\cF_p^\lambda(A,u) = \cF(A,u) + \lambda\big||A| - \fm\big|.
$$
\end{theorem}

\noindent
A standard example of $W_p$ is 
$$
W_p(x,\bM) = f(x) |\bM|^p + g(x)
$$
for some $f\in L^\infty(\Ins{\Omega})$ with $f\ge c>0$ a.e.\  and $g\in L^1(\Ins{\Omega}).$ 

Now we study the existence of minimizers in models related to the SDRI setting, but with Dirichlet boundary conditions.

\begin{theorem}[\textbf{Dirichlet case with a $p$-growth elastic density}]\label{teo:dirichlet_plow}
For $p>1$ let
$$
\admissible_{\rm Dir}:=\{(A,u):\,\,A\in BV(\Omega;\{0,1\}),\,\, u\in GSBD^p(\Ins{\Omega}),\,\,u=u_0\,\,\text{in $S$}\},
$$
where $u_0\in H^1(\R^n)$ is fixed, and let 
$$
\cF_{\rm Dir}:=\cS + \cW_{\rm Dir} \quad\text{in $\admissible_{\rm Dir}$},
$$
where
$$
\cW_{\rm Dir}(A,u):=\int_{A} W_p(x,\str{u})\,\d x 
$$
and the elastic energy density $W_p$ satisfies all assumptions of Theorem \ref{teo:elastic_plow}. Then for any $\fm\in (0,|\Omega|]$ the minimum problem 
\begin{equation}\label{min_prob_dir_plow}
\min\limits_{(A,u)\in\admissible_{\rm Dir},\,\,|A| = \fm} \,\cF_{\rm Dir}(A,u)
\end{equation}
admits a solution. Moreover, there exists $\lambda_0>0$ such that for any $\lambda>\lambda_0$ a configuration $(A,u)$ is a solution to \eqref{min_prob_dir_plow} and if only if it is a minimizer of
$$
\cF_{\rm Dir}^\lambda(A,u) = \cF_{\rm Dir}(A,u) + \lambda\big||A| - \fm\big|.
$$
\end{theorem}

\begin{remark}[\textbf{Relation to some literature results}]\label{rem:literature_models}
As a consequence of Theorem  \ref{teo:dirichlet_plow}, we have:
\begin{itemize}[left=12pt]
\item[(i)] Let $\beta(x) = -\varphi(x,\nu_\Sigma(x))$ for $\cH^{n-1}$-a.e.\ $x\in \Sigma$ and let $W_p:\mtwo\to\R$ satisfy
$$
c'|\bM|^p - c'' \le W_p(\bM) \le c'' (|\bM|^p + 1) 
$$
for some $c'',c'>0.$ Then Theorem \ref{teo:dirichlet_plow} coincides with the existence result \cite[Proposition 5.8]{CF:2020_arma} in the setting of material voids.

\item[(ii)] Let $\beta=0$ and $W_p$ be as in (i). Then the minimizers of $\cF_{\rm Dir}$ in $\admissible_{\rm Dir}$ with volume constraint $|\fm| = |\Omega|$ (i.e., free-crystal regions have full $\mathcal{L}^n$-measure) coincide with the (strong) Griffith minimizers in \cite{ChC:2019_arxiv} under Dirichlet boundary condition.

\item[(iii)] In the proof of Theorem \ref{teo:regularity_of_minimizers} we work only in $\Omega,$ i.e., we study the regularity of $\p^*A$ and $J_u$ only in the points of $\Omega.$ Therefore,  the essential closedness of $J_u$ and $\p^*A$ holds also for minimizers of $\cF_{\rm Dir}$ with
$W_p(x, \bM)=\C(x)\bM:\bM.$ In particular, this covers a partial regularity result in \cite{ChC:2019_arxiv}.
\end{itemize}
\end{remark}

We anticipate here that we equip both $\admissible_p$ and $\admissible_{\rm Dir}$ with the same type of convergence introduced in $\admissible,$ i.e. 
\begin{equation}
(A_k,u_k) \overset{\tau}\to (A,u)\quad\Longleftrightarrow \quad A_k\overset{L^1(\R^n)}{\longrightarrow} A\,\, \text{and}\,\, u_k\to u\,\,\text{a.e.\  in $\Omega\cup S.$}
\label{tau_convergence_CACA}
\end{equation}

\section{Replacing cracks with voids}\label{sec:prelim_results}

In this section we provide some technical results that allow to replace a portion of the jump set of the displacement fields with an open set without modifying too much the corresponding SDRI energy. These results will be used in the proofs of both lower semicontinuity and compactness results. We start with the following main ingredient of all crack-opening results.

\begin{lemma}\label{lem:creating_hole}
Let $\delta\in(0,1/4),$ $Q:=Q_{r,\nu}(x_0)$ be a cube, $\Gamma\subset Q$ is an $(n-1)$-dimensional Lipschitz graph and $K\subset Q$ be an $\cH^{n-1}$-rectifiable set. 
Assume that 
\begin{itemize}[left=15pt]
 \item[\rm (a1)]  $x_0\in\Gamma,$  $\nu$ is the unit normal to $\Gamma$ at $x_0$ and   $|(x-x_0)\cdot\nu|\le r/2$ for all $x\in\Gamma$; 
 
 \item[\rm (a2)] $\Gamma$ separates $Q$ into two open connected components $G_1$ and $G_2$;
 
 \item[\rm (a3)] $\theta(K,x_0)=\theta(K\cap\Gamma,x_0)=1,$ $\nu$ is the generalized unit normal to $K$ at $x_0,$ and 
 $$
(1-\delta)r^{n-1}\le  \cH^{n-1}( K\cap\Gamma) \le \cH^{n-1}(\Gamma)\le (1+\delta)r^{n-1};
 $$
 
 \item[\rm (a4)] $\cH^{n-1}( K \setminus \Gamma)<\delta r^{n-1}$.
\end{itemize}
Then there exist open sets $C,D\strictlyincluded Q$ of finite perimeter such that 
\begin{itemize}[left=20pt]

\item[\rm (i)] $C\subset G_1,$  and $\cH^{n-1}(\p C\setminus \p^*C)=\cH^{n-1}(\p D\setminus \p^*D)=0;$
 
 \item[\rm (ii)] $\cH^{n-1}(K\setminus \cl{C}) <2\delta r^{n-1}$ and $\cH^{n-1}(K\setminus D)<2\delta r^{n-1};$ 
 
 \item[\rm (iii)] $|C|< \delta r^n$ and $|D|<\delta r^n;$ 
 
 \item[\rm (iv)] 
 $
 (1-2\delta)r^{n-1}\le \cH^{n-1}(K\cap \p C\cap \Gamma) \le \cH^{n-1}(\p C\cap \Gamma) < (1+\delta)r^{n-1}; 
 $

 \item[\rm (v)] for any norm $\phi$ in $\R^n$
 satisfying
 \begin{equation}\label{fjskiii098}
\bound_1\le \phi(\nu) \le \bound_2,\quad \nu\in\S^{n-1},
\end{equation}
 one has
 \begin{equation}\label{two_dim_esimates}
 \int_{\p D} \phi(\nu_D)\d\cH^{n-1} \le 2 \int_K  \phi(\nu_K)\d\cH^{n-1} + 5\bound_2\delta r^{n-1}.  
 \end{equation} 
 \begin{equation}\label{one_dim_esimates_C}
 \int_{\p C} \phi(\nu_C)\d\cH^{n-1} \le 2 \int_K \phi(\nu_K)\d\cH^{n-1} + 5\bound_2\delta r^{n-1} 
 \end{equation}
 and
 \begin{equation}\label{one_dim_esimates}
 \int_{G_1\cap \p C} \phi(\nu_C)\d\cH^{n-1} \le \int_K \phi(\nu_K)\d\cH^{n-1} + 3\bound_2\delta r^{n-1},
 \end{equation} 
\end{itemize}
\end{lemma}

\begin{proof}
Without loss of generality we assume that $\nu=\e_n,$ $x_0=0$ and  $G_1$ lies above $\Gamma$. Since $\Gamma$ is a Lipschitz graph, $f\in \Lip(V)$ such that $\Gamma={\rm graph}(f),$ where $V=[-\frac r2,\frac r2]^{n-1}\subset\R^{n-1}.$ By (a1), $\|f\|_\infty\le r/2$ and hence $\Gamma$ intersects only the  lateral sides of $Q.$
Let 
$$
\epsilon:=\frac{\delta}{4(1+\Lip(f))}.
$$
Let $V''\strictlyincluded V'\strictlyincluded V$ be any $(n-1)$-dimensional cubes in $\R^{n-1}$ such that 
\begin{equation}\label{good_cubes_for08172}
\cH^{n-1}(V\setminus V'')<\epsilon r^{n-1}. 
\end{equation}
For $\gamma\in(0,\epsilon r)$ let $g\in \Lip_c(V;[0,\gamma])$ be such that $g\equiv\gamma$ in $V'',$ $\supp(g)=\cl{V'}$ and $\|g\|_\infty\le 1.$ Let $C$ be the open set enclosed by the graphs of $f$ and $f+g$ and let $D$ be the open set enclosed by the graphs of $f+g$ and $f-g.$ Since both $\p C$ and $\p D$ consist of the union of two Lipschitz graphs, they have finite perimeter.

We claim that $C$ and $D$ satisfy the assertions of the lemma.

(i) Since $\|f\pm g\|_\infty<3r/4$ (by (a1) and choice of $\gamma$) and $g=0$ on $V\setminus V',$ we have $C\subset G_1$ and $C,D\strictlyincluded Q_r.$ Since $V'$ is an $(n-1)$-dimensional cube, by the area formula,
$$
\cH^{n-1}(\p C\setminus \p^*C) =\cH^{n-1}(\p D\setminus \p^*D) \le (1+\Lip(f)) \cH^{n-1}(\cl{V'}\setminus V')=0.
$$

(ii) By  (a4),
\begin{align*}
\cH^{n-1}(K \setminus \cl{C}) \le & \cH^{n-1}(\Gamma\cap K \setminus \cl{C}) + \cH^{n-1}(K \setminus \Gamma)
< \cH^{n-1}(\Gamma\setminus \cl{C}) +\delta r^{n-1}.
\end{align*}
Moreover, by contruction,
$$
\Gamma\setminus \cl{C}=\Gamma\setminus\p C=\Gamma\setminus\cl{D}=f(V\setminus \cl{V'}),
$$ 
and hence by the area formula and \eqref{good_cubes_for08172},
\begin{equation}\label{bad_part_ingamma}
\cH^{n-1}(\Gamma\setminus \cl{C}) \le \int_{V\setminus V'} \sqrt{1+|\nabla f|^2}\d x' \le (1 + \Lip(f)) \cH^{n-1}(V\setminus V') < \frac{\delta}{4}\, r^{n-1}. 
\end{equation}
Thus, $\cH^{n-1}(K \setminus C)<\frac54\,\delta r^{n-1}.$
Similarly, $\cH^{n-1}(\Gamma\setminus \cl{D}) =\cH^{n-1}(\Gamma\setminus D)<\frac{1}{4}\delta r^{n-1}.$

(iii) By the Fubini's theorem, the choice of $\gamma$ and also the area formula,
\begin{align*}
|C| = & \int_{V'} (f+g -f)\d x\le \gamma \cH^{n-1}(V') <
\epsilon r \int_V \sqrt{1+|\nabla f|^2}\d x' =\epsilon r \cH^{n-1}(\Gamma) 
\end{align*}
and 
\begin{align*}
|D| = & \int_{V'} (f+g -(f-g))\d x\le 2\gamma \cH^{n-1}(V') <
2\epsilon r \int_V \sqrt{1+|\nabla f|^2}\d x' =2\epsilon r \cH^{n-1}(\Gamma) 
\end{align*}
Hence, by (a3), $|C| < \frac{\delta(1+\delta)}{4}r^n$ and $|C| < \frac{\delta(1+\delta)}{2}r^n.$

(iv) By (a3),
$$
\cH^{n-1}(\p C\cap \Gamma) < \cH^{n-1}(\Gamma) \le (1+\delta)r^{n-1}.
$$
Moreover, by \eqref{bad_part_ingamma},
$$
\cH^{n-1}(K\cap \Gamma) - \cH^{n-1}(K\cap \p C\cap \Gamma) =
\cH^{n-1}(K\cap\Gamma\setminus \p C) \le \cH^{n-1}(\Gamma\setminus \p C) =\cH^{n-1}(\Gamma\setminus C) < \frac{\delta}{4}r^{n-1}.
$$
Hence, again by (a3),
$$
\cH^{n-1}(K\cap \p C\cap \Gamma)  \ge \cH^{n-1}(K\cap\Gamma)-\frac{\delta}{4}r^{n-1}>(1-\tfrac54\delta)r^{n-1}.
$$

(v)  By the definition of $C,$ the area formula, the convexity of $\phi$, the definition of $g$, \eqref{fjskiii098} and \eqref{good_cubes_for08172},
\begin{align*}
\int_{G_1\cap \p C}\phi(\nu_C)\d\cH^{n-1} = & \int_{G_1\cap {\rm graph}(f+g)} \phi(\nu_C)\d\cH^{n-1}=\int_{V'} \phi(-\nabla(f+g),1)\d\cH^{n-1}\\
\le & \int_{V'} \phi(-\nabla f,1)\d\cH^{n-1} + \int_{V'} \phi(-\nabla g,0)\d\cH^{n-1}\\
\le& \int_{V} \phi(-\nabla f,1)\d\cH^{n-1} + \int_{V'\setminus V''} \phi(-\nabla g,0)\d\cH^{n-1}\\
\le & \int_\Gamma \phi(\nu_\Gamma)\d\cH^{n-1} + \bound_2\|\nabla g\|_\infty \cH^{n-1}(V'\setminus V'')\\
\le & \int_\Gamma \phi(\nu_\Gamma)\d\cH^{n-1} + \frac{\bound_2 \delta}{4}\,r^{n-1}.
\end{align*}
Moreover, by (a3),
\begin{equation*} 
\cH^{n-1}(\Gamma\setminus K) = \cH^{n-1}(\Gamma) - \cH^{n-1}(\Gamma\cap K) \le 2\delta r^{n-1}, 
\end{equation*}
and hence by \eqref{fjskiii098},
\begin{equation}\label{gamma_les_K}
\int_\Gamma \phi(\nu_\Gamma)\d\cH^{n-1} \le 
\int_{K\cap\Gamma} \phi(\nu_K)\d\cH^{n-1} + 
\bound_2\cH^{n-1}(\Gamma\setminus K) \le 
\int_K \phi(\nu_K)\d\cH^{n-1} + \frac{9\bound_2}{4}\,\delta r^{n-1}. 
\end{equation}
Thus, \eqref{one_dim_esimates} follows. Since $\p C\cap \p G_1=\Gamma,$ the proof of \eqref{one_dim_esimates_C} follows from \eqref{gamma_les_K} and \eqref{one_dim_esimates}. Similarly,
\begin{align*}
\int_{\p D}\phi(\nu_D)\d\cH^{n-1} = & \int_{V'}\Big[\phi(-\nabla (f+g),1)\d\cH^{n-1} + \phi(-\nabla(f-g),1)\Big]\d\cH^{n-1}\\
\le & 2 \int_{V'} \phi(-\nabla f,1)\d\cH^{n-1} 
+ 2\int_{V'}\phi(-\nabla g,0)\d\cH^{n-1}\\
\le& 2\int_\Gamma\phi(\nu_\Gamma)\d\cH^{n-1} +\frac{\bound_2}{2}\,\delta r^{n-1}\\
\le & 2\int_K\phi(\nu_K)\d\cH^{n-1} +\frac{9\bound_2}{2}\,\delta r^{n-1}.
\end{align*} 
\end{proof}

The following result will be used in the proof of Proposition \ref{prop:estimate_inner_jump} with $K=A_k^{(1)}\cap J_{u_k}$ and allows to replace $u_k$ with $v_k,$ whose jump set is a reduced boundary of an open set of finite perimeter (see Corollary \ref{cor:functions_with_good_cracks} below). Recall that this property is important to obtain the surface tension $2\varphi$ in the ``interior'' jump energy in the functional $\cS.$

\begin{lemma}\label{lem:opening_cracks}
Let $U\subset \R^n$ be an open set, $K\subset U$ be a $\cH^{n-1}$-rectifiable set and $\delta>0.$ There exists an at  most countable family $\{C_i\}_{i\ge1}$ of open sets of finite perimeter such that 
\begin{itemize}[left=15pt]
\item[\rm(i)] $C_i\strictlyincluded U$ and $\cH^{n-1}(\p C_i\setminus \p^* C_i)=0$; 

\item[\rm(ii)]  $\cH^{n-1}(K \setminus \bigcup_i C_i) < \delta $ and $|\bigcup_i C_i|< \delta;$ 

\item[\rm(iii)] for any norm $\phi$ in $\R^n$ satisfying \eqref{fjskiii098}
$$
\sum\limits_{i\ge1} \int_{\p C_i} \phi(\nu_{C_i})\,\d\cH^{n-1} < 2\int_K \phi(\nu_{K}^{})\,\d\cH^{n-1} + \delta.
$$
\end{itemize}
\end{lemma}

\begin{proof}
First we consider a special case.

{\it Claim.} Let $K={\rm graph}(f)$ for some $f\in \Lip(V),$ where $V\subset\R^{n-1}$ is a bounded open set. Let $V''\strictlyincluded V'\strictlyincluded V$ be smooth open sets such that 
\begin{equation}
\Big(1 + \frac{1}{\bound_1}\Big)\,\int_{V\setminus V''} \phi(-\nabla f,1)\d x' + \cH^{n-1}(V\setminus V'') < \frac{\delta }{2 + 2\bound_2}. 
\label{ohohohoa} 
\end{equation}
For $\gamma\in\big(0,\frac{\delta }{4[1 + \cH^{n-1}(V')]}\big)$ let $g\in \Lip(V;[0,\gamma])$ be such that $\supp(g)=\cl{V'},$  $g\equiv \gamma$ in $V''$ and $\|\nabla g\|_{L^\infty(V)}\le 1.$ Then $g=0$ on $\p V'.$ Moreover, taking $\gamma$ small enough we assume that the graphs of $f\pm g\big|_{V'}$ are compactly contained in $U.$ Let $C$ be the bounded  open set whose boundary consists of the graphs of $f-g:V'\to\R$ and $f+g:V'\to\R.$ Then $C\strictlyincluded U,$
and by the area formula, the triangle inequality for $\phi,$ \eqref{fjskiii098}, \eqref{ohohohoa} and the inequality $\|\nabla g\|_\infty\le1$
$$
\begin{aligned}
\int_{\p C} \phi(\nu_C)\d\cH^{n-1}= &\int_{V'} \Big(\phi(-\nabla(f+g),1) + \phi(-\nabla (f-g),1) \Big)\d x'\\
\le & 2\int_{V'} \phi(-\nabla f,1)\d x' + 2\int_{V'} \phi(-\nabla g,0)\d x'\\
\le & 2\int_{V} \phi(-\nabla f,1)\d x'+  2\int_{V'\setminus V''} \phi(\nabla g,0)\d x'\\
\le & 2\int_K \phi(\nu_K)\d\cH^{n-1} + \delta .
\end{aligned} 
$$
Moreover, by \eqref{fjskiii098} and \eqref{ohohohoa},
$$
\cH^{n-1}(K\setminus C) = \int_{V\setminus V'} \sqrt{1 + |\nabla f|^2}\d x' \le \frac{1}{\bound_1}\,\int_{V\setminus V''} \phi(-\nabla f,1)\d x'< \delta . 
$$
Finally, since $0\le g\le \frac{\delta }{4[1+\cH^{n-1}(V')]},$
$$
|C| =\int_{V'} [f+g - (f-g)]\d x' \le 2\|g\|_\infty\cH^{n-1}(V') <\delta .
$$
The equality 
$\cH^{n-1}(\p C\setminus \p^*C)=0$ follows from the smoothness of $V'$.

Now we prove the lemma. By the countable $\cH^{n-1}$-rectifiability of $K,$ there exists an at most countable family $\{\Gamma_i\}$ of Lipschitz graphs such that $\Gamma_i\subset U,$ $\Gamma_i\cap \Gamma_j=\emptyset$
for $i\ne j,$ and $\cH^{n-1}(K \setminus \bigcup_i \Gamma_i)=0.$ Since $\cH^{n-1}\res \Gamma_i$ is Radon, by the regularity of Radon measures, for each $i$ there exists a relatively open subset $\Gamma_i'$ of $\Gamma_i$ such that $\Gamma_i'\cap K \subset \Gamma_i \cap K $ and
\begin{equation}\label{almost_good_coverchuk}
\cH^{n-1}(\Gamma_i' \setminus K )<\frac{\delta}{2^{i+2}(1+\bound_2)},\quad i\ge1. 
\end{equation}
For shortness, we assume $\Gamma_i=\Gamma_i'$.
Then applying the claim above with $\delta:=\frac{\delta}{2^{i+1}(1+\bound_2)}$ and $\Gamma=\Gamma_i$ we find an open set $C_i\strictlyincluded U $ such that 
\begin{equation}\label{saidhaidab1}
|C_i|< \frac{\delta}{2^{i+1}(1+\bound_2)},\qquad 
\cH^{n-1}(\Gamma_i\setminus C_i) < \frac{\delta}{2^{i+1} (1+\bound_2)} 
\end{equation}
and 
\begin{equation}\label{c_ning_defeee}
\int_{\p C_i} \phi(\nu_{C_i}^{})\d\cH^{n-1} \le 2\int_{\Gamma_i} \phi(\nu_{\Gamma_i})\d\cH^{n-1} + \frac{\delta}{2^{i+1}  (1+\bound_2)}. 
\end{equation}
Thus, by the pairwise disjointness of $\{\Gamma_i\},$
\begin{align*}
\cH^{n-1}\Big(K\setminus \bigcup_j C_j\Big) \le & \cH^{n-1}\Big(\bigcup_i \Big(\Gamma_i\setminus \bigcup_j C_j\Big)\Big) = \sum_i \cH^{n-1}\Big(\Gamma_i\setminus \bigcup_j C_j\Big)\\
\le & \sum_i \cH^{n-1}(\Gamma_i\setminus C_i) < \delta 
\end{align*}
and by \eqref{almost_good_coverchuk} and \eqref{c_ning_defeee},
$$
\begin{aligned}
\sum\limits_i \int_{\p C_i} \phi(\nu_{C_i}^{})\d\cH^{n-1} \le & 2\sum\limits_i \int_{\Gamma_i} \phi(\nu_{\Gamma_i})\d\cH^{n-1} + \frac{\delta}{2}\\
\le & 2\sum\limits_i \int_{\Gamma_i\cap  K} \phi(\nu_{K})\d\cH^{n-1} +2 \sum\limits_i \int_{\Gamma_i\setminus K} \phi(\nu_{\Gamma_i})\d\cH^{n-1}+ \frac{\delta}{2}\\
\le & 2\int_{\cup_i \Gamma_i \cap K} \phi(\nu_{K})\d\cH^{n-1} + 
\sum\limits_i \frac{2\bound_2\delta}{2^{i+2}(1+\bound_2)} + \frac{\delta}{2}\\
= & 2\int_K \phi(\nu_{K})\d\cH^{n-1} + \delta.
\end{aligned}
$$
Finally, by the estimate for $|C_i|$ in \eqref{saidhaidab1}
$$
\Big|\bigcup_i C_i\Big| \le \sum\limits_{i} |C_i| <\delta.
$$
\end{proof}

\begin{corollary}\label{cor:functions_with_good_cracks}
Let $U\strictlyincluded \Omega$ be an open set, $(A,u)\in\admissible$ and $\delta>0.$ Then there exists an open set $G\strictlyincluded U$ of finite perimeter such that  
\begin{itemize}[left=15pt]
 \item[\rm(i)] the configuration $(B,v)$ with $B:=A\setminus G$ and $v:=u\chi_{B\cup S}$ belongs to $\admissible;$ 
  
 \item[\rm(ii)] $|G|<\delta;$

 \item[\rm(iii)] $\cH^{n-1}(U\cap B^{(1)} \cap J_v)<\delta;$
 
 \item[\rm(iv)] for any norm $\phi$ in $\R^n$ satisfying \eqref{fjskiii098}
 \begin{align*}
 & \int_{U\cap\p^*A} \phi(\nu_A)\d\cH^{n-1} +  
 2\int_{U\cap A^{(1)}\cap J_u} \phi(\nu_{J_u})\d\cH^{n-1}\ge 
 \int_{U\cap\p^*B} \phi(\nu_B)\d\cH^{n-1} - \delta.  
 \end{align*}
\end{itemize}

\end{corollary}

\begin{proof}
Let   
$
\epsilon:=\frac{\delta}{8}.
$
Since $\cH^{n-1}(U\cap A^{(1)}\cap J_u)<+\infty,$ there exists an open set $U'\strictlyincluded U$ such that 
\begin{equation}\label{jumps_inside_omega918}
\cH^{n-1}((U\setminus U')\cap A^{(1)}\cap J_u)< \epsilon. 
\end{equation}
By Lemma \ref{lem:opening_cracks} applied with $U',$ $K:=U'\cap A^{(1)}\cap J_u$ and $\epsilon,$ we find an at most countable family $\{C_i\}_{i\ge1}$ of open sets of finite perimeter such that

\begin{itemize}
\item[($a_1$)] $C_i\strictlyincluded U'$ and $\cH^{n-1}(\p C_i\setminus \p^*C_i)=0$; 

\item[($a_2$)]  $\cH^{n-1}([U' \cap K]\setminus \bigcup_i C_i) < \epsilon$ and $|\bigcup_i C_i|< \epsilon;$ 

\item[($a_3$)]  
$$
\sum\limits_{i\ge1} \int_{\p C_i} \phi(\nu_{C_i})\,\d\cH^{n-1} < 2\int_{U' \cap K} \phi(\nu_{K}^{})\,\d\cH^{n-1} + \epsilon.  
$$
\end{itemize} 
Define 
$$
G:=\bigcup_{i\ge1} C_i.
$$
We claim that $G$ satisfies the assertion of the lemma. Indeed, (i) is obvious and (ii) follows  from ($a_2$). By construction, $B^{(1)}\cap J_v =_{\cH^{n-1}} B^{(1)}\cap J_u$ and hence, by \eqref{jumps_inside_omega918} and ($a_2$) we have
$$
\cH^{n-1}(U\cap B^{(1)}\cap J_v) \le \cH^{n-1}((U\setminus U')\cap A^{(1)}\cap J_u) + \cH^{n-1}(U'\cap A^{(1)}\cap J_u\setminus G) <2\epsilon .
$$
Finally, since $\p^*B\setminus \p^*A\subset \p^*G,$ by ($a_3$)
\begin{align*}
\int_{U\cap \p^*B} \phi(\nu_B)\d\cH^{n-1} = &
\int_{U\cap \p^*B\cap \p^*A} \phi(\nu_B)\d\cH^{n-1} +
\int_{U\cap \p^*B\setminus \p^*A} \phi(\nu_B)\d\cH^{n-1}\\
\le & \int_{U\cap \p^*A} \phi(\nu_A)\d\cH^{n-1}
+\sum\limits_{i\ge1} \int_{\p C_i} \phi(\nu_{C_i})\,\d\cH^{n-1}\\
\le & \int_{U\cap \p^*A} \phi(\nu_A)\d\cH^{n-1} + 
2\int_{U'\cap K} \phi(\nu_{J_u}^{})\,\d\cH^{n-1}   
+\epsilon.
\end{align*} 
\end{proof}

The next lemma is a counterpart of Lemma \ref{lem:opening_cracks},  related to ``opening'' cracks along $\Sigma.$ Notice that in this case the opening should not go beyond $\Omega,$ and thus  we need to replace the jump of $u$ only from one side (see Corollary \ref{cor:functions_with_cracks_Sigma}). This is the reason for having $\varphi$ (without factor $2$) in the jump energy along $\Sigma$ in the functional $\cS.$

\begin{lemma}\label{lem:holes_near_Sigma}
Let $U\strictlyincluded \Ins{\Omega}$ be an open set, $\delta\in(0,1)$ and  $K\subset U\cap\Sigma$ be any $\cH^{n-1}$-measurable set. Then there exist an open set $C\subset U\cap \Omega$ of finite perimeter such that 
\begin{itemize}[left=15pt]
 \item[\rm(i)] $C\strictlyincluded U$ and $\cH^{n-1}(\p C\setminus \p^*C)=0;$
 
 \item[\rm(ii)] $\cH^{n-1}(K\setminus \p C)=\cH^{n-1}(K\setminus \cl{C})<\delta$ and $|C|<\delta;$
 
 \item[\rm(iii)] $\cH^{n-1}(U\cap \Sigma \cap\p C \setminus K)<\delta;$
 
 \item[\rm(iv)] for any norm $\phi$ in $\R^n$ satisfying \eqref{fjskiii098}
 $$
 \int_{\Omega\cap \p C} \phi(\nu_C)\d\cH^{n-1} \le \int_K\phi(\nu_\Sigma) \d\cH^{n-1} +\delta, 
 $$
 and 
 $$
 \int_{\p C} \phi(\nu_C)\d\cH^{n-1} \le 2\int_K\phi(\nu_\Sigma) \d\cH^{n-1} +\delta .
 $$ 
\end{itemize}
\end{lemma}

\begin{proof}
Let 
$$
\epsilon:= \frac{\delta}{8(1+\bound_2)(1 + \cH^{n-1}(\Sigma))}.
$$
We divide the proof into two steps.

{\it Step 1.} Let $Q_r(x_0)\subset U$ be a cube centered at $x\in\Sigma$ such that $\Sigma\cap Q_r(x_0) = {\rm graph}(f)$ for some Lipschitz function $f:V\to\R$ and a cube $V\subset \R^{n-1},$ and assume that $S\cap Q_r(x_0)$ is a subgraph of $f.$ 
Let $V''\strictlyincluded V'\strictlyincluded V$ be open sets such that 
$$
\cH^{n-1}(V\setminus V'') < \frac{\cH^{n-1}(Q_r(x_0)\cap\Sigma)}{1 + \Lip(f)}\,\epsilon
$$ 
and for $\gamma\in(0,\frac{\cH^{n-1}(Q_r(x_0)\cap\Sigma)}{1+\cH^{n-1}(V)}\,\epsilon)$ let $g\in \Lip_c(V;[0,\gamma])$ be such that $g\equiv \gamma$ in $V'',$ $\supp(g)=V'$ and $\Lip(g)<1.$ We may assume that $\gamma$ is so small that the set open set $C,$ whose boundary lies on the graphs of $f$ and $f+g,$ is compactly contained in $Q_r(x_0)$ and $C\cap S=\emptyset.$
Then 
\begin{align*}
\int_{\Omega\cap \p C}  \phi(\nu_C)\d\cH^{n-1} = & \int_{V'} \phi(-\nabla(f+g),1)\d\cH^{n-1} \le 
\int_{V'} \Big(\phi(-\nabla f,1) + \phi(-\nabla g,0)\Big)\d\cH^{n-1} \\
\le & \int_{V} \phi(-\nabla f,1)\d\cH^{n-1} + \bound_2\Lip(f)\cH^{n-1}(V'\setminus V'')\\
< & \int_{Q_r(x_0)\cap\Sigma}\phi(\nu_\Sigma)\d\cH^{n-1} + \bound_2\cH^{n-1}(Q_r(x_0)\cap\Sigma)\epsilon.
\end{align*}
Similarly, 
\begin{align*}
\int_{\p C}  \phi(\nu_C)\d\cH^{n-1} = & \int_{V'} [\phi(-\nabla(f+g),1) + \phi(-\nabla f,1)]\d\cH^{n-1} \\
\le &  
\int_{V'} \Big(2\phi(-\nabla f,1) + \phi(-\nabla g,0)\Big)\d\cH^{n-1} \\
< & 2\int_{Q_r(x_0)\cap\Sigma}\phi(\nu_\Sigma)\d\cH^{n-1} + \bound_2\cH^{n-1}(Q_r(x_0)\cap\Sigma)\epsilon.
\end{align*}
Also by the Fubini's theorem,
$$
|C| = \int_{V'} g \d x'\le \gamma \cH^{n-1}(V') < \cH^{n-1}(Q_r(x_0)\cap\Sigma)\,\epsilon.
$$
Finally, 
$$
\begin{aligned}
\cH^{n-1}(Q_r(x_0)\cap\Sigma\setminus \p C) = & \cH^{n-1}(Q_r(x_0)\cap\Sigma\setminus \cl{C}) 
=\int_{V\setminus V'}\sqrt{1+|\nabla f|^2}\d \cH^{n-1}\\
\le & (1+\Lip(f)) \cH^{n-1}(V\setminus V') < \cH^{n-1}(Q_r(x_0)\cap\Sigma)\,\epsilon.
\end{aligned}
$$

{\it Step 2.} Since $\Sigma$ is Lipschitz and $K$ is $\cH^{n-1}$-rectifiable, we can find a finite  family $Q_{r_1,\nu_1}(x_1),\ldots, Q_{r_m,\nu_m}(x_m) \subset U$ of pairwise disjoint cubes centered at $K$ such that 
\begin{itemize}
 \item[($\rm a_1$)] for each $j,$ $\Sigma\cap Q_{r_j,\nu_j}(x_j)$ is a graph of a Lipschitz function in $\nu_j$ direction;

 \item[($\rm a_2$)] $\theta(K,x_j)=\theta(\Sigma,x_j)=1,$ and the unit normals $\nu_K(x_j)$ and $\nu_{\Sigma}(x_j)$ exist and coincide with $\nu_j;$
 
 \item[($\rm a_3$)] $(1-\epsilon)r_j^{n-1} <\cH^{n-1}(Q_{r_j,\nu_j}(x_j)\cap \Sigma\cap K)\le \cH^{n-1}(Q_{r_j,\nu_j}(x_j)\cap \Sigma)<(1+\epsilon)r_j^{n-1};$
 
 \item[($\rm a_4$)] $\cH^{n-1}\Big(K\setminus \bigcup_{j=1}^m Q_{r_j,\nu_j}(x_j)\Big)<\epsilon.$
\end{itemize}

Note that by ($\rm a_3$),
\begin{equation}\label{K_out_sigma89}
\cH^{n-1}(Q_{r_j,\nu_j}(x_j)\cap \Sigma\setminus K) <2\epsilon r_j^{n-1} < \frac{2\epsilon}{1-\epsilon} \cH^{n-1}(Q_{r_j,\nu_j}(x_j)\cap \Sigma). 
\end{equation}
By Step 1, for each $j$ we can contruct an open set $C_j\strictlyincluded Q_{r_j,\nu_j}(x_j)$ satisfying $C_j\cap\substrate =\emptyset,$
\begin{align}\label{est_Cjjjj9099}
\int_{\Omega\cap \p C_j} \phi(\nu_{C_j})\d\cH^{n-1} < \int_{Q_{r_j,\nu_j}(x_j)\cap\Sigma}\phi(\nu_\Sigma)\d\cH^{n-1} + \bound_2\cH^{n-1}(Q_{r_j,\nu_j}(x_j)\cap\Sigma)\epsilon 
\end{align}
and 
\begin{align*} 
\int_{\p C_j} \phi(\nu_{C_j})\d\cH^{n-1} < 2 \int_{Q_{r_j,\nu_j}(x_j)\cap\Sigma}\phi(\nu_\Sigma)\d\cH^{n-1} + \bound_2\cH^{n-1}(Q_{r_j,\nu_j}(x_j)\cap\Sigma)\epsilon.
\end{align*}
Moreover,
\begin{equation}\label{vol_estCjjjj}
|C_j| < \cH^{n-1}(Q_{r_j,\nu_j}(x_j)\cap\Sigma)\epsilon  
\end{equation}
and
\begin{align}\label{out_jshsh00}
\cH^{n-1}(Q_{r_j,\nu_j}(x_j)\cap\Sigma\setminus \p C_j) < \cH^{n-1}(Q_{r_j,\nu_j}(x_j)\cap\Sigma)\,\epsilon.
\end{align}

We claim that $C= \bigcup_{i=1}^m C_j$ satisfies the assertions of the lemma.

(i) By construction, $C\strictlyincluded U$ and since each $C_j$ is almost Lipschitz, $\cH^{n-1}(\p C_j\setminus \p^*C_j)=0.$ Hence, by the pairwise disjointness of $\cl{C_j},$ $\cH^{n-1}(\p C\setminus \p^*C)=0$ and (i) follows.

(ii) By ($\rm a_4$) and \eqref{out_jshsh00},
\begin{align*}
\cH^{n-1}(K\setminus \p C) \le & \cH^{n-1}\Big(K\setminus \bigcup_{j=1}^m Q_{r_j,\nu_j}(x_j)\Big) + 
\sum\limits_{j=1}^m \cH^{n-1}( Q_{r_j,\nu_j}(x_j)\cap \Sigma  \setminus \p C_j) \\
< & \epsilon +  \sum\limits_{j=1}^m \cH^{n-1}( Q_{r_j,\nu_j}(x_j) \cap\Sigma)\,\epsilon \le
(1+\cH^{n-1}(\Sigma))\epsilon <\delta.
\end{align*}
Moreover, by \eqref{vol_estCjjjj},
$$
|C| \le \sum\limits_{j=1}^n |C_j| \le \cH^{n-1}(\Sigma)\epsilon <\delta.
$$

(iii) By \eqref{K_out_sigma89},
\begin{align*}
\cH^{n-1}( U\cap\Sigma\cap \p C \setminus K) = & \sum\limits_{j=1}^m \cH^{n-1}(Q_{r_j,\nu_j}(x_j)\cap \Sigma\cap \p C_j \setminus K) \\
\le & \sum\limits_{j=1}^m \cH^{n-1}(Q_{r_j,\nu_j}(x_j)\cap \Sigma\setminus K) <\frac{2\epsilon}{1-\epsilon}\cH^{n-1}(\Sigma) <\delta.
\end{align*}

(iv) Since $\p^* C\subset \cup_j \p^* C_j,$ by \eqref{est_Cjjjj9099} we have
$$
\int_{\Omega\cap\p C} \phi(\nu_C)\d\cH^{n-1} 
\le \sum\limits_{j=1}^m \int_{Q_{r_j,\nu_j}(x_j)\cap\Sigma}\phi(\nu_\Sigma)\d\cH^{n-1} + \bound_2\cH^{n-1}(\Sigma)\epsilon
$$
Moreover, by \eqref{fjskiii098},
$$
\int_{Q_{r_j,\nu_j}(x_j)\cap\Sigma}\phi(\nu_\Sigma)\d\cH^{n-1} \le 
\int_{Q_{r_j,\nu_j}(x_j)\cap K} \phi(\nu_\Sigma)\d\cH^{n-1} + \bound_2 \cH^{n-1}(Q_{r_j,\nu_j}(x_j)\cap \Sigma\setminus K),
$$
and thus by \eqref{K_out_sigma89},
$$
\int_{\Omega\cap\p C} \phi(\nu_C)\d\cH^{n-1}  
\le 
\int_K \phi(\nu_\Sigma)\d\cH^{n-1} + \Big(\tfrac{\bound_2}{1-\epsilon} + \bound_2) \cH^{n-1}(\Sigma)\epsilon < \int_K \phi(\nu_\Sigma)\d\cH^{n-1} +\delta.
$$
Finally, since $\Sigma\cap \p C\subset K\cup (\Sigma\cap \p C\setminus K),$  
\begin{align*}
\int_{\p C} \phi(\nu_C)\d\cH^{n-1} = & \int_{\Omega\cap\p C} \phi(\nu_C)\d\cH^{n-1} + \int_{\Sigma\cap \p C} \phi(\nu_\Sigma)\d\cH^{n-1}\\
\le & 2\int_K\phi(\nu_\Sigma)\d\cH^{n-1} + 3\bound_2\cH^{n-1}(\Sigma) \epsilon + 
\bound_2\cH^{n-1}(\Sigma\cap \p C\setminus K)\\
< & 2\int_K\phi(\nu_\Sigma)\d\cH^{n-1} + 7\bound_2\cH^{n-1}(\Sigma) \epsilon 
< 2\int_K \phi(\nu_\Sigma)\d\cH^{n-1} +\delta. 
\end{align*}
\end{proof}

\begin{corollary}\label{cor:functions_with_cracks_Sigma}
Let $U\strictlyincluded \Ins{\Omega}$ be an open set, $(A,u)\in\admissible$ and $\delta>0.$ Then there exists an open set $G\subset\Omega$ of finite perimeter such that  
\begin{itemize}[left=15pt]
 \item[\rm (i)] $G\strictlyincluded U$ and $|G|<\delta;$; 

 \item[\rm (ii)] the configuration $(B,v)$ with $B:=A\setminus G$ and $v:=u\chi_{B\cup S}$ belongs to $\admissible;$

 \item[\rm (iii)] 
 $$
 \cH^{n-1}(\Sigma\cap \p^*G\setminus (\p^*A\cap J_u)) + \cH^{n-1}(U\cap \Sigma\cap J_u\cap \p^*A\setminus \p^*G) < \delta
 $$ 
 and 
 $$
 \cH^{n-1}(U\cap B^{(1)}\cap J_v)<\delta + 
 \cH^{n-1}(U\cap \Sigma\cap J_v\cap \p^*B) < \delta;
 $$
 
 \item[\rm (iv)] for any norm $\phi$ in $\R^n$ satisfying \eqref{fjskiii098} 
 \begin{multline}
\int_{U\cap\Omega\cap \p^*A} \phi(\nu_A)\d\cH^{n-1} + 2\int_{U\cap A^{(1)}\cap J_u} \phi(\nu_\Sigma)\d\cH^{n-1}  +
\int_{U\cap \Sigma\cap\p^*A\cap J_u}  \phi(\nu_\Sigma)\d\cH^{n-1} \\
  \ge  \int_{U\cap\Omega\cap \p^*B} \phi(\nu_B)\d\cH^{n-1} -\delta 
 \ge \int_{U\cap\p^*G} \phi(\nu_G)\d\cH^{n-1} -\delta  
  \label{matmat343}
 \end{multline}
 and 
 \begin{multline}
 \int_{U\cap\Omega\cap \p^*A} \phi(\nu_A)\d\cH^{n-1} +  2 
\int_{U\cap A^{(1)}\cap J_u} \phi(\nu_\Sigma)\d\cH^{n-1}   +   
 2 \int_{U\cap \Sigma\cap\p^*A\cap J_u} \phi(\nu_\Sigma)\d\cH^{n-1} \\
  \ge \int_{U\cap\Omega\cap \p^*B} \phi(\nu_B)\d\cH^{n-1} 
 +\int_{\Sigma\cap \p^*G} \phi(\nu_\Sigma)\d\cH^{n-1} -\delta 
 \ge \int_{\p^*G} \phi(\nu_G)\d\cH^{n-1} 
 \label{matmat347}
 \end{multline}
\end{itemize}

\end{corollary}

\begin{proof}
The last inequalities in \eqref{matmat343} and \eqref{matmat347} follow from the definition of $B.$
To prove the first inequalities, let $\epsilon:=\frac{\delta}{16(1+\bound_2)}$ and $U'\strictlyincluded \Omega\cap U$ be any open set such that
\begin{equation*}
\cH^{n-1}(\Omega\cap U\cap J_u\setminus U')<\epsilon.
\end{equation*}
By Corollary \ref{cor:functions_with_good_cracks} applied with $U',$ $(A,u)\in \admissible$ and $\epsilon,$ we find an open set $D'\strictlyincluded U'$ of finte perimeter such that
\begin{itemize}
 \item[($\rm a_1$)] the configuration $(B',v')$ with $B':=A\setminus D'$ and $v':=u\chi_{S\cup B'}$ belongs to $\admissible$;
 
 \item[($\rm a_2$)] $|D'|<\epsilon;$
 \item[($\rm a_3$)] $\cH^{n-1}(U'\cap [B']^{(1)} \cap J_{v'})<\epsilon;$
 \item[($\rm a_4$)] 
 $$
 \int_{U'\cap \p^*A}\phi(\nu_A)\d\cH^{n-1} + 2\int_{U'\cap A^{(1)}\cap J_u}\phi(\nu_{J_u})\d\cH^{n-1}
 \ge \int_{U'\cap \p^*B'}\phi(\nu_{B'})\d\cH^{n-1} -\epsilon.
 $$
\end{itemize}
Next, choose another open set $U''\strictlyincluded U$ such that $\cl{U'}\cap \cl{U''}=\emptyset$ and
\begin{equation*}
\cH^{n-1}((U\setminus U'')\cap \Sigma\cap \p^*A \cap J_u)< \epsilon. 
\end{equation*}
By Lemma \ref{lem:holes_near_Sigma} applied with $U'',$ $\epsilon$ and $K:=U''\cap \Sigma\cap\p^*A\cap J_u,$ we find an open set $C'\subset U''\cap \Omega$ of finite perimeter such that
\begin{itemize}
\item[($\rm b_1$)] $C'\strictlyincluded U'$ and $\cH^{n-1}(\p C'\setminus \p^*C')=0$; 

\item[($\rm b_2$)]  $\cH^{n-1}(K\setminus \p C') = \cH^{n-1}(K\setminus \cl{C'}) < \epsilon$ and $|C'|< \epsilon;$ 

\item[($\rm b_3$)]  $\cH^{n-1}(\Sigma\cap\p C'\setminus K) <\epsilon;$

\item[($\rm b_4$)]
 $$
 \int_{\Omega\cap \p C'} \phi(\nu_{C'})\d\cH^{n-1} \le \int_K \phi(\nu_\Sigma) \d\cH^{n-1} +\epsilon, 
 $$
 and 
 $$
 \int_{\p C'} \phi(\nu_{C'})\d\cH^{n-1} \le 2\int_K\phi(\nu_\Sigma) \d\cH^{n-1} +\epsilon.
 $$ 
\end{itemize}
Now, the set
$$
G:=C'\cup D'
$$
satisfies all assertions of the lemma. Indeed, the assertions (i)-(iii) follow  from ($\rm a_1$)-($\rm a_3$) and ($\rm b_1$)-($\rm b_3$), whereas (iv) follows from the inclusion
$\p^*B\setminus \p^*A\subset\Omega\cap  \p C'\cup \p D'$ and conditions ($\rm a_4$)  and ($\rm b_4$).
\end{proof}

\section{$\tau_\admissible$-lower semicontinuity}\label{sec:lower_semicontos}
 
In this section we prove Theorem \ref{teo:lower_semicontinuity}, following the arguments of \cite[Proposition 4.1]{HP:2020_arma}, and in particular using density estimates for certain Radon measures associated with $\cF.$
We start with a lower bound for the localized surface energy.

\begin{proposition}\label{prop:estimate_inner_jump}
Let $\delta\in(0,1),$ $Q_{r,\nu}(x_0)\strictlyincluded \Int{\Omega\cup\Sigma \cup\substrate},$ $r>0,$ $\nu\in\S^{n-1},$ be a cube and $\Gamma\subset Q_{r,\nu}(x_0)$ be an $(n-1)$-dimensional Lipschitz graph separating $Q_{r,\nu}(x_0)$ into two connected components such that 
\begin{itemize}[left=20pt]
\item[\rm(a1)] $x_0\in\Gamma,$ $\nu_\Gamma(x_0)=\nu$ and 
 $$
 |\nu_\Gamma(x)-\nu|<\delta\quad \text{and}\quad 
 |(x-x_0)\cdot \nu|<\tfrac{\delta r}{2}
 \quad\text{for all $x\in \Gamma;$}
 $$

\item[\rm(a2)] $\cH^{n-1}(Q_{r,\nu}(x_0)\cap \Gamma) <(1+\delta)r^{n-1}.$
\end{itemize}
Assume that a sequence $\{(A_k,u_k)\}\subset \admissible $ and a configuration $(A,u)\in\admissible$ satisfy 

\begin{itemize}
\item[\rm(a3)] $u_k=\xi$ for some $\xi\in (0,1)^n\setminus \Xi_{\{(A_k,u_k)\}}$  (see Remark \ref{rem:ext_u_out_A}) and 
$$
M:=\sup_{k\ge1} \cF(A_k,u_k)<+\infty;
$$
 
\item[\rm(a4)] $A_k\to A$ in $L^1(\R^n);$

\item[\rm(a5)] $\cH^{n-1}(Q_{r,\nu}(x_0)\cap \p^*(A\cup S))<\delta r^{n-1}$ and $|(A\cup S)\cap Q_{r,\nu}(x_0)|>(1-\delta)r^n;$

\item[\rm(a6)] either 
$$
u_k\to u\quad\text{a.e.\  in $Q_{r,\nu}(x_0)$}
$$
and  
$$
K:=Q_{r,\nu}(x_0)\cap J_u
$$ 
or there exists a set of finite perimeter $E\subset Q_{r,\nu}(x_0)$ such that 
$$
u_k\to u\quad \text{a.e.\  in $Q_{r,\nu}(x_0)\setminus E$}\qquad \text{and}\qquad  
|u_k|\to +\infty \quad \text{a.e.\  in $Q_{r,\nu}(x_0)\cap E,$}
$$
and 
$$
K:=Q_{r,\nu}(x_0)\cap\p^*E
$$ 
(see Figure \ref{fig:setE_for_lemma}).

\begin{figure}[htp]
\includegraphics[width=0.9\textwidth]{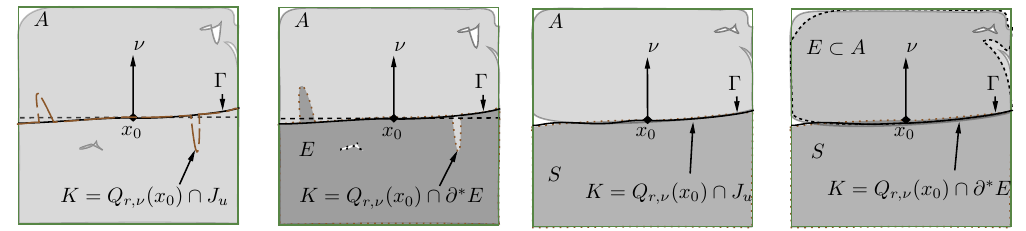} 
\caption{{\small Set $K$ in Proposition \ref{prop:estimate_inner_jump}.}} \label{fig:setE_for_lemma}
\end{figure}

\item[\rm(a7)] the set $K$ satisfies
\begin{itemize}
 \item[\rm(a7.1)] $\nu_K(x_0)=\nu$ and $\theta(K,x_0)=\theta(\Gamma\cap K,x_0)=1;$ 
 
 \item[\rm(a7.2)] $\cH^{n-1}(K\cap\Gamma)>(1-\delta)r^{n-1};$
 
 \item[\rm(a7.3)] $\cH^{n-1}(K\setminus\Gamma)<\delta r^{n-1}.$
\end{itemize}  

\end{itemize}
\noindent
We also denote by $\phi$ a norm in $\R^n$ satisfying \eqref{fjskiii098}.
Let $C,D\strictlyincluded Q_{r,\nu}(x_0)$ be given by Lemma \ref{lem:creating_hole} applied with $\delta,$ $\Gamma$ and $K.$ Then there exist $c'=c_{\bound_2}'>0$ and $k_\delta':=k_\delta'(\bound_2)>0$ such that for any $k>k_\delta':$

\begin{itemize}[left=17pt]
\item[\rm(i)]  if $Q_{r,\nu}(x_0)\strictlyincluded \Omega,$ then 
\begin{align}\label{interior_estimates_CO}
\int_{D\cap\p^*A_k} \phi(\nu_{A_k})\d\cH^{n-1} +   2\int_{D \cap A_k^{(1)} \cap J_{u_k}}  \phi(\nu_{J_{u_k}}) \d\cH^{n-1}
\ge   2\int_K\phi(\nu_K)\d\cH^{n-1} - c'\delta r^{n-1} \nonumber \\
\ge   \int_{\p D} \phi(\nu_D)\d\cH^{n-1} - (c'+5\bound_2)\delta r^{n-1};
\end{align}

\item[\rm(ii)] if $x_0\in \Sigma$ and $\Gamma = Q_{r,\nu}(x_0)\cap \Sigma,$ then 
\begin{align}\label{Sigma_2estimates_CO}
\int_{C\cap\p^*A_k} \phi(\nu_{A_k})\d\cH^{n-1} +   2\int_{C \cap A_k^{(1)} \cap J_{u_k}}   \phi(\nu_{J_{u_k}}) & \d\cH^{n-1}
+ 2\int_{\Sigma\cap\p C\cap \p^*A_k\cap J_{u_k}} \phi(\nu_\Gamma)\d\cH^{n-1} \nonumber \\
\ge & 2\int_K\phi(\nu_K)\d\cH^{n-1} - c'\delta r^{n-1} \nonumber \\
\ge  & \int_{\p C} \phi(\nu_C)\d\cH^{n-1} - (c'+ 5\bound_2)\delta r^{n-1}.
\end{align}
\end{itemize}
\end{proposition}

The proof of this proposition is left after the proof of Theorem \ref{teo:lower_semicontinuity}.  Actually, in the proof of lower semicontinuity we only use the following corollary of Proposition \ref{prop:estimate_inner_jump}; the assertions involving the sets $C,$ $D$ and $E$ will used in the proof of compactness.

\begin{corollary}\label{cor:further_estimates}
Under assumptions of Proposition \ref{prop:estimate_inner_jump}, together with 
\begin{itemize}[left=20pt]
\item[\rm(a6)] $u_k\to u\quad\text{a.e.\  in $Q_{r,\nu}(x_0)$}$
and  
$$
K:=Q_{r,\nu}(x_0)\cap J_u,
$$ 
\end{itemize}
there exist $c'=c_{\bound_2}'>0$ and $k_\delta':=k_\delta'(\bound_2)>0$ such that for any $k>k_\delta':$

\begin{itemize}[left=17pt]
\item[\rm(i)]  if $Q_{r,\nu}(x_0)\strictlyincluded \Omega,$ then 
\begin{align*}
\int_{Q_{r,\nu}(x_0) \cap\p^*A_k} \phi(\nu_{A_k})\d\cH^{n-1} +   2\int_{Q_{r,\nu}(x_0) \cap A_k^{(1)} \cap J_{u_k}}   \phi(\nu_{J_{u_k}}) \d\cH^{n-1}
\ge  2\int_K\phi(\nu_K)\d\cH^{n-1} - c'\delta r^{n-1} ;
\end{align*}

\item[\rm(ii)] if $x_0\in \Sigma$ and $\Gamma = Q_{r,\nu}(x_0)\cap \Sigma,$ then 
\begin{multline*}\int_{Q_{r,\nu}(x_0)\cap\p^*A_k} \phi(\nu_{A_k})\d\cH^{n-1} +   2\int_{Q_{r,\nu}(x_0) \cap A_k^{(1)} \cap J_{u_k}} \phi(\nu_{J_{u_k}}) \d\cH^{n-1} \\ 
+ 2\int_{Q_{r,\nu}(x_0)\cap \Sigma\cap \p^*A_k\cap J_{u_k}} \phi(\nu_\Gamma)\d\cH^{n-1} 
\ge2\int_K\phi(\nu_K)\d\cH^{n-1} - c'\delta r^{n-1} .
\end{multline*}
\end{itemize}

\end{corollary}

\begin{proof}[\textbf{Proof of Theorem \ref{teo:lower_semicontinuity}}]
In view of Remark \ref{rem:ext_u_out_A}, we may assume that $u_k=\xi$ for some $\xi\in (0,1)^n\setminus \Xi_{\{(A_k,u_k)\}}.$
Moreover, there is no loss of generality in assuming $liminf$ in \eqref{eq:lsc_functional} is a finite limit. Thus,
$$
M:=\sup\limits_{k\ge1} \,\cF(A_k,u_k)<+\infty.
$$
In particular, $\{(A_k,u_k)\}$ satisfies the assumptions (a3) and (a4) of Proposition \ref{prop:estimate_inner_jump}.

Let  
\begin{align*}
\mu_k(B):=  & \int_{B\cap \Omega\cap \p^*A_k} \varphi(x,\nu_{A_k})\,\d \cH^{n-1} + 2\int_{B\cap A_k^{(1)}\cap J_{u_k}} \varphi(x,\nu_{J_{u_k}}) \,\d\cH^{n-1} \nonumber \\
& +  \int_{B\cap \Sigma\cap \p^*A_k\setminus J_{u_k}} [\beta + \varphi(x,\nu_\Sigma)]\,\d\cH^{n-1} + 2\int_{B\cap \Sigma\cap \p^*A_k\cap J_{u_k}} \varphi(x,\nu_\Sigma)\,\d\cH^{n-1} \nonumber\\
& + \int_{B\cap \Sigma\setminus \p^*A_k} \varphi(x,\nu_\Sigma)\,\d\cH^{n-1} + \int_{B\cap (A\cup \substrate)} W(x,\str{u_k}-\bM_0)\d x
\end{align*}
and 
\begin{align*}
\mu(B):=  & \int_{B\cap \Omega\cap \p^*A} \varphi(x,\nu_A)\,\d \cH^{n-1} + 2\int_{B\cap A^{(1)} \cap J_u} \varphi(x,\nu_{J_u}) \,\d\cH^{n-1} \nonumber \\
& +  \int_{B\cap \Sigma\cap \p^*A\setminus J_u} [\beta + \varphi(x,\nu_\Sigma)]\,\d\cH^{n-1} + 2\int_{B\cap \Sigma\cap \p^*A\cap  J_u} \varphi(x,\nu_\Sigma)\,\d\cH^{n-1} \nonumber\\
& + \int_{B\cap \Sigma\setminus \p^*A} \varphi(x,\nu_\Sigma)\,\d\cH^{n-1}+\int_{B\cap (A\cup \substrate)} W(x,\str{u}-\bM_0)\d x
\end{align*}
be positive Radon measures in $\R^n.$
Notice that 
\begin{equation}\label{sankdnaida}
\mu_k(\R^n) = \cF(A_k,u_k) + \int_\Sigma \varphi(x,\nu_\Sigma)\,\d\cH^{n-1} 
\end{equation}
and
\begin{equation}\label{sandada1}
\mu(\R^n) = \cF(A,u) + \int_\Sigma \varphi(x,\nu_\Sigma)\,\d\cH^{n-1}. 
\end{equation}
In particular, 
$$
\sup\limits_{k\ge1} \mu_k(\R^n) \le M + \int_\Sigma \varphi(x,\nu_\Sigma)\,\d\cH^{n-1},
$$
and thus, there exist a positive Radon measure $\mu_0$ in $\R^n$ and a not relabelled subsequence $\{\mu_k\}$ such that $\mu_k\wk^* \mu_0.$ 
Note that if we show
\begin{equation}\label{compar_mu0_vs_mu}
\mu_0\ge\mu,
\end{equation}
then \eqref{eq:lsc_functional} directly follows from \eqref{compar_mu0_vs_mu}, \eqref{sankdnaida} and \eqref{sandada1}.
Therefore, the remaining part of the proof is devoted to the proof \eqref{compar_mu0_vs_mu}. Since $\mu,\mu_0\ge0$ and the support of $\mu$ is explicit, to establish \eqref{compar_mu0_vs_mu} it suffices to prove  the following density estimates:
\begin{subequations}
\begin{align}
& \frac{\d\mu_0}{\d\cH^{n-1}\res[ \Omega\cap \p^*A ]}\,(x) \ge \varphi(x,\nu_A(x))\quad\text{$\cH^{n-1}$-a.e.\  $x\in (\Omega\cap\p^*A)\cup (\Sigma\setminus \p^*A),$} \label{at_essential_boundary}\\
& \frac{\d\mu_0}{\d\cH^{n-1}\res [A^{(1)}\cap J_u]}\,(x) \ge 2\varphi(x,\nu_{J_u}(x)) \quad\text{$\cH^{n-1}$-a.e.\  $x\in A^{(1)}\cap J_u,$} \label{at_internal_crack}\\
& \frac{\d\mu_0}{\d\cH^{n-1}\res [\Sigma\cap \p^*A\cap J_u]}\,(x) \ge 2\varphi(x,\nu_\Sigma(x)) \quad\text{$\cH^{n-1}$-a.e.\  $x\in \Sigma\cap \p^*A\cap J_u,$} \label{at_delamination}\\
& \frac{\d\mu_0}{\d\cH^{n-1}\res [\Sigma\cap \p^*A]}\,(x) \ge \beta(x) + \varphi(x,\nu_\Sigma(x))\quad\text{$\cH^{n-1}$-a.e.\  $x\in \Sigma\cap\p^*A,$} \label{at_contact_boundary}\\
& \frac{\d\mu_0}{\d\cH^{n-1}\res [\Sigma\setminus  \p^*A]}\,(x) \ge \varphi(x,\nu_\Sigma(x))\quad\text{$\cH^{n-1}$-a.e.\  $x\in \Sigma\setminus\p^*A,$} \label{at_exposed}\\
&\frac{\d\mu_0}{\d\cL^n\res [A\cup \substrate]}\,(x) \ge W(x,\str{u}(x)-\bM_0(x))\quad\text{$\cL^n$-a.e.\  $x\in A\cup \substrate.$} \label{at_bulk}
\end{align}
\end{subequations}
\smallskip 

{\it Proofs of \eqref{at_essential_boundary},  \eqref{at_contact_boundary} and \eqref{at_exposed}.} 
By assumptions (H1)-(H3), the capillary functional
$$
\mathcal{C}(E;U) = \int_{U\cap \p^*E} \varphi(x,\nu_E)\d \cH^{n-1} + \int_{U\cap \Sigma\cap \p^*E} [\beta + \varphi(x,\nu_\Sigma)]\,\d\cH^{n-1} + \int_{U\cap \Sigma\setminus \p^*E} \varphi(x,\nu_\Sigma)\,\d\cH^{n-1}
$$ 
is $L^1(U)$-lower semicontinuous in any open set $U\subset\R^n$ (see e.g., \cite[Theorem 3.4]{ADT:2017}). As $A_k\to A$ and $\mu_k\wk^*\mu_0,$ for any ball  $B_r(x_0)$ with $\mu_0(\p B_r(x_0))=0$ we have
\begin{align*}
\mu_0(B_r(x_0)) = &  \lim\limits_{k\to+\infty} \mu_k(B_r(x_0))
\ge \liminf\limits_{k\to+\infty} \mathcal{C}(A_k, B_r(x_0)) \ge \mathcal{C}(A, B_r(x_0)).
\end{align*}
This inequality and the Besicovitch differentiation theorem imply  \eqref{at_essential_boundary},  \eqref{at_contact_boundary} and \eqref{at_exposed}.  
\smallskip

{\it Proof of \eqref{at_internal_crack}.} Fix $\epsilon\in(0,2^{-10})$ and let $K:=A^{(1)}\cap J_u.$ By the $\cH^{n-1}$-rectifiability of $K$, there exists an at most countable family $\{\Gamma_l\}$ of $(n-1)$-dimensional $C^1$-graphs such that 
$$
\cH^{n-1}\Big(K \setminus \bigcup_{l\ge1} \Gamma_l \Big)=0.
$$
Let $x_0\in L$ be such that 

\begin{itemize}
 \item[($\rm a_1$)]  $x_0\in\Gamma_l$ for some $l\ge1$ so that the generalized unit normal $\nu_0:=\nu_K(x_0)$ to $L$ at $x_0 $ exists and equals to $\nu_{\Gamma_l}(x_0)$;
 
 \item[($\rm a_2$)]  $\theta(K,x_0)=\theta(\Gamma_l\cap K,x_0)=1;$ 
 
 \item[($\rm a_3$)] $\frac{\d\mu_0}{\d\cH^{n-1}\res K}(x_0)$ exists;
 
 \item[($\rm a_4$)]  $\lim\limits_{r\to0} \frac{1}{r^{n-1}} \int_{Q_{r,\nu_0}(x_0)\cap K}\varphi(x_0,\nu_K)\d\cH^{n-1} =\varphi(x_0,\nu_0).$
\end{itemize}
By the $\cH^{n-1}$-rectifiability of $K,$ \cite[Theorem 2.63]{AFP:2000} and the Besicovitch differentiation theorem, the set of $x_0\in K$ for which at least one of these conditions fails is $\cH^{n-1}$-negligible.
Since $\varphi$ is uniformly continuous in $\cl{\Omega},$ there exists $r_{1,\epsilon}>0$ such that 
\begin{equation}\label{unif_phi_contiu}
|\varphi(x,\nu) - \varphi(y,\nu)|<\epsilon \quad\text{whenever $|x-y|<r_{1,\epsilon}$ and $\nu\in\S^{n-1}.$} 
\end{equation}
Decreasing $r_{1,\epsilon}$ if necessary, we assume that $Q_{r_{1,\epsilon},\nu_0}(x_0)\strictlyincluded \Omega.$ Then for any $r\in(0,r_{1,\epsilon})$ 
\begin{align}
\mu_k(Q_{r,\nu_0}(x_0)) \ge \alpha_k(Q_{r,\nu_0}(x_0))
- \epsilon \Big( \cH^{n-1}(Q_{r,\nu_0}(x_0) \cap \p^*A_k) + 2\cH^{n-1}(Q_{r,\nu_0}(x_0)\cap A_k^{(1)}\cap J_{u_k})\Big), 
\label{lower_estimate_muk_alha_k}
\end{align}
where 
$$
\alpha_k(U):=\int_{U\cap\p^*A_k} \phi(\nu_{A_k})\d\cH^{n-1} + 2\int_{U\cap A_k^{(1)}\cap J_{u_k}} \phi(\nu_{J_{u_k}})\d\cH^{n-1}
$$
and $\phi(\nu):=\varphi(x_0,\nu).$ By assumption \eqref{finsler_norm} and the nonnegativity of the summands of $\mu_k,$ we have an a priori bound
$$
\cH^{n-1}(Q_{r,\nu_0}(x_0) \cap \p^*A_k) + 2\cH^{n-1}(Q_{r,\nu_0}(x_0)\cap A_k^{(1)}\cap J_{u_k}) \le \frac{\mu(Q_{r,\nu_0}(x_0))}{\bound_1},
$$
and thus inserting this inequality in \eqref{lower_estimate_muk_alha_k} we get
\begin{equation}\label{mu_k_below_estima}
\Big(1 + \frac{\epsilon}{\bound_1}\Big) \mu_k(Q_{r,\nu_0}(x_0)) \ge \alpha_k(Q_{r,\nu_0}(x_0)).
\end{equation}

Now we estimate $\alpha_k$ from below using Corollary \ref{cor:further_estimates} (a). Since $\Gamma_l$ is a $C^1$-graph, by ($a_1$) there exists $r_{2,\epsilon}\in(0,r_{1,\epsilon})$ such that 
\begin{itemize}  
\item  $\Gamma_l$ divides the cube $Q_{r_{2,\epsilon},\nu_0}(x_0)$ into two connected components; 
 
\item  $|\nu_{\Gamma_l}(x) - \nu_0|<\epsilon$ for any $x\in Q_{r_{2,\epsilon},\nu_0}(x_0)\cap\Gamma_l;$
 
\item  $|(x-x_0)\cdot \nu_0| <\epsilon r/2$ for any $r\in (0,r_{2,\epsilon})$ and $x\in Q_{r,\nu_0}(x_0)\cap \Gamma_l;$ 
 
\item  $\cH^{n-1}(Q_{r,\nu_0}(x_0)\cap \Gamma_l) <(1+\epsilon)r^{n-1}$ for all $r\in(0,r_{2,\epsilon}).$
\end{itemize}
In particular, for any $r\in(0,r_{2,\epsilon})$ the cube $Q_{r,\nu_0}(x_0)$ and the $C^1$-graph $\Gamma:=Q_{r,\nu_0}(x_0)\cap \Gamma_l$ satisfy the assumptions (a1)-(a2) of Proposition \ref{prop:estimate_inner_jump}. As we mentioned in the beginning of the proof, $\{(A_k,u_k)\}$ satisfies the assumptions (a3)-(a4) of Proposition \ref{prop:estimate_inner_jump}.
Moreover, by assumptions $x_0\in A^{(1)}$ and ($a_2$), there exists $r_{3,\epsilon}\in(0,r_{2,\epsilon})$ such that

\begin{itemize}

\item  $P(A,Q_{r,\nu_0}(x_0))< \epsilon r^{n-1}$ and $|A\cap Q_{r,\nu_0}(x_0)|>(1-\epsilon)r^{n-1}$ for all $r\in(0,r_{3,\epsilon});$
  
\item  $\cH^{n-1}(Q_{r,\nu_0}(x_0)\cap K\cap \Gamma_l)>(1-\epsilon)r^{n-1}$  for any $r\in(0,r_{3,\epsilon});$ 
 
\item  $\cH^{n-1}(Q_{r,\nu_0}(x_0)\cap K\setminus\Gamma_l)<\delta r^{n-1}$ for any $r\in(0,r_{3,\epsilon}).$ 
 \end{itemize}
Thus, assumptions (a5)-(a7) of Proposition \ref{prop:estimate_inner_jump} also hold. Therefore, by Corollary \ref{cor:further_estimates} (i), there exist  $k_\epsilon>0$ and $c'>0$ such that
$$
\alpha_k(Q_{r,\nu_0}(x_0)) \ge 2\int_{Q_{r,\nu_0}(x_0)\cap K} \phi(\nu_K)\d\cH^{n-1} -c'\epsilon r^{n-1} \quad \text{for all $k>k_\epsilon.$}
$$
This and \eqref{mu_k_below_estima} yield
$$
\Big(1 + \frac{\epsilon}{\bound_1}\Big) \mu_k(Q_{r,\nu_0}(x_0)) \ge 2\int_{Q_{r,\nu_0}(x_0)\cap K}\phi(\nu_K)\d\cH^{n-1} -c'\epsilon r^{n-1}.
$$
Now letting $k\to+\infty,$ for a.e.\  $r\in(0,r_{3,\epsilon})$ we get
$$
\Big(1 + \frac{\epsilon}{\bound_1}\Big) \mu_0(Q_{r,\nu_0}(x_0)) \ge 2\int_{Q_{r,\nu_0}(x_0)\cap K}\phi(\nu_K)\d\cH^{n-1} -c'\epsilon r^{n-1}.
$$
Hence, by ($a_3$) and ($a_4$),
$$
\Big(1 + \frac{\epsilon}{\bound_1}\Big) \frac{\d\mu_0}{\d\cH^{n-1}\res K} (x_0) = \Big(1 + \frac{\epsilon}{\bound_1}\Big) \lim\limits_{r\to0^+}\frac{\mu_0(Q_{r,\nu}(x_0))}{r^{n-1}} 
\ge 2\varphi(x_0,\nu_0) - c'\epsilon.
$$
Now letting $\epsilon\to0$ we obtain  \eqref{at_internal_crack}.
\smallskip 

{\it Proof of \eqref{at_delamination}.} Let $\epsilon\in(0,2^{-10})$ and  $L:=\Sigma\cap \p^*A\cap J_u.$ Since $\Sigma$ is Lipschitz, $L$ is $\cH^{n-1}$-rectifiable.

Let $x_0\in L$ be such that 
\begin{itemize}
 \item[($\rm b_1$)] $\nu_0:=\nu_\Sigma(x_0)$ exist and equals to $\nu_L(x_0);$ 
 
 \item[($\rm b_2$)] $\theta(L,x_0)=\theta(\Sigma,x_0)=\theta(\p^*A,x_0) = 1;$ 
 
 \item[($\rm b_3$)] $\frac{\d\mu_0}{\d\cH^{n-1}\res L}(x_0)$ exists.
 
 \item[($\rm b_4$)]  $\lim\limits_{r\to0} \frac{1}{r^{n-1}} \int_{Q_{r,\nu_0}(x_0)\cap L}\varphi(x_0,\nu_{J_u})\d\cH^{n-1} =\varphi(x_0,\nu_0).$
 
\end{itemize}
By the lipschitzianity of $\Sigma,$ $\cH^{n-1}$-rectifiablity of $\p^*A,$ \cite[Theorem 2.63]{AFP:2000} and Besicovitch differentiation theorem, the set of $x_0\in L$ for which at least one of these conditions fails is $\cH^{n-1}$-negligible.   

Let $r_{1,\epsilon}>0$ be such that \eqref{unif_phi_contiu} holds and $Q_{r_{1,\epsilon},\nu_0}(x_0)\strictlyincluded \Ins{\Omega}.$ Then as in \eqref{mu_k_below_estima}
\begin{align*}
\Big(1+\frac{\epsilon}{\bound_1}\Big)\mu_k(Q_{r,\nu_0}(x_0)) \ge \gamma_k(Q_{r,\nu_0}(x_0)) 
\end{align*}
for any $r\in(0,r_{1,\epsilon}),$ where 
$$
\gamma_k(U):= \int_{U\cap\Omega\cap\p^*A_k} \phi(\nu_{A_k})\d\cH^{n-1} +2\int_{U \cap A_k^{(1)} \cap J_{u_k}} \phi(\nu_{J_{u_k}}) \d\cH^{n-1} + 2 \int_{U\cap\Sigma\cap\p^*A_k\cap J_{u_k}} \phi(\nu_{\Sigma})\d\cH^{n-1}.
$$
Since $\Sigma$ is Lipschitz continuous, by ($\rm b_1$) and ($\rm b_2$) there exists $r_{2,\epsilon}\in(0,r_{1,\epsilon})$ such that 
\begin{itemize}
 \item $\Sigma$ divides the cube $Q_{r_{2,\epsilon},\nu_0}(x_0)$ into two connected components; 
 
 \item  $|\nu_{\Sigma}(x) - \nu_\Sigma(x_0)|<\epsilon$ for any $x\in Q_{r_{2,\epsilon},\nu_0}(x_0)\cap\Sigma;$
 
 \item  $|(x-x_0)\cdot \nu_0| <\epsilon r/2$ for any $r\in (0,r_{2,\epsilon})$ and $x\in Q_{r,\nu_0}(x_0)\cap \Sigma;$ 
 
 \item  $\cH^{n-1}(Q_{r,\nu_0}(x_0)\cap \Sigma) <(1+\epsilon)r^{n-1}$ for all $r\in(0,r_{2,\epsilon}).$
\end{itemize}
Moreover, since $x_0\in \Sigma\cap\p^*A$ and $\theta(L,x_0)=\theta(\p^*A,x_0)=1,$  there exists $r_{3,\epsilon}\in(0,r_{2,\epsilon})$ such that 
\begin{itemize}
 \item $\cH^{n-1}(Q_{r,\nu_0}(x_0)\cap\Sigma\cap\p^*A)>(1-\epsilon)r^{n-1}$ and $\cH^{n-1}(Q_{r,\nu_0}(x_0)\cap\p^*A\setminus \Sigma)<\delta r^{n-1}.$ 
\end{itemize}
Thus, applying Corollary \ref{cor:further_estimates} (b) we find $k_\epsilon''>0$ and $c''>0$ such that 
$$
\gamma_k(Q_{r,\nu_0}(x_0)) \ge 2 \int_{Q_{r,\nu_0}(x_0)\cap L} \phi(\nu_\Sigma)\d\cH^{n-1} - c''\delta r^{n-1}\quad \text{for all $k>k_\epsilon''.$}
$$
Therefore,
$$
\Big(1+\frac{\epsilon}{b_1}\Big)\mu_k(Q_{r,\nu_0}(x_0)) \ge 2 \int_{Q_{r,\nu_0}(x_0)\cap L} \phi(\nu_\Sigma)\d\cH^{n-1} - c''\delta r^{n-1} 
$$
and hence  by ($\rm b_3$) and ($\rm b_4$),
$$
 \frac{\d\mu_0}{\d\cH^{n-1}\res L} (x_0) \ge 2\varphi(x_0,\nu_0).
$$

{\it Proof of \eqref{at_bulk}.} By the nonnegativity of $\mu_k$ and our assumption $u_k=\xi$ on $\Omega\setminus A_k,$
\begin{align}\label{est_belowww}
\mu_k(B_r(x))\ge & \int_{B_r(x)\cap (A_k\cup \substrate)} W(y,\str{u_k} - \bM_0)\d y\nonumber \\
= & \int_{B_r(x)\cap (\Omega\cup \substrate)} W(y,\str{u_k} - \bM_0)\d y - \int_{B_r(x)\cap (\Omega\setminus A_k)} W(y,-\bM_0)\d y.
\end{align}
Since $\mu_k\wk^*\mu_0,$ $\str{u_k}\wk \str{u}$ in $L^2(\Omega\cup\substrate)$ (see \eqref{strain_converges}) and $A_k\to A$ in $L^1(\R^n),$  letting $k\to+\infty$ in \eqref{est_belowww} in any ball $B_r(x)$ with $\mu_0(\p B_r(x))=0$  we have
\begin{align*}
\mu_0(B_r(x))= & \lim\limits_{k\to+\infty} \mu_k(B_r(x))\\
\ge & \int_{B_r(x)\cap (\Omega\cup \substrate)} W(y,\str{u} - \bM_0)\d y - \int_{B_r(x)\cap (\Omega\setminus A)} W(y,-\bM_0)\d y\\
=& \int_{B_r(x)\cap (A\cup \substrate)} W(y,\str{u} - \bM_0)\d y,
\end{align*}
where in the inequality we used the convexity and lower semicontinuity of $W(y,\cdot),$ while in the last equality we used $u=\xi$ in $\Omega\setminus A.$ Now \eqref{at_bulk} follows from the Besicovitch differentiation theorem.
\end{proof}

\begin{remark}
According to the proof of Theorem \ref{teo:lower_semicontinuity}, both $\cS$ and $\cW$ are $\tau_\admissible$-lower semicontinuous in $\admissible.$
\end{remark}

Now we prove bounds \eqref{interior_estimates_CO}-\eqref{Sigma_2estimates_CO}.

\begin{proof}[\textbf{Proof of Proposition \ref{prop:estimate_inner_jump}}]
We prove only (i) in the case $K=Q_{r,\nu}(x_0)\cap \p^*E$ (that is, when $|u_k|\to +\infty$ a.e.\  in $Q_{r,\nu}(x_0)\cap E$). The case $K=Q_{r,\nu}(x_0)\cap J_u$ as well as the assertion (ii) can be handled by similar arguments.
The last inequality in \eqref{interior_estimates_CO} directly follows from \eqref{two_dim_esimates}-\eqref{one_dim_esimates}.
Therefore, we establish only the first estimate.
Without loss of generality, we assume $x_0=0,$ $r=1$ and $\nu=\e_n.$
Note that by (a1),  we have $\Gamma\subset (-\frac12,\frac12)^{n-1}\times(-\frac{\delta}{2},\frac{\delta}{2}),$ and
by (a3) and a priori estimates in Remark \ref{rem:apriori_bounds_seq},
\begin{equation}\label{ebergy_bounded_uni_qaraya00}
M_1:=\sup\limits_{k\ge1}\Big(\int_{\Omega\cup S}|\str{u_k}|^2\d x + 
\cH^{n-1}(J_{u_k})\Big)<+\infty.  
\end{equation}
For any open set $G\subset Q_1$ define
$$
\alpha_k(G):=\int_{G\cap\p^*A_k} \phi(\nu_{A_k})\d\cH^{n-1} + 2\int_{G\cap A_k^{(1)} \cap J_{u_k}} \phi(\nu_{J_{u_k}})\d\cH^{n-1}.
$$

{\it Step 1.} Let 
$$
\Upsilon:=\{\xi\in\S^{n-1}:\,\, |\xi\cdot e_n| \ge 2\delta \}.
$$
Then by (a1) for any $\xi\in\Upsilon$ and $x\in Q_1\cap\Gamma$
$$
|\xi \cdot \nu_\Gamma(x)| \ge |\xi \cdot \e_n| -|\xi\cdot (\nu_\Gamma(x) - \e_n)| >\delta ,
$$
and hence $Q_1\cap \Gamma$ is a graph also in $\xi$-direction, i.e., for any $y\in\Pi_\xi$ the line $\pi_\xi^{-1}(y)$ intersects $Q_1\cap \Gamma$ at most at one point.

{\it Step 2.} Let $D$ be given by Lemma \ref{lem:creating_hole} and let $U\strictlyincluded D$ be any open set such that $U\cap \Gamma\cap K\ne\emptyset.$ 
Let also $(B_k^U,v_k^U)$ be given by Corollary \ref{cor:functions_with_good_cracks} applied with $U,$ $\delta  =\frac{|U|}{k}$ and $(A_k,u_k)$. 
Then for all $k:$
\begin{itemize}
 \item[($\rm a_1$)] $B_k^U\subset A_k,$ $A_k\setminus B_k^U\strictlyincluded U$ and $|A_k\setminus B_k^U| <1/k;$
 
 \item[($\rm a_2$)] $v_k^U = u_k$ in $B_k^U\cup S;$
 
 \item[($\rm a_3$)] $\cH^{n-1}(U\cap [B_k^U]^{(1)}\cap J_{v_k^U})<1/k;$
 
 \item[($a_4$)]  $\alpha_k(U) + |U|/k \ge \Lambda_k(U),$ where 
\end{itemize}
$$
\Lambda_k(U):= \int_{U\cap\p^*B_k^U} \phi(\nu_{B_k^U})\d\cH^{n-1}. 
$$
By ($\rm a_1$), $B_k^U\to A$ in $L^1(\R^n)$  and by ($\rm a_1$), ($\rm a_2$) and also (a6),
\begin{equation}\label{bad_conver_vk_U}
v_k^U\to u\,\,\,\,\text{a.e.\  in $U\setminus E$}\qquad\text{and}\qquad |v_k^U|\to+\infty\,\,\,\, \text{a.e.\  in $U\cap E.$}  
\end{equation}
Moreover, by \eqref{ebergy_bounded_uni_qaraya00} and ($\rm a_2$),
$$
\sup_{k\ge1} \Big(\int_U|\str{v_k^U}|^2\d x + \cH^{n-1}(U\cap J_{v_k^U})\Big) <+\infty.
$$

We claim that 
\begin{align}\label{eq:Step1}
\liminf\limits_{k\to+\infty}  \, \Lambda_k(U) \ge  & \frac{2}{\phi^o(\xi) }\,
\int_{U\cap \Gamma}|\nu_\Gamma\cdot\xi| \d\cH^{n-1} 
 -
 2\bound_2P(A,U) - 2\bound_2\cH^{n-1}(U \cap [\Gamma\setminus \p^*E]).
\end{align}
for $\cH^{n-1}$-a.e.\  $\xi\in \Upsilon.$

To prove \eqref{eq:Step1} we study some properties of one-dimensional slices $[\hat v_k^U]_y^\xi$ of $v_k^U.$ We closely follow the arguments of \cite[pp. 11-13]{ChC:2020_jems}; see also \cite{ChC:2020_arxiv}. Let $k_j:=k_j^U$ be such that  
\begin{equation*}
\liminf\limits_{k\to+\infty} \int_{U\cap J_{v_k^U}} \phi(\nu_{J_{v_k^U}}^{})\d\cH^{n-1} = \lim\limits_{j\to+\infty} \int_{U\cap J_{v_{k_j}^U}} \phi(\nu_{J_{v_{k_j}^U}}^{})\d\cH^{n-1}.
\end{equation*}
Applying \eqref{estimate_I} and \eqref{estimate_II} with $v=v_{k_j}^U$ and \eqref{jump_estimate} with $L =J_{v_{k_j}^U}$ as well as using  \eqref{ebergy_bounded_uni_qaraya00}
we find 
\begin{equation}
\liminf\limits_{j\to+\infty} \int_{\Pi_\xi} \Big[\cH^0(J_{[\hat v_{k_j}^U]_y^\xi}) + \kappa I_{y,\xi}^U(v_{k_j}^U) + \kappa II_{y,\xi}^U(v_{k_j}^U) \Big]\d\cH^{n-1}(y) <+\infty 
\label{set_of_xis} 
\end{equation}
for any $\kappa>0$ and $\cH^{n-1}$-a.e.\  $\xi\in\Upsilon.$  
Moreover, by \cite[Lemma 2.7]{ChC:2020_jems} and \eqref{bad_conver_vk_U},
\begin{equation}\label{inf_strain1098}
|v_k^U\cdot\xi| \to+\infty\quad\text{a.e.\  in $U\cap E$} 
\end{equation}
for $\cH^{n-1}$-a.e.\  $\xi\in\Upsilon.$  
Fix any $\xi\in\Upsilon$ satisfying \eqref{set_of_xis} and \eqref{inf_strain1098} and consider the one-dimensional slices $[\hat v_{k_j}^U]_y^\xi$ and $\hat u_y^\xi.$ 
In view of \eqref{set_of_xis} and  Fatou's lemma, for $\cH^{n-1}$-a.e.\  $y\in \pi_\xi(U)$  
$$
\liminf\limits_{k\to+\infty} \Big[\cH^0(J_{[\hat v_{k_j}^U]_y^\xi}) + \kappa I_{y,\xi}^U(v_{k_j}^U) +\kappa  II_{y,\xi}^U(v_{k_j}^U) \Big] < +\infty.
$$
Thus, for $\cH^{n-1}$-a.e.\  $y\in \pi_\xi(U)$ there exists a subsequence $\{k_j^y\}\subset\{k_j\}$ (depending also on $\kappa>0$) such that 
\begin{multline}
\liminf\limits_{j\to+\infty} \Big[\cH^0(J_{[\hat v_{k_j}^U]_y^\xi}) + \kappa I_{y,\xi}^U(v_{k_j}^U) + \kappa II_{y,\xi}^U(v_{k_j}^U) \Big] \\
= \lim\limits_{j\to+\infty} \Big[\cH^0(J_{[\hat v_{k_j^y}^U]_y^\xi}) + \kappa I_{y,\xi}^U(v_{k_j^y}^U) + \kappa II_{y,\xi}^U(v_{k_j^y}^U) \Big], \label{mana_liminf_ana_liminf} 
\end{multline}
and by \eqref{bad_conver_vk_U} and \eqref{inf_strain1098},
\begin{equation} \label{conv_one_dim_slices}
[\hat v_{k_j^y}^U]_y^\xi \to \hat u_y^\xi\quad\text{$\cL^1$-a.e.\  in $[U\setminus E]_y^\xi$}\quad \text{and}\quad
\Big|[\hat v_{k_j^y}^U]_y^\xi\Big| \to +\infty\quad\text{$\cL^1$-a.e.\  in $[U\cap E]_y^\xi.$}
\end{equation}
For $\tau(t)=\arctan(t),$ set $f_j:=\tau\circ [\hat v_{k_j^y}^U]_y^\xi.$ Then $f_j\in SBV_\loc^2(U_y^\xi)$ and $J_{[\hat v_{k_j^y}^U]_y^\xi}=J_{f_j}.$ By \eqref{mana_liminf_ana_liminf}, \eqref{conv_one_dim_slices}  and \cite[Proposition 4.2]{Ambrosio:1989}, we find a not relabelled  subsequence $\{v_{k_j^y}^U\}$ such that
\begin{equation*}
f_j\to f_0 \quad \text{$\cL^1$-a.e.\  in $U_y^\xi$ as $j\to+\infty.$}
\end{equation*}
 By \eqref{conv_one_dim_slices},
$$
\begin{cases}
f_0=\tau\circ \hat u_y^U  &  \text{in $[U\setminus E]_y^\xi,$}\\
|f_0| = \pi/2  &  \text{in $[U\cap E]_y^\xi.$}
\end{cases}
$$
Moreover, by \cite[Proposition 4.2]{Ambrosio:1989},
\begin{equation}
\liminf\limits_{j\to+\infty} \cH^0\Big(J_{[\hat v_{k_j^y}^U]_y^\xi}\Big) =
\liminf\limits_{j\to+\infty} \cH^0(J_{f_j}) \ge  \cH^0 (J_{f_0}).
\label{lsc_0_dimenaina}
\end{equation}
Thus, $\cH^0(U_y^\xi\cap J_{f_0} )<+\infty$ and hence  $[U\cap E]_y^\xi$ consists of finitely many segments in each of which either $f_0\equiv\pi/2$  or $f_0\equiv-\pi/2.$

By \eqref{mana_liminf_ana_liminf},
$\cH^0(J_{f_j})$ is uniformly bounded and hence, there exists a further not relabelled subsequence and $N_y \in\N_0$ such that 
\begin{equation*}
\cH^0(J_{f_j}) = N_y\quad\text{and}\quad J_{f_j} = \{t_j^1,\ldots,t_j^{N_y}\}\subset U_y^\xi\quad\text{for all $j.$} 
\end{equation*} 
Then points of $J_{f_j}$ converges to $M_y\le N_y$ points $t^1<\ldots<t^{M_y}.$ 
Since $II_{y,\xi}^U(v_{k_j^y}^U)$ is uniformly bounded, the precise representatives of $f_j$ uniformly bounded in $W_\loc^{1,1}(t^l,t^{l+1})$ so that $f_j\to f_0$ locally uniformly in $(t^l,t^{l+1})$ and $J_{f_0}\subset \{t^1,\ldots,t^{M_y}\}.$ Repeating the arguments of \cite[Section 1]{ChC:2020_jems} we can show that $t^1:=U_y^\xi\cap [\p^* E]_y^\xi \in J_{f_0}.$ 

Let us estimate the $\cH^{n-1}$-measures of the sets 
\begin{align*}
Y_0:= & \{y\in \Pi_\xi\cap \pi_\xi(U\cap K):\,\, N_y=0\},\\
Y_1:= & \{y\in \Pi_\xi\cap \pi_\xi(U\cap K):\,\, N_y=1\},\\
Y_2:= & \{y\in \Pi_\xi\cap \pi_\xi(U\cap K):\,\, N_y\ge2\}.
\end{align*}
By \eqref{lsc_0_dimenaina}, $\cH^0(J_{f_0})=0$ for any $y\in Y_0.$ Hence, $U\cap \pi_\xi^{-1}(y)\cap (\p^*E \cup J_u)=\emptyset$ and therefore,
$
Y_0\subset \pi_\xi(U\cap \Gamma\setminus \p^*E).
$
Then by the $1$-Lipschitz continuity of the projection $\pi_\xi,$
\begin{equation}\label{est_Y0_9142}
\cH^{n-1}(Y_0) \le \cH^{n-1}(\pi_\xi(U \cap [\Gamma\setminus \p^*E])) 
\le \cH^{n-1}(U \cap [\Gamma\setminus \p^*E]).  
\end{equation}
Now consider any $y\in Y_1.$ By definition, $\pi_\xi^{-1}(y)$ intersects $U\cap J_{v_{k_j^y}^U}$ just once, and thus by the construction of  $(B_k^U,v_k^U)$ (see the proof of Corollary \ref{cor:functions_with_good_cracks}), either
$
y\in \pi_\xi\big(U\cap [B_{k_j^y}^U]^{(1)}\cap J_{u_{k_j^y}}\cap J_{v_{k_j^y}^U}\big)
$
or 
$
y\in \pi_\xi(U\setminus B_{k_j^y}^U).
$
If $y\in \pi_\xi(U\setminus B_{k_j^y}^U),$ then $t_j^1$ divides the line $U\cap \pi_\xi^{-1}(y)$ into two parts: one is a subset of $U\cap B_{k_j^y}^U$ and the other is that of $U\setminus B_{k_j^y}^U.$ Since $B_{k_j^y}^U \to A$ and $t^1=U_y^\xi\cap [\p^* E]_y^\xi \in J_{f_0},$ it follows that $t^1\in \p^*A$ and divides $U\cap \pi_\xi^{-1}(y)$ into two parts one belonging to $U\cap A$ the other to $U\setminus A.$ In particular, $y\in \pi_\xi(U\cap \p^*A).$
Hence,
$$
y\in\Big[U\cap [B_{k_j^y}^U]^{(1)}\cap J_{u_{k_j^y}}\cap J_{v_{k_j^y}^U})\Big]_y^\xi \cup \Big[U\cap \p^*A\Big]_y^\xi \quad \text{for all $j.$ }
$$
Thus,
\begin{align*}
\cH^{n-1}(Y_1) =  & \int_{Y_1} \cH^0\left(
\bigcap_j\Big( \Big[U\cap [B_{k_j^y}^U]^{(1)}\cap J_{u_{k_j^y}}\cap J_{v_{k_j^y}^U})\Big]_y^\xi \cup \Big[U\setminus B_{k_j^y}^U\Big]_y^\xi \Big)\right)\, \d\cH^{n-1}(y)\\
\le &
\int_{Y_1} \lim \limits_{j\to+\infty} \cH^0\Big(\Big[U\cap [B_{k_j^y}^U]^{(1)}\cap J_{u_{k_j^y}}\cap J_{v_{k_j^y}^U})\Big]_y^\xi\Big) \,\d\cH^{n-1}(y)\\
& +
\int_{Y_1} \cH^0\Big(\Big[U\cap\p^*A\Big]_y^\xi \Big)\, \d\cH^{n-1}(y).
\end{align*}
By the choice of $\{k_j^y\}$, the Fatou's lemma, the second equality in \eqref{jump_estimate} and ($\rm a_3$),
\begin{multline*}
\int_{Y_1} \lim\limits_{j\to+\infty} \cH^0\Big(\Big[U\cap [B_{k_j^y}^U]^{(1)}\cap J_{u_{k_j^y}}\cap J_{v_{k_j^y}^U})\Big]_y^\xi\Big) \,\d\cH^{n-1}(y)\\
\le  \liminf\limits_{k\to+\infty} \cH^{n-1}\Big(U\cap [B_k^U]^{(1)} \cap J_{u_k}\cap J_{v_k^U})\Big)=0.
\end{multline*}
Similarly, 
\begin{align*}
\int_{Y_1} \cH^0\Big(\Big[U\cap\p^*A\Big]_y^\xi \Big)\, \d\cH^{n-1}(y) 
\le & P(A,U).
\end{align*}
Hence,
\begin{equation}\label{est_Y1_442}
\cH^{n-1}(Y_1) \le P(A,U).
\end{equation}

Now using $\Pi_\xi\cap \pi_\xi(U) = Y_0\cup Y_1\cup Y_2,$ from \eqref{est_Y0_9142} and \eqref{est_Y1_442} we obtain
$$
\cH^{n-1}([\Pi_\xi\cap \pi_\xi(U)]\setminus Y_2) \le P(A,U) + \cH^{n-1}(U \cap [\Gamma\setminus \p^*E]).
$$
Moreover, let 
$$
X:=\{y\in \Pi_\xi\cap \pi_\xi(U):\,\,\text{$\pi_\xi^{-1}(y)\cap \Gamma\cap \p^*E$ is a singleton}\}.
$$
Then as above 
\begin{align*}
\cH^{n-1}(Y_2\setminus X) \le & \cH^{n-1}([\Pi_\xi\cap \pi_\xi(U)]\setminus X )\le \cH^{n-1}([U\cap (\Gamma\cup \p^*A)]\setminus \p^*E) \\
\le &
\cH^{n-1}(U \cap [\Gamma\setminus \p^*E]) + P(A,U),
\end{align*}
and therefore, 
\begin{equation}\label{yahudlar7465}
\cH^{n-1}([\Pi_\xi\cap \pi_\xi(U)] \setminus [Y_2\cap X ]) \le 2P(A,U) + 2\cH^{n-1}(U \cap [\Gamma\setminus \p^*E]).
\end{equation}
By the definition of $X $ and $Y_2,$ for any $y\in Y_2\cap X$ we have $\cH^0(J_{f_0})=1$ and $N_y\ge2.$ Therefore,
we can improve \eqref{lsc_0_dimenaina} as 
\begin{equation*}
\lim\limits_{j\to+\infty} \cH^0\Big(J_{[\hat v_{k_j^y}^U]_y^\xi}\Big) \ge  2= 2\cH^0([U\cap \Gamma]_y^\xi).
\end{equation*}
For such $y,$ from  \eqref{mana_liminf_ana_liminf}  we get
$$
\liminf\limits_{j\to+\infty} \Big[\cH^0(J_{[\hat v_{k_j}^U]_y^\xi}) + \kappa I_{y,\xi}^U(v_{k_j}^U) + \kappa II_{y,\xi}^U(v_{k_j}^U) \Big] \ge 
2\cH^0([U\cap \Gamma]_y^\xi)
$$
Now integrating both sides  over $X \cap Y_2$ and using \eqref{set_of_xis} and the Fatou's lemma we obtain
$$ 
\liminf\limits_{k\to+\infty} \int_{\Pi_\xi} \Big[\cH^0(J_{[\hat v_k^U]_y^\xi}) + \kappa I_{y,\xi}^U(v_k^U) + \kappa II_{y,\xi}^U(v_{k}^U) \Big]\d\cH^{n-1}(y) \ge 2\int_{X \cap Y_2}\cH^0([U\cap \Gamma]_y^\xi)\d\cH^{n-1}(y).
$$ 
By the definition of $\Upsilon,$ $\cH^0([U\cap \Gamma]_y^\xi)=1$ for all $y\in\Pi_\xi\cap  \pi_\xi(U).$ Thus, by \eqref{yahudlar7465},
\begin{align*}
\int_{X\cap Y_2}\cH^0([U\cap \Gamma]_y^\xi)\d\cH^{n-1}(y)
\ge  \int_{\Pi_\xi\cap \pi_\xi(U)} \cH^0([U\cap \Gamma]_y^\xi)\d\cH^{n-1}(y)
  - P(A,U)- \cH^{n-1}(U \cap [\Gamma\setminus \p^*E]).
\end{align*}
Hence, 
\begin{multline*}
\liminf\limits_{k\to+\infty} \int_{\Pi_\xi}  \Big[\cH^0(J_{[\hat v_k^U]_y^\xi}) + \kappa I_{y,\xi}^U(v_k^U) + \kappa II_{y,\xi}^U(v_{k}^U) \Big]\d\cH^{n-1}(y)\\
\ge 2 \int_{\Pi_\xi\cap \pi_\xi(U)} \cH^0([U\cap \Gamma]_y^\xi)\d\cH^{n-1}(y) 
- 2P(A,U) - 2\cH^{n-1}(U \cap [\Gamma\setminus \p^*E]).
\end{multline*} 
This inequality, \eqref{jump_estimate}, \eqref{estimate_I}, \eqref{estimate_II} as well as \eqref{ebergy_bounded_uni_qaraya00} yield
\begin{multline}\label{olimlar465}
\liminf\limits_{k\to+\infty} \int_{U\cap J_{v_k^U}} |\nu_{J_{v_k^U}}\cdot \xi| \d\cH^{n-1}   +  (M_1+|U|)\kappa  \\
 \ge  2 \int_{U\cap \Gamma}|\nu_\Gamma\cdot\xi| \d\cH^{n-1}
-2P(A,U)- 2\cH^{n-1}(U \cap [\Gamma\setminus \p^*E]) 
\end{multline} 
Let $\phi^o$ be the dual norm to $\phi,$ i.e., 
$$
\phi^o(\xi) = \sup\limits_{\phi(\nu)=1} |\xi\cdot\nu|.
$$
Then 
$
|\xi\cdot \nu| \le \phi^o(\xi)\phi(\nu)
$
and hence, from  \eqref{olimlar465}  and the arbitrariness of $\kappa$ we get
\begin{multline}\label{istatat898}
\phi^o(\xi) \liminf\limits_{k\to+\infty} \int_{U\cap J_{v_k^U}} \phi(\nu_{J_{v_k^U}})\d\cH^{n-1} \\
\ge 2
\int_{U\cap \Gamma}|\nu_\Gamma\cdot\xi| \d\cH^{n-1}
  -2P(A,U)- 2\cH^{n-1}(U \cap [\Gamma\setminus \p^*E]).
\end{multline}
Now using $\phi^o(\xi)\ge 1/\bound_2,$ from \eqref{istatat898} we get
\eqref{eq:Step1}.

{\it Step 3.} Now we prove \eqref{interior_estimates_CO}.  

{\it Substep 3.1.} Let 
$$
\S_{\phi^o}^{n-1}:=\{\xi\in\R^n:\,\, \phi^o(\xi)=1\}.
$$
Since $\S_{\phi^o}^{n-1}$ is compact, 
$$
\phi(\eta) =\max\limits_{i\ge1}  \eta\cdot \xi_i 
$$
for any countable set 
$\{\xi_j\}_j\subset \S_{\phi^o}^{n-1}$ dense in $\S_{\phi^o}^{n-1}.$ 

Fix any such dense set $\{\xi_j\}_j\subset \S_{\phi^o}^{n-1}$  that if $\xi=\xi_j/|\xi_j|\in\Upsilon$, then \eqref{set_of_xis} and \eqref{inf_strain1098} hold with $\xi.$ By \cite[Lemma 6]{DBD:1983}, there exists a finite family  $U_1,\ldots,U_m$ of disjoint open set compactly contained in $D$ such that
\begin{equation}\label{asdgar987}
2\int_{D\cap \Gamma} \phi(\nu_{\Gamma}) \d\cH^{n-1}  \le   \sum\limits_{j=1}^m  2\int_{U_j\cap \Gamma} |\nu_{\Gamma}\cdot \xi_j| \d\cH^{n-1} + \delta . 
\end{equation}
Recalling the definition of $(B_k^U,u_k^U)$ from Step 2, let us define 
$$
B_k=\bigcap\limits_{j=1}^m B_k^{U_j}\quad\text{and}\quad v_k:=u_k\chi_{B_k\cup S}.
$$
Then by ($\rm a_2$), $B_k\subset A_k,$ $A_k\setminus B_k\strictlyincluded D$ and
 $$
 |A_k\setminus B_k| \le \sum\limits_{j=1}^m |U_j\cap (A_k\setminus B_k^{U_j})| \le \sum\limits_{j=1}^m \frac{|U_j|}{k} \le \frac{|D|}{k}.
 $$ 
Let $\Lambda_k(D)$ be defined as in ($\rm a_4$) of Step 2 with $(B_k,v_k)$ in place of  $(B_k^U,v_k^U)$.  Then by the definition of $(B_k,v_k)$  and $\alpha_k(D)$ as well as by ($a_4$),
\begin{align*}
\alpha_k(D) - \Lambda_k(D) = \sum\limits_{j=1}^m\big( \alpha_k(U_j) - \Lambda_k(U_j) \big) \ge  -\sum\limits_{j=1}^m \frac{|U_j|}{k} \ge- \frac{|D|}{k}. 
\end{align*}
Thus, 
\begin{equation}\label{lower_estimate_for_alpha}
\alpha_k(D) \ge \Lambda_k(D) - \frac{|D|}{k}. 
\end{equation}

{\it Substep 3.2.} Now we estimate $\Lambda_k(D)$ from below.
 Note that if $\xi_j/|\xi_j| \in \Upsilon,$ then since $\phi^o(\xi_j)=1,$ by \eqref{eq:Step1},
\begin{align}\label{is0r98}
2 \int_{U_j\cap \Gamma}|\nu_\Gamma\cdot\xi_j| &\d\cH^{n-1} \le   \liminf\limits_{k\to+\infty} \Lambda_k(U_j) 
+
2\bound_2P(A,U_j)+2\bound_2\cH^{n-1}(U_j \cap [\Gamma\setminus \p^*E]).
\end{align}

Now assume that $\xi:=\xi_j/|\xi_j|\notin\Upsilon.$ Then by the definition of $\Upsilon$ and (a3),
$$
|\nu_\Gamma(x)\cdot\xi| \le |(\nu_\Gamma(x) - \e_n)\cdot\xi| + |\e_n\cdot\xi| <3\delta  
$$
for any $x\in U_j\cap \Gamma.$ Thus,
\begin{equation}\label{ahmoqlar987}
2\int_{U_j\cap \Gamma} |\nu_{\Gamma}\cdot \xi_j| \d\cH^{n-1} \le 6\delta  \cH^{n-1}(U_j\cap\Gamma). 
\end{equation}
Now by \eqref{asdgar987}, \eqref{is0r98} and \eqref{ahmoqlar987},
\begin{multline*}
2\int_{D\cap \Gamma} \phi(\nu_{\Gamma}) \d\cH^{n-1}  \le  \delta  +
\sum\limits_{j=1,\,j\in \Upsilon}^m \liminf\limits_{k\to+\infty}\,\Lambda_k(U_j) + 
6\delta  \sum\limits_{j=1,\,j\notin\Upsilon}^m \cH^{n-1}(U_j\cap\Gamma) \\
 + 2\bound_2\sum\limits_{j=1}^n \Big[P(A,U_j)+2\cH^{n-1}(U_j \cap [\Gamma\setminus \p^*E])\Big].
\end{multline*}
Since set function $Q\mapsto \Lambda_k(Q)$ is additive and non-increasing and the family $\{U_j\}$ is pairwise disjoint,
$$
\sum\limits_{j=1,\,j\in \Upsilon}^m \liminf\limits_{k\to+\infty}\,\Lambda_k(U_j) \le \liminf\limits_{k\to+\infty} 
\Lambda_k(\cup_j U_j) \le \liminf\limits_{k\to+\infty} 
\Lambda_k(D).
$$
Moreover, by (a2),
$$
\sum\limits_{j=1,\,j\notin\Upsilon}^m \cH^{n-1}(U_j\cap\Gamma) \le \cH^{n-1}(Q_1\cap \Gamma) <1+\delta , 
$$
and by (a2), (a5), (a7.2)  and (a7.3),
\begin{align*}
\sum\limits_{j=1}^n \Big[ P(A,U_j)+\cH^{n-1}(Q_1 \cap [\Gamma\setminus \p^*E]) \Big]
\le 
P(A,Q_1) + \cH^{n-1}(Q_1 \cap \Gamma) - \cH^{n-1}(Q_1\cap \Gamma\cap\p^*E)
\le 6\delta.
\end{align*}
Then
\begin{align*}
2\int_{D\cap \Gamma} \phi(\nu_{\Gamma}) \d\cH^{n-1}  \le & \delta  +  \liminf\limits_{k\to+\infty} 
\Lambda_k(D) + 6\delta (1+\delta ) + 6\bound_2\delta, 
\end{align*}
and hence,
\begin{align}\label{Step1122}
\liminf\limits_{k\to+\infty} 
\Lambda_k(D) \ge 2\int_{D\cap \Gamma} \phi(\nu_{\Gamma}) \d\cH^{n-1} - c_0\delta , 
\end{align}
where 
$$
c_0:=13+  6\bound_2
$$
depends only on $\bound_2.$
 
{\it Substep 3.3.}  In view of  \eqref{lower_estimate_for_alpha} and \eqref{Step1122}, there exist $k_0:=k_0(\delta,\bound_2)>0$ such that
\begin{equation}\label{shahahah}
\Lambda_k(D) \ge 2\int_{D\cap \Gamma} \phi(\nu_{\Gamma}) \d\cH^{n-1} - 2c_0\delta  \quad\text{for all $k>k_0.$}
\end{equation}
 Since $|D|<|Q| = 1,$ one has $|D|/k < c_0\delta$ provided $k>\frac{1}{c_0\delta}.$ Let
$$
k_\delta':=\max\Big\{k_0,\frac{1}{c_0\delta}\Big\}.
$$
Observe that
\begin{align*}
\int_{D\cap \Gamma} \phi(\nu_{\Gamma}) \d\cH^{n-1} 
\ge  \int_{D\cap \Gamma\cap \p^*E} \phi(\nu_{\Gamma}) \d\cH^{n-1}.
\end{align*}
Moreover, by (a7.3),
$$
\int_{D\cap \p^*E\setminus \Gamma} \phi(\nu_E)\d\cH^{n-1} \le \bound_2\delta
$$
and by Lemma \ref{lem:creating_hole} (ii),
$$
\cH^{n-1}(Q_1\cap \p^*E \setminus D) < 2\delta,
$$
and therefore,
$$
\int_{D\cap \Gamma} \phi(\nu_{\Gamma}) \d\cH^{n-1} \ge 
\int_{Q_1\cap\p^*E} \phi(\nu_E)\d\cH^{n-1} - 3\bound_2\delta\qquad \text{for all $k>k_\delta'$}.
$$
Combining these estimates with \eqref{lower_estimate_for_alpha} and \eqref{shahahah} we deduce
$$
\alpha_k(D) \ge 2\int_K \phi(\nu_K)\d\cH^{n-1}  - (2c_0+6\bound_2)\delta.
$$
Hence,
$
c':=c_{\bound_2}'=(2c_0+6\bound_2)
$
satisfies the assertion.
\end{proof}

\subsection{Lower semicontinuity of $\cF_p$ and $\cF_{\rm Dir}$}

We conclude this section by showing that the functionals $\cF_p$ and $\cF_{\rm Dir}$ in Theorems \ref{teo:elastic_plow} and \ref{teo:dirichlet_plow}, respectively, are lower semicontinuous with respect to the $\tau$-convergence defined in \eqref{tau_convergence_CACA}. 
Indeed, the proof of the $\tau$-lower semicontinuity of $\cS$ in $\admissible_p$ and $\admissible_{\rm Dir}$ is exactly the same as the $\tau_\admissible$-lower semicontinuity of  $\cS$ in $\admissible$ (see the proof of Theorem \ref{teo:lower_semicontinuity}). To prove the $\tau$-lower semicontinuity of $\cW_p$ and $\cW_{\rm Dir}$ we notice that according to the proof of the density estimate \eqref{at_bulk}, we only need the convexity of $W_p(x,\cdot)$ and the weak convergence of $\str{u_k}$ to $\str{u}$ in $L^p(\Ins{\Omega});$ the first condition is already stated in the assumption (a1) of $W_p,$ while the second condition follows from the lower bound in (a2) and the compactness  result \cite[Theorem 1.1]{ChC:2020_jems}.

\section{Compactness in $\admissible$}\label{sec:compactoser}

 In this section we prove Theorem \ref{teo:compactness}. Note that if $\{(A_k,u_k)\}$ is an energy-equibounded sequence, then by a priori estimates (see Remark \ref{rem:apriori_bounds_seq}), we can find a set of finite perimeter $A\subset\Omega$ such that, up to a subsequence,  $A_k\to A$ in $L^1(\R^n).$ Moreover, since each connected component $S_i$ of $S$ is Lipschitz, the convergence of $u_k$ in $S_i$ can be obtained  by adding rigid displacements in $S_i.$ However, since the rigid displacements for $S_i$ may differ from those for $S_j,$ $j\ne i,$ we need to create extra jumps for the  resulting displacement field. Hence, as in \cite{HP:2020_arma} we need to partition $A_k$ to compensate those jumps. The following proposition provides such a partition up to some error.

 \begin{proposition}\label{prop:pass_to_good_seq_compacte}
Let $(A_k,u_k),(A,u)\in\admissible$ be admissible configurations, $S^i$ for $i\in\{1,\ldots,m\}$ be a nonempty union of some connected components of $S$ such that $S^i\cap S^j =  \emptyset$ and $S = \bigcup_{i=1}^m S^i,$ $\{a_k^1\},\ldots,\{a_k^m\}$ be sequences of rigid displacements, $u^1,\ldots,u^m\in GSBD^2(\Ins{\Omega})$ and  
$F^1,\ldots,F^m\subset A$ be pairwise disjoint sets of finite perimeter. Assume that 
\begin{itemize}[left=12pt]
\item  $\sup_k\cF(A_k,u_k)<+\infty$ and $A_k\to A$ in $L^1(\R^n);$
 
\item for any $i\in\{1,\ldots,m\}$  one has $u_k - a_k^i \to u^i$ a.e.\ in $S^i \cup F^i$  and 
$|u_k - a_k^i |\to +\infty$ a.e.\ $(S\setminus S^i) \cup (A\setminus F^i).$
\end{itemize}
Then for any $\delta\in(0,\frac18\,\min_{i\ne j}\{1,\dist(S^i,S^j)\})$ there exist a (not relabelled) subsequence $\{(A_k,u_k)\},$ $k_\delta>0,$ $s_\delta\in(0,\delta)$ 
and a sequence $\{G_k^\delta\} \subset BV(\Omega;\{0,1\})$ 
such that
\begin{subequations}
\begin{align}
& \cH^{n-1}([A_k\setminus A]^{(1)}\cap \{\dist(\cdot, S) = s_\delta\}) < c^*\delta, \label{good_cut_sk}\\
& \cH^{n-1}(\{\dist(\cdot, S) < s_\delta\} \cap \p^*A)<c^*\delta, \label{good_cut_skper}\\
& |G_k^\delta| < c^*\sqrt{\delta} \sum_{0\le i\le m} P(F^i), \label{vol_est_Gk}\\
& P(G_k^\delta) \le c^*\sum_{0\le i\le m} P(F^i), \label{per_est_Gk}
\end{align}
\end{subequations}
and the sequence $\{(B_k^\delta,v_k^\delta)\},$ defined as 
\begin{equation}\label{def_BBkdelta}
B_k^\delta:=A_k\setminus G_k^\delta 
\end{equation}
and
\begin{equation}\label{def_vvkdelta}
v_k^\delta:=
\begin{cases}
u_k - a_k^i & \text{in $S^i \cup [F^i\setminus G_k^\delta] \cup [R_\delta^i \cap (B_k^\delta\setminus A)]$ for $i=1,\ldots,m,$} \\
u_0 & \text{in $B_k^\delta \cap F^0,$} \\
\xi &\text{in $(\Omega\setminus B_k^\delta)\cup (B_k^\delta \setminus [A\cup \bigcup_{i=1}^m R_\delta^i]),$}
\end{cases}
\end{equation}
where $\xi\in(0,1)^n,$
$$
R_\delta^i:=\{x\in\Omega:\,\, \dist(x, S^i)< s_\delta\},\qquad F^0:=A\setminus \bigcup_{i=1}^m F^i,
$$
satisfies 
\begin{equation}\label{error_estimate_Bv}
\cS(A_k,u_k) \ge  \cS(B_k^\delta,v_k^\delta)- c^*
  \sqrt{\delta} \Big[1+ P(A_k) + \cH^{n-1}(J_{u_k}) +  \sum\limits_{i=0}^m \cH^{n-1}(\p^*F^i)\Big]
\end{equation}
for all $k>k_\delta.$ Here constant $c^*>0$ depends only on $n,$ $\bound_1$ and $\bound_2.$
\end{proposition}

\begin{figure}[htp!]
\includegraphics[width=0.96\textwidth]{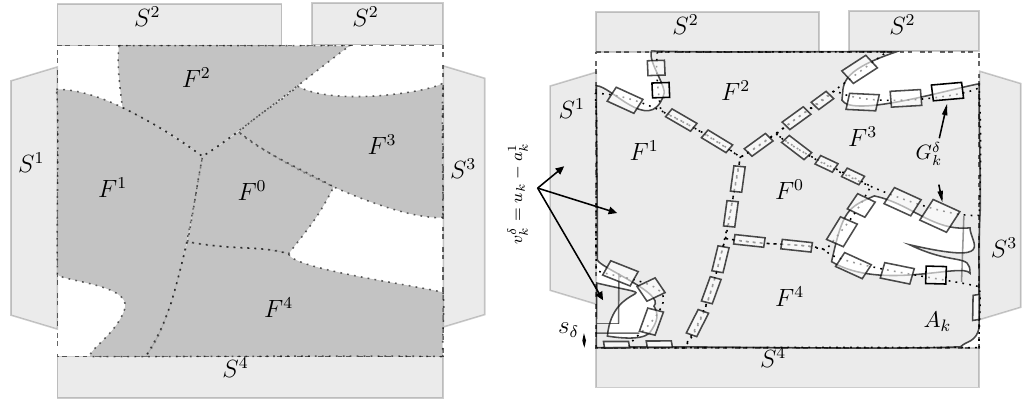}
\caption{{\small The partition of $A = \bigcup_{i\ge0} F^i$ and the construction of $B_k^\delta:=A_k\setminus G_k^\delta$ in Proposition \ref{prop:pass_to_good_seq_compacte}. The set $G_{\delta,k}$ is a finite union of holes along the boundaries $F^i\cup \bigcup_{j\ne i} S^j$ in which $u_k - a_k^i$ converges. Note that the sets $\{F^i\setminus G_k^\delta\}_{i=0}^m$ partition  $B_k^\delta.$ Since $F^0$ is a ``hanging'' component of $A,$ i.e., not linked to the substrate, it is reasonable to assume that the elastic energy in $F^0$ is $0.$ Then we define the displacement fields $v_k^\delta$ as follows: in $S^i\cup (F^i\setminus G_k^\delta)$ for $i=1,\ldots,m$ we set $v_k^\delta:=u_k - a_k^i$ and in $F^0\setminus G_k^\delta$ we write $v_k^\delta:=u_0.$ Finally, since $A_k\setminus A$ may  present large trace portions along $\p S$ on which $v_k^\delta$ forms a jump, we need to change the values of $v_k^\delta$ in $R_\delta^i\setminus A$ near $S^i.$}}
\label{fig:construction_Gkdelta}
\end{figure}

We postpone the proof of Proposition \ref{prop:pass_to_good_seq_compacte} after the proof of Theorem \ref{teo:compactness}.

\begin{proof}[\textbf{Proof of Theorem \ref{teo:compactness}}]
Since $S$ is Lipschitz open set with finitely many connected components, applying the Poincar\'e-Korn inequality and the Rellich-Kondrachov compactness theorem we find a not relabelled subsequence $\{(A_k,u_k)\},$ a partition $\{S^i\}_{i=1}^m$ of $\substrate$ and $m$ sequences $\{a_k^1\},\ldots,\{a_k^m\}$ of rigid displacements such that 
\begin{itemize}
 \item[($\rm a_1$)] each $S^i$ is the union of some connected components of $\substrate$ and $S=\bigcup_{i=1}^m S^i;$
 
 \item[($\rm a_2$)] for each $i\in\{1,\ldots,m\}$ there exists $w^i\in H^1(S^i)$ such that $u_k - a_k^i$ converges to $w^i$ weakly in $H^1(S^i)$ and a.e.\ in $S^i;$
 
 \item[($\rm a_3$)] if $i\ne j,$  then $|a_k^i - a_k^j|\to+\infty$ a.e.\  in $\R^n.$
\end{itemize}
We may also assume $A_k\to A$ in $L^1(\R^n)$ for some $A\in BV(\Omega;\{0,1\}).$ 
Since $\str{v} =\str{(v + a)}$ for any rigid displacement $a$, by Remark \ref{rem:apriori_bounds_seq} we have
$$
\sup_{k\ge1} \Big(P(A_k) + \cH^{n-1}(J_{(u_k - a_k^i)\chi_{A_k\cup S}}) + \int_{A_k \cup S} |\str{(u_k-a_k^i)}|^2\d x\Big)<+\infty 
$$
for any $i.$ Hence, by \cite[Theorem 1.1]{ChC:2020_jems} there exist a not relabelled subsequence $\{(A_k,u_k)\}$ such that for each $i$ the set 
$$
F_i:=\{x\in \Omega:\,\,\limsup\limits_{k\to+\infty}|(u_k(x) - a_k^i(x))\chi_{A_k}(x)|=+\infty \}
$$
has finite perimeter and there exists a function $u^i\in GSBD^2(\Ins{\Omega})$ such that 
\begin{align*}
u_k-a_k^i \to u^i \quad\text{a.e.\  in $S^i\cup F^i,$} 
\end{align*}
where 
$$
F^i:=A\setminus F_i.
$$
By assumption ($\rm a_3$), the sets $F^1,\ldots,F^m$ are pairwise disjoint
(see Figure \ref{fig:construction_Gkdelta}).

Let $\delta_0 := 2^{-10}\min_{i\ne j}\{1,\dist(S^i,S^j)\}$ and consider any sequence $\delta_l\searrow0$ with $\delta_1< \delta_0.$  By Proposition \ref{prop:pass_to_good_seq_compacte}, for any $l\ge1$ there exists a subsequence $\{(A_{k,l},u_{k,l})\}_k \subset \{(A_{k,l-1},u_{k,l-1})\}_k,$ real numbers $k_{\delta_l}>0$  and $s_{\delta_l}\in (0,\delta_l)$ and a sequence $\{G_k^{\delta_l}\}_k$ of sets of finite perimeter satisfying \eqref{good_cut_sk}-\eqref{per_est_Gk} with $\delta = \delta_l$ such that the sequence $\{(B_k^{\delta_l}, v_k^{\delta_l})\}_k$ defined as \eqref{def_BBkdelta}-\eqref{def_vvkdelta} satisfies
\begin{align}
\cS(A_{k,l},u_{k,l}) \ge  \cS(B_k^{\delta_l},v_k^{\delta_l}) 
- c^*
 \sqrt{\delta_l} \Big[1+ P(A_{k,l}) + \cH^{n-1}(J_{u_{k,l}}) +  \sum\limits_{i=0}^m \cH^{n-1}(\p^*F^i)\Big] \label{surface_bahooo}
\end{align}
for all $k>k_{\delta_l}.$ Here we set $(A_{k,0},u_{k,0}) = (A_k,u_k).$ By \eqref{per_est_Gk}, we may also assume that $G_k^{\delta_l} \to G^{\delta_l}$ in $L^1(\R^n)$ as $k\to+\infty,$ and therefore, $B_k^{\delta_l} \to A\setminus G^{\delta_l}.$
Moreover, setting $v_k^{\delta_l} = \xi$ in $\Omega\setminus B_k^{\delta_l}$ and $B_k^{\delta_l}\setminus [\cup_i R_{\delta_l}^i\cup A]$ for some $\xi\in (0,1)^n\setminus \Xi_{\{B_k^{\delta_l},u_k^{\delta_l}\}_{k,l}}$ (see Remark \ref{rem:ext_u_out_A}),  by the choice of $a_k^i,$ we have
$
v_k^{\delta_l} \to v^{\delta_l}
$
a.e.\  in $\Omega\cup S,$ where 
$$
v^{\delta_l}: = \sum\limits_{i=1}^m u^i\chi_{S^i\cup (F^i\setminus G^{\delta_l})} + u_0\chi_{F^0\setminus G^{1/l}} + \xi \chi_{(\Omega\setminus A)\cup G^{1/l}}.
$$
By \eqref{vol_est_Gk}-\eqref{per_est_Gk},
$$
|G^{\delta_l}| \le c^*\sqrt{\delta_l}  \sum_{i=0}^m P(F^i),\qquad 
P(G^{\delta_l}) \le c^*\sum_{i=0}^m P(F^i),
$$
and hence $G^{\delta_l} \to \emptyset$ in $L^1(\R^n)$ as $l\to+\infty.$ Therefore, $v^{\delta_l} \to u$ a.e.\  in $\Omega\cup S$ as $l\to+\infty,$ where
$$
u: = \sum\limits_{i=1}^m u^i\chi_{S^i\cup F^i} + u_0\chi_{F^0} + \xi \chi_{\Omega\setminus A}.
$$
By the nonnegativity and invariance w.r.t. rigid displacements of the elastic energy density, we have also
\begin{equation}
 \cW(A_{k,l},u_{k,l}) \ge \cW(B_k^{\delta_l},v_k^{\delta_l}). 
\label{elastic_bahooo}
\end{equation}
For each $l\ge1$ let us choose $k_l > k_{\delta_l}$  and consider the sequences $\{(A_{k_l,l},u_{k_l,l})\}_l$ and let $(B_l,v_l):=(B_{k_l}^l,u_{k_l}^l).$ We may also assume that $l\mapsto k_l$ is strictly increasing. By construction and the definition of $u,$ one readily check that 
$(B_l,v_l) \overset{\tau_\admissible}{\longrightarrow} (A,u).$ Moreover, by construction and \eqref{vol_est_Gk}, $|A_{k_l,l}\Delta B_l| = |G_{k_l}^{\delta_l}|\to 0.$
Finally, from \eqref{surface_bahooo} and \eqref{elastic_bahooo} we immediately get 
$$
\liminf\limits_{l\to+\infty} \cF(A_{k_l,l},u_{k_l,l}) \ge 
\liminf\limits_{l\to+\infty} \cF(B_l,u_l). 
$$
Thus, the subsequence $\{(A_{k_l,l},u_{k_l,l})\}_l,$ the sequence $\{(B_l,u_l)\}$ and the configuration $(A,u)$ satisfy the assertions of Theorem \ref{teo:compactness}.
\end{proof}

Note that by construction $|B_l|\le |A_{k_l}|$ and hence, in general our technique does not imply the compactness of energy-equibounded sequences $\{(A_k,u_k)\}$ satisfying a volume constraint.

\subsection{Proof of Proposition \ref{prop:pass_to_good_seq_compacte}}

We start with the following estimates near the points of reduced boundary of $A$ (in Proposition \ref{prop:pass_to_good_seq_compacte}).

\begin{proposition}\label{prop:estimate_red_boundary}
Let $\delta\in(0,1/8),$ $U\subset \R^n$ be an open set, $E_k,E\in BV(U;\{0,1\}),$ and $Q_{r,\nu}(x_0)\strictlyincluded U,$ $r>0,$ $\nu\in\S^{n-1},$ be a cube such that 
\begin{itemize}
\item[\rm(a1)] $x_0\in \p^*E,$ $\nu_E(x_0)=\nu$ and
$$
1-\delta < \frac{1}{\phi(\nu)r^{n-1}} \int_{Q_{r,\nu}(x_0)\cap \p^*E}\phi(\nu_E)\d\cH^{n-1} < 1+\delta;
$$

\item[\rm(a2)] 
$$
\Big(\frac12-\delta\Big)r^n <|E\cap Q_{r,\nu}^-(x_0)|, |E\cap Q_{r,\nu}^+(x_0)| < \Big(\frac12+\delta\Big)r^n,
$$ 
where 
$
Q_{r,\nu}^\pm(x_0) =\{x\in Q_{r,\nu}(x_0):\,\, (x-x_0)\cdot\nu \gtrless0\};
$

\item[\rm(a3)] $E_k\to E$ in $L^1(U).$
\end{itemize}
\noindent
We also denote by $\phi$ a norm in $\R^n$ satisfying \eqref{fjskiii098}. 
Then there exists $k_\delta>0$ such that for any $k>k_\delta$ there is $t_k^\delta\in(\sqrt{\delta},2\sqrt{\delta})$ such that $\cH^{n-1}(T_{t_k^\delta r}\cap \p^* E_k) = 0$ and
$$
\cH^{n-1}(T_{t_k^\delta r}\cap E_k^{(1)}) + \cH^{n-1}(T_{-t_k^\delta r}\cap(Q_1^-\setminus  E^{(1)})) + \cH^{n-1}(T_{-t_k^\delta r}\cap (E_k^{(1)}\Delta E^{(1)})) <4\sqrt{\delta}r^{n-1},
$$
where
$$
T_t:=\{x\in Q_{r,\nu}(x_0):\,\, (x-x_0)\cdot\nu =t\},\quad t\in(-r,r),
$$
and  the set 
$$
D_k^\delta:=Q_{r,\nu}(x_0)\cap \{|(x-x_0)\cdot\nu|<t_k^\delta r\}
$$
satisfies  
\begin{align*}
\int_{D_k^\delta\cap \p^*E_k} \phi(\nu_{E_k})\d\cH^{n-1} 
\ge & \phi(\nu)\cH^{n-1}(T_{-t_k^\delta r})-  (4n +12) \bound_2\sqrt{\delta}r^{n-1}
\end{align*}
(see Figure \ref{fig:min_planes}).
\end{proposition}

\noindent
In the proof of Proposition \ref{prop:pass_to_good_seq_compacte} we apply this proposition with $U=\Omega,$ $E_k:=A_k$ and $E = A.$
\begin{figure}[htp]
\includegraphics[width=0.45\textwidth]{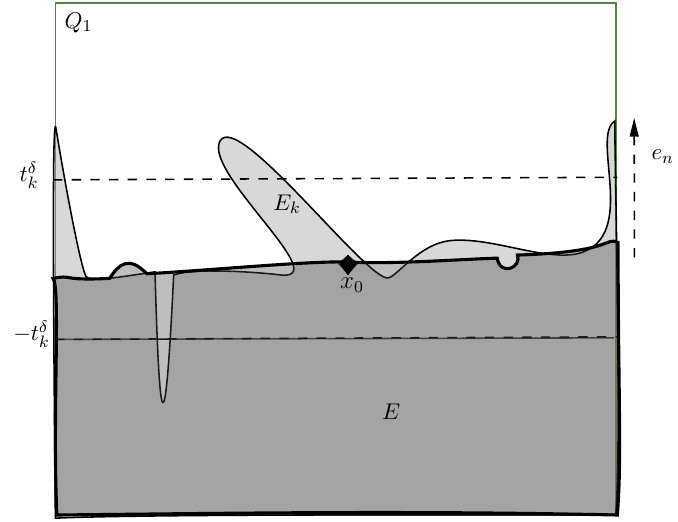} 
\caption{{\small The sets $E_k$ and $E$ in Proposition \ref{prop:estimate_red_boundary}.}}
\label{fig:min_planes}
\end{figure}

\begin{proof}
Without loss of generality we assume that $x_0=0,$ $\nu=\e_n$ and $r=1.$  
By (a2),
$$
|Q_1^+\cap E| \le |E| - |E\cap Q_1^-|<2\delta,
$$
and hence by (a3), there exists $k_\delta>0$ such that
\begin{equation}\label{dadblad}
|Q_1^+\cap E_k| < 2\delta \quad \text{and}\quad |E_k\Delta E|<\delta \quad\text{for all $k>k_\delta.$ }
\end{equation}
Again by (a2),
$$
|Q_1^-\setminus E| \le |Q_1^-| - |Q_1^-\cap E|<\delta,
$$
thus by \eqref{dadblad} and the coarea formula,
\begin{multline*}
4\delta >   |Q_1^+\cap E_k| + |Q_1^-\setminus E| + |E_k\Delta E|
= \int_0^{1/2} \Big[\cH^{n-1}(T_t\cap E_k^{(1)}) + \cH^{n-1}(T_{-t}\cap [Q_1^-\setminus E^{(1)}])\\
 + \cH^{n-1}(T_t\cap [E_k^{(1)}\Delta E^{(1)}]) + \cH^{n-1}(T_{-t}\cap [E_k^{(1)}\Delta E^{(1)}])\Big]\d t.
\end{multline*}
In particular, there exists $t_k^\delta\in(\sqrt{\delta},2\sqrt{\delta})$ such that
\begin{equation}\label{choice_tk}
\cH^{n-1}(T_{t_k^\delta}\cap E_k^{(1)}) + \cH^{n-1}(T_{-t_k^\delta}\cap(Q_1^-\setminus  E^{(1)})) + \cH^{n-1}(T_{-t_k^\delta}\cap (E_k^{(1)}\Delta E^{(1)})) <4\sqrt{\delta}. 
\end{equation}

Set
$$
D_k^\delta:= \big(-1/2,1/2\big)^{n-1}\times (-t_k^\delta,t_k^\delta) 
$$
(see Figure \ref{fig:min_planes}).
Note that 
\begin{multline*}
\int_{D_k^\delta \cap \p^* E_k} \phi(\nu_{E_k})\d\cH^{n-1} =
\int_{\{x\cdot \e_n>-t_k^\delta\}\cap\p^*(D_k^\delta\cap E_k)} \phi(\nu_{D_k^\delta\cap E_k})\d\cH^{n-1} \\
- \int_{\p^*E_k \cap \cl{D_k^\delta}\cap \p Q_1}\phi(\nu_{Q_1})\cH^{n-1} -
\int_{E_k^{(1)}\cap T_{t_k^\delta}}\phi(\e_n)\cH^{n-1}.
\end{multline*}
By the choice \eqref{choice_tk} of $t_k^\delta,$
$$
\int_{E_k^{(1)}\cap T_{t_k^\delta}}\phi(\e_n)\cH^{n-1} \le \bound_2 \cH^{n-1}(E_k^{(1)}\cap T_{t_k^\delta}) < 4\bound_2\sqrt{\delta} 
$$
and 
$$
\int_{\p^*E_k \cap \cl{D_k^\delta}\cap \p Q_1}\phi(\nu_{Q_1})\cH^{n-1} \le 
\bound_2\cH^{n-1}(\p D_k^\delta \cap \p Q_1) < 4(n-1)\bound_2\sqrt{\delta},
$$
where $2(n-1)$ is the perimeter of $(-1/2,1/2)^{n-1}.$
Moreover, by the anisotropic (local) minimality of half-spaces (see e.g. \cite[Example 2.4]{BNH:2017}),
$$
\int_{\{x\cdot \e_n>-t_k^\delta\}\cap \p^*(D_k\cap E_k)} \phi(\nu_{D_k\cap E_k})\d\cH^{n-1} \ge \phi(\e_n)\cH^{n-1}(E_k^{(1)}\cap T_{-t_k^\delta}), 
$$ 
and hence, by \eqref{choice_tk} (we can replace $E_k$ with 
$E$)
$$
\int_{\{x\cdot \e_n>-t_k^\delta\}\cap \p^*(D_k\cap E_k)} \phi(\nu_{D_k\cap E_k})\d\cH^{n-1} \ge \phi(\e_n)\cH^{n-1}(E^{(1)}\cap T_{-t_k^\delta}) - 4\bound_2\sqrt{\delta}. 
$$
Again by \eqref{choice_tk}
$$
\cH^{n-1}(E^{(1)}\cap T_{-t_k^\delta}) = 
\cH^{n-1}(T_{-t_k^\delta}) - \cH^{n-1}((Q_1^- \setminus E^{(1)})\cap T_{-t_k^\delta}) >\cH^{n-1}(T_{-t_k^\delta}) -4\sqrt{\delta} 
$$
and therefore,
$$
\int_{D_k^\delta\cap\p^*E_k} \phi(\nu_{E_k})\d\cH^{n-1} \ge \phi(\e_n)\cH^{n-1}(T_{-t_k^\delta}) - 4(n+3)\bound_2\sqrt{\delta}.
$$
\end{proof}

Now applying Proposition \ref{prop:estimate_inner_jump} and \ref{prop:estimate_red_boundary} we construct the set $G_k^\delta$ in Proposition \ref{prop:pass_to_good_seq_compacte}.

\begin{proof}[\textbf{Proof of Proposition \ref{prop:pass_to_good_seq_compacte}}] 
Without loss of generality we assume $u_k=\xi$ in $\Omega\setminus A_k$ for some $\xi\in (0,1)^n\setminus \Xi_{\{(A_k,u_k)\}\}}$ (see Remark \ref{rem:ext_u_out_A}).

By the uniform continuity of $\varphi,$ there exists  $r_\delta\in(0,1)$ such that 
\begin{equation}\label{jaaryon_phi_cont1818}
|\varphi(x,\nu) - \varphi(y,\nu)|<\delta\quad\text{for all $x,y\in\cl{\Omega}$ with $|x-y|<r_\delta.$} 
\end{equation}
Let
\begin{align*}
\tilde K_1:= & \Sigma \cap \p^*A \cap \bigcup_{i=1}^m \Big( \p S^i \cap \bigcup_{j\ne i}\p^*F^j\Big),\\
\tilde K_2:= & \Omega\cap A^{(1)}\cap \bigcup_{i=0}^m \p^*F^i,\\ 
\tilde K_3:= & \Omega\cap \p^*A \cap \bigcup_{i=0}^m \p^*F^i. 
\end{align*}
Since these sets are $\cH^{n-1}$-rectifiable and pairwise disjoint, (by a simple covering argument) we can find open sets $U_1\strictlyincluded\Ins{\Omega}$ and $U_2,U_3\strictlyincluded \Omega$ with disjoint closures such that
\begin{equation}\label{four_disjoint_sets}
\sum\limits_{i=1}^3 \cH^{n-1}\big(\tilde K_i \setminus U_i\big)
+
\sum\limits_{i=1}^3 \cH^{n-1}\Big(\tilde K_i \cap \bigcup_{j\ne i} U_j\Big) 
< \delta.
\end{equation}
Set 
\begin{align*}
K_i:= U_i\cap \tilde K_i,\quad i=1,2,3. 
\end{align*}
Note that around $\cH^{n-1}$-a.e.\  point of $\cup_i K_i$ there exist $j\in\{1,\ldots,m\}$ and a cube $Q$ such that $\cup_i K_i$ ``roughly divides'' $Q$ into two parts in one $u_k-a_k^j$ converges and in the other either $u_k$ is constant or $|u_k-a_k^j| \to+\infty.$
For convenience of the reader we divide the construction of $G_k^\delta$ into smaller steps.

{\it Step 1.} Using the $\cH^{n-1}$-rectifiability of $K_i,$ $\p^*A,$ $\p^*F^i,$ the lipschitzianity of $\Sigma$ and the Borel regularity of corresponding unit normals we construct a fine cover of $\cup_iK_i$ as follows.
\smallskip

{\it Substep 1.1: fine cover for $K_1.$}  For $\cH^{n-1}$-a.e.\  $x\in K_1$ there exist $i_x,j_x\in\{1,\ldots,m\}$ with $i_x\ne j_x$ and $r_x>0$ such that $x\in (\p S^{i_x}\setminus \p^*F^{i_x})\cap \p^*F^{j_x}$ and:
\begin{itemize}
\item[($\rm a_{1.1}$)] $r_x<\frac14\min\{r_\delta,\dist(x,\p U_1)\},$ where $r_\delta$ is defined in \eqref{jaaryon_phi_cont1818};

\item[($\rm a_{1.2}$)] 
$\theta(\Sigma,x) = \theta(K_1,x) =
\theta(\p^*F^{j_x},x) = 
\theta(\p^*A,x)=1$ and $\nu_{\Sigma}(x),$ $\nu_{K_1}(x),$ $\nu_{F^{j_x}}(x)$ and $\nu_A(x)$ exist and are parallel each other. For shortness, we set $\nu_x:= \nu_\Sigma(x);$
 
 \item[($\rm a_{1.3}$)] $\Gamma_x:=Q_{r_x,\nu_x}(x)\cap\Sigma $ separates $Q_{r_x,\nu_x}(x)$ into two connected components;
 
 \item[($\rm a_{1.4}$)] for any $r\in(0,r_x)$
 \begin{subequations}
 \begin{align}
 & |\nu_{\Gamma_x}(y) - \nu_x|<\delta \quad \text{and}\quad |(y-x)\cdot \nu_x|<\tfrac{\delta r}{2}\quad\text{for all $y\in \Gamma_x,$} \label{good_Gamma_K1}\\
 & (1 - \delta)r^{n-1} < 
 \cH^{n-1}(Q_r \cap  \Gamma_x \cap \p^*F^{j_x})
 \le
 \cH^{n-1}(Q_r\cap \Gamma_x) <(1+\delta)r^{n-1}, \label{lowe_estima_K1}\\
 & \cH^{n-1}\Big(\Big[Q_r \cap \bigcup_{j=0}^m \p^*F^j\Big]\setminus \Gamma_x\Big) +
 \cH^{n-1}(Q_r \cap [\p^*F^{j_x}\Delta \Gamma_x])
 < \delta r^{n-1}, \label{error_esto_K1} \\
 & |(F^{j_x}\cup S)\cap Q_r| \ge (1-\delta)r^n, \label{vol_full_K1}
 \end{align} 
 \end{subequations}
 where $Q_r:=Q_{r,\nu_x}(x).$
\end{itemize}
Removing an $\cH^{n-1}$-negligible set from $K_1$ if necessary we assume that for all points $x\in K_1$ there exist $r_x$ and $i_x,j_x$ satisfying ($a_{1.1}$)-($a_{1.4}$). 

Let us show that for any $x\in K_1$ and $r\in (0,r_x)$, the cube $Q_{r,\nu_x}(x)$, the sequence $\{(A_k,u_k-a_k^{j_x})\},$ the configuration $(A,u^{j_x})$, conditions ($\rm a_{1.1}$)-($\rm a_{1.4}$), the sets $E:=Q_{r_x,\nu_x}(x)\setminus F^{j_x}$ and $K:=Q_{r_x,\nu_x}(x)\cap\p^* F^{j_x}$ satisfy all assumptions of Proposition \ref{prop:estimate_inner_jump}. Indeed, conditions for $\Gamma$ follow from ($\rm a_{1.3}$), \eqref{good_Gamma_K1} and \eqref{lowe_estima_K1}, while conditions (a3)-(a4) for $\{(A_k,u_k)\}$ follows from  our assumption in the beginning of the proof and the assumption of Proposition \ref{prop:pass_to_good_seq_compacte}. The definition of $F^{j_x}$ implies condition (a6) with $E:=Q_{r_x,\nu_x}(x)\setminus F^{j_x}$ and $K:=Q_{r_x,\nu_x}(x)\cap\p^* F^{j_x}.$ Finally, the estimates \eqref{lowe_estima_K1} and \eqref{error_esto_K1} together with ($\rm a_{1.2}$) yield that $A\cup S$ and $K$ satisfy conditions (a5) and (a7), respectively.
\smallskip

{\it Substep 1.2: fine cover for $K_2.$} For $\cH^{n-1}$-a.e.\  $x\in K_3$ there exist $r_x>0,$ $i_x,j_x\in \{0,\ldots,m\}$ with $i_x\ne j_x$ and an $(n-1)$-dimensional $C^1$-graph $\Gamma_x$ containing $x$ such that 
\begin{itemize}
\item[($\rm a_{2.1}$)] $r_x<\frac14\min\{r_\delta,\dist(x,\p U_2)\}.$

\item[($\rm a_{2.2}$)] $\theta(K_2,x) = \theta(\p^*F^{i_x},x) = \theta(\p^*F^{j_x},x) = \theta(K_2\cap \p^* F^{i_x}\cap \p^* F^{j_x}\cap \Gamma_x, x) = 1$ and unit normals $\nu_{K_2},$ $\nu_{F^{i_x}}(x)$ and $\nu_{F^{j_x}}(x)$ exist and is parallel to $\nu_x:=\nu_{\Gamma_x}(x);$

 \item[($\rm a_{2.3}$)] $\Gamma_x$ separates $Q_{r_x,\nu_x}(x)$ into two connected components;
 
 \item[($\rm a_{2.4}$)] for any $r\in (0,r_x)$
 \begin{subequations}
 \begin{align}
 & |\nu_{\Gamma_x}(y) - \nu_x|<\delta \quad \text{and}\quad |(y-x)\cdot \nu_x|<\tfrac{\delta r}{2}\quad\text{for all $y\in \Gamma_x\cap Q_r$,}  \\
 & (1-\delta) r^{n-1} < \cH^{n-1}(Q_r\cap  \Gamma_x\cap  K_2\cap \p^*F^{i_x} \cap \p^*F^{j_x}) \nonumber \\
 &\hspace*{5cm} \le  
 \cH^{n-1}(Q_r \cap \Gamma_x)<(1+\delta)r^{n-1},\label{lowe_denssos_K2}\\
 & \cH^{n-1}(Q_r \cap [\Gamma_x\Delta (\p^*F^{i_x}\cap \p^*F^{j_x})]) +
 \cH^{n-1}\Big(\Big[Q_r \cap \bigcup_{j=0}^{N_2}\p^*F^j\Big] \setminus \Gamma_x\Big) < \delta r^{n-1},\label{error_estoss_K2}\\
 & \Big(\tfrac12-\delta\Big)r^n \le 
 |F^{i_x}\cap Q_r^-|, |F^{j_x}\cap Q_r^+| \le \Big(\tfrac12 +\delta\Big)r^n,\label{volume_densos1092}
 \end{align}
 \end{subequations}
 where $Q_r:=Q_{r,\nu_x}(x)$ and 
 $Q_r^\pm:=\{y\in Q_r:\,\, (y-x)\cdot \nu_x \gtrless 0\}.$ Here the volume density estimates \eqref{volume_densos1092} follow from the definition of the reduced boundary.
\end{itemize}
Removing an $\cH^{n-1}$-negligible set from $K_2$ if necessary we assume that for all points $x\in K_2$  there exist $r_x$ and $i_x,j_x$ satisfying ($\rm a_{2.1}$)-($\rm a_{2.4}$). Then using $A=\cup_{j=0}^{N_2} F^j$ and $\p^*A\subset \cup_{j=0}^{N_2}\p^*F^j$ as in Substep 1.1. one can check that for any $x\in K_2$ and $r\in (0,r_x)$, the cube $Q_{r,\nu_x}(x)$, the sequence $\{(A_k,u_k-a_k^{i_x})\},$ the configuration $(A,u^{i_x})$ and the sets $E:=Q_{r_x,\nu_x}(x)\setminus F^{i_x}$ and $K= Q_{r_x,\nu_x}(x)\cap \p^* F^{i_x}$ satisfy all conditions of Proposition \ref{prop:estimate_inner_jump}. 
\smallskip

{\it Substep 1.3: fine cover for $K_3.$} For $\cH^{n-1}$-a.e.\  $x\in K_3$ there exist $r_x>0,$ $i_x\in \{0,\ldots,m\}$ and an $(n-1)$-dimensional $C^1$-graph $\Gamma_x$ containing $x$ such that 
\begin{itemize}
\item[($\rm a_{3.1}$)]  $r_x<\frac{1}{4}\min\{r_\delta,\dist(x,\p U_4)\};$

\item[($\rm a_{3.2}$)] $\theta(K_3,x) = \theta(\p^*F^{i_x},x) = \theta(\p^*A,x) = \theta(K_3\cap \Gamma_x\cap \p^*A \cap \p^*F^{i_x},x) = 1$ and the unit normals $\nu_{K_3}(x),$ $\nu_A(x)$ and $\nu_{F^{i_x}}(x)$ exist and coincide with $\nu_x:=\nu_{\Gamma_x}(x);$

\item[($\rm a_{3.3}$)] $\Gamma_x$ separates $Q_{r_x,\nu_x}(x)$ into two connected components;

\item[($\rm a_{3.4}$)] for any $r\in(0,r_x)$ 
\begin{subequations}
\begin{align}
& |\nu_{\Gamma_x}(y) - \nu_x|<\delta \quad \text{and}\quad |(y-x)\cdot \nu_x|<\tfrac{\delta r}{2}\quad\text{for all $y\in \Gamma_x\cap Q_r,$} \label{good_surfaces_K3}\\
& (1-\delta)r^{n-1}<\cH^{n-1}(Q_r \cap\Gamma_x \cap K_3\cap \p^*F^{i_x}\cap \p^*A)\nonumber \\
&\hspace*{6cm}\le \cH^{n-1}(Q_r\cap \Gamma_x) < (1+\delta)r^{n-1}, \label{lower_estos_K3}\\
& \cH^{n-1}(Q_r\cap [\Gamma_x\Delta (\p^*F^{i_x}\cap \p^*A)]) + \cH^{n-1}\Big(\Big[Q_r\cap  \bigcup_{j=0}^{N_3} \p^*F^{i_x}\Big]\setminus \Gamma_x\Big) < \delta r^{n-1}, \label{error_estos_K3}\\
& (1-\delta)r^{n-1} < \frac{1}{\varphi(x,\nu_x)}\,\int_{Q_r \cap \p^*F^{i_x}} \varphi(x,\nu_{F^{i_x}}(y))\d\cH^{n-1}(y) \\
&\hspace*{3cm} \le \frac{1}{\varphi(x,\nu_x)}\,\int_{Q_r \cap \p^*A} \varphi(x,\nu_{A}(y))\d\cH^{n-1}(y)
< (1+\delta)r^{n-1}, \label{lebesgue_K3} \\
& \Big(\tfrac12-\delta\Big)r^n < |Q_r^- \cap F^{i_x}| \le |Q_r^- \cap A| < \Big(\tfrac12+\delta\Big)r^n, \label{volume_conve_K3}\\
& |Q_r^+ \cap A| < \delta r^n,
\end{align}
\end{subequations}
where $Q_r:=Q_{r,\nu_x}(x).$ 
\end{itemize}
Removing an $\cH^{n-1}$-negligible set from $K_3$ if necessary we assume that for all points $x\in K_3$ there exists $r_x>0$ and $i_x$ satisfying ($\rm a_{3.1}$)-($\rm a_{3.4}$). Then for any $x\in K_3$ and $r\in (0,r_x)$ the set $U=U_3,$ the cube $Q_{r,\nu_x}(x),$ the sequence $E_k:=Q_{r,\nu_x}(x)\cap A_k,$ the set $E:=Q_{r,\nu_x}(x)\cap A$ and conditions ($\rm a_{3.1}$)-($\rm a_{3.4}$) satisfy all assumptions of Proposition \ref{prop:estimate_red_boundary}. Indeed, conditions (a1)-(a2) are given in \eqref{lebesgue_K3} and \eqref{volume_conve_K3}, whereas (a3) follows from the assumption $A_K\to A$ in $L^1(\R^n)$ as $k\to+\infty.$ 
\smallskip

\begin{figure}[htp!]
\includegraphics[width=0.97\textwidth]{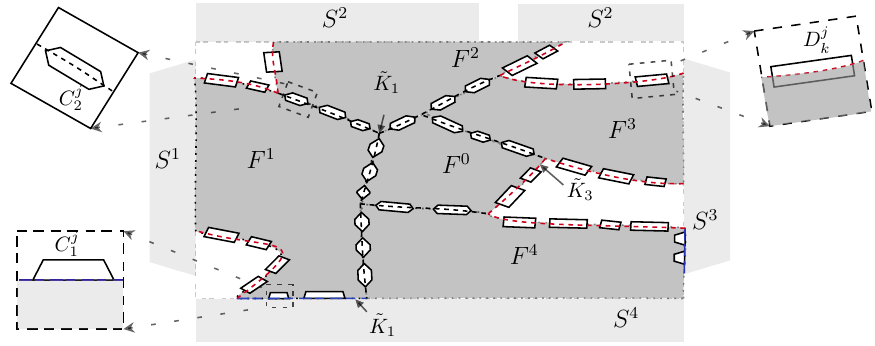}
\caption{{\small Construction of holes $C_1^j,$ $C_2^j$ and $D_k^j.$}}
\label{fig:set_Gdelta}
\end{figure}

{\it Step 2.} Now we extract finitely many covering cubes still covering $\cup_iK_i$ up to some error of order $O(\sqrt\delta),$ and create ``holes'' inside those cubes (i.e., the sets $C_1^j,$ $C_2^j$ and $D_k^j$ in Figure \ref{fig:set_Gdelta}). By Step 1, for each $i\in \{1,2,3\}$ the collection $\{\cl{Q_{r,\nu_x}(x)}:\,\,x\in K_i,\,\, r\in(0,r_x)\}$ of cubes provides a fine cover for $K_i,$ and hence by the Vitali covering lemma, we can extract an at most countable pairwise disjoint family $\{Q_{r_j^i,\nu_{x_j^i}}(x_j^i),\,x_j^i\in K_i\}$ such that
$$
\cH^{n-1}\Big(K_i\setminus \bigcup_{j} Q_{r_j^i,\nu_{x_j^i}}(x_j^i)\Big)=0. 
$$
Since $\cH^{n-1}(K_i)<+\infty$,  there exists $N_i\ge1$ such that  
\begin{equation}\label{almost_cover_Ki}
\cH^{n-1}\Big(K_i\setminus \bigcup_{j > N_i} Q_{r_j^i,\nu_{x_j^i}}(x_j^i)\Big) < \delta.
\end{equation}
Moreover, decreasing $r_j$ a bit necessary, we assume that $\cl{Q_{r_j^i,\nu_{x_j^i}}(x_j^i)}\cap \cl{ Q_{r_{j'}^i,\nu_{x_{j'}^i}}(x_{j'}^i)}=\emptyset$ for all $1\le j< j' \le N_i.$ 
Since $\cl{U_i}\cap \cl{U_j} =\emptyset$ for $i\ne j,$ cubes belonging to the union of $\cG_i:=\{Q_{r_j^i,\nu_{x_j^i}}(x_j^i)\}_{j=1}^{N_i},$ $i=1,2,3,$ have disjoint closures. When no confusion arises, we drop the dependence of $x_j^i$ and $r_j^i$ on $i.$
\smallskip

{\it Substep 2.1: definition of $C_1^j.$} Let $Q_{r_j,x_j}(x_j) \in \cG_1$ for some $j\in\{1,\ldots,N_1\}.$ By Substep 1.1 $x_j\in K_1\cap \p S^{l_j}\cap \p^*F^{h_j}$ for some $l_j,h_j\in\{1,\ldots,m\}$ with $l_j\ne h_j.$ Applying
Proposition \ref{prop:estimate_inner_jump} (ii) with $Q_{r_j,\nu_{x_j}}(x_j)\strictlyincluded \Ins{\Omega},$ $\Gamma_{x_j}:=Q_{r_j,\nu_{x_j}}(x_j)\cap \Sigma,$ $\{(A_k,u_k-a_k^{h_j})\},$ $(A,u^{h_j}),$ $E:=Q_{r_j,\nu_{x_j}}(x_j)\setminus F^{h_j},$ $K:=Q_{r_j,\nu_{x_j}}(x_j)\cap \p^* F^{h_j}$ and $\phi(\cdot)=\varphi(x_j,\cdot)$ we find an open set $C_1^j\subset \Omega\cap Q_{r_j,\nu_{x_j}}(x_j)$ of finite perimeter (given by Lemma \ref{lem:creating_hole}) and $k_\delta^{1,j} > 0$ such that 
\begin{align}
\int_{C_1^j\cap\p^*A_k} \phi(\nu_{A_k})\d\cH^{n-1} +  2\int_{C_1^j \cap A_k^{(1)} \cap J_{u_k}}  \phi(\nu_{J_{u_k}}) &\d\cH^{n-1}   + 2 \int_{\Sigma\cap \p^* C_1^j \cap \p^*A_k \cap J_{u_k}} \phi(\nu_{J_{u_k}}) \d\cH^{n-1}\nonumber \\
\ge & 2\int_{Q_{r_j,\nu_j}(x_j)\cap \p^*F^{h_j}} \phi(\nu_{F^{h_j}})\d\cH^{n-1} - c'\delta r_j^{n-1}\nonumber \\
\ge & 
\int_{\p^* C_1^j} \phi(\nu_{C_1^j})\d\cH^{n-1} - (c' + 5\bound_2)\delta r_j^{n-1} \label{estimate_Cj1}
\end{align}
for all $k>k_\delta^{1,j}$ and for some $c'>0$ (depending only on $\bound_2$).

Let us estimate the perimeter and the volume of $\cup_j C_1^j.$ By \eqref{lowe_estima_K1},
\begin{equation}\label{rjn_1_estimates}
r_j^{n-1} \le \frac{1}{1-\delta}\cH^{n-1}(Q_{r_j,\nu_j}\cap \Sigma\cap \p^*F^{h_j}) 
\end{equation}
and hence  by \eqref{finsler_norm} and \eqref{one_dim_esimates_C},
\begin{align*}
\bound_1\cH^{n-1}(\p^*C_1^j) \le & \int_{\p^*C_1^j}\phi(\nu_{C_1^j})\d\cH^{n-1} \le 2 \int_{Q_{r_j,\nu_j}(x_j)\cap \p^*F^{h_j}} \phi(\nu_{F^{h_j}})\d\cH^{n-1} + 5\bound_2\delta r_j^{n-1} \\
\le & 2\bound_2 \cH^{n-1}(Q_{r_j,\nu_j}(x_j)\cap \p^*F^{h_j}) + 5\bound_2\delta r_j^{n-1} 
\end{align*}
so that
\begin{equation}
\cH^{n-1}(\p^*C_1^j) \le
\frac{3\bound_2}{\bound_1} \cH^{n-1}(Q_{r_j,\nu_j}(x_j) \cap \p^*F^{h_j}).
\label{per_C1j_ooopopo}
\end{equation}
Moreover,
\begin{equation*}
|C_1^j|\le \delta r_j^n < \delta r_j^{n-1} \le 2\delta  \cH^{n-1}(Q_{r_j,\nu_j}(x_j)\cap \Sigma\cap \p^*F^{h_j})
\end{equation*}
and therefore,
\begin{equation}\label{tot_volest_K1}
\Big|\bigcup_{j=1}^{N_i} C_1^j\Big| \le 2\delta \sum\limits_{h=1}^{m}\cH^{n-1}(\p^*F^h).
\end{equation}

Let us estimate the error in covering $K_1$ by $\{C_1^j\}.$ Fix some $j\in\{1,\ldots,N_1\}.$ By the definition of $K_1,$ the error estimate \eqref{error_esto_K1} and Lemma \ref{lem:creating_hole} (ii),
\begin{multline*}
\cH^{n-1}((Q_{r_j,\nu_{x_j}}(x_j)\cap K_1)\setminus \cl{C_1^j}) \le \cH^{n-1}\Big(\Big[Q_{r_j,\nu_{x_j}}(x_j)\cap \bigcup_{j=0}^{N_1}\p^*F^j\Big]\setminus \Gamma_{x_j}\Big) + 
\cH^{n-1}(\Gamma_{x_j}\setminus \p^*F^{h_j})\\
+ \cH^{n-1}([Q_{r_j,\nu_{x_j}}(x_j)\cap\p^*F^{h_j}] \setminus \cl{C_1^j}) < 3\delta r_j^{n-1}.
\end{multline*}
Thus, by \eqref{rjn_1_estimates}  and the choice $\delta<1/8,$
\begin{equation}\label{lkp89888}
\cH^{n-1}((Q_{r_j,\nu_{x_j}}(x_j)\cap K_1)\setminus \cl{C_1^j}) < 4\delta\,\cH^{n-1}(Q_{r_j,\nu_{x_j}}(x_j)\cap \Sigma\cap \p^*F^{h_j}). 
\end{equation}
From \eqref{almost_cover_Ki} and \eqref{lkp89888} it follows that
\begin{align*}
\cH^{n-1}\Big(K_1\setminus \bigcup_{j=1}^{N_1}\cl{C_1^j}\Big) = & \cH^{n-1}\Big(K_1\setminus \bigcup_{j>N_1}Q_{r_j,\nu_j}(x_j)\Big) + \sum\limits_{j=1}^{N_1} \cH^{n-1}([Q_{r_j,\nu_{x_j}}(x_j) \cap K_1]\setminus \cl{C_1^j})\\
< & \delta + 
4\delta\sum_{j=1}^{N_1}\cH^{n-1}(Q_{r_j,\nu_{x_j}}(x_j)\cap \Sigma\cap \p^*F^{h_j})
\end{align*}
so that  by the disjointness of $\{F^h\},$
\begin{equation}\label{outside_Cj1}
\cH^{n-1}\Big(K_1\setminus \bigcup_{j=1}^{N_1}\cl{C_1^j}\Big) < \delta + 4\delta\sum\limits_{h=1}^{m}\cH^{n-1}(\p^*F^h).
\end{equation}

{\it Substep 2.2: construction of $C_2^j$.} Let $Q_{r_j,\nu_j}(x_j) \in \cG_2$ for some $j\in \{1,\ldots,N_2\}$ so that there exist $l_j,h_j\in\{0,\ldots,m\}$ with $l_j\ne h_j\ne 0$ such that $x_j\in \p^*F^{l_j}\cap \p^*F^{h_j}.$  As in Substep 2.1 applying Proposition \ref{prop:estimate_inner_jump} with $Q_{r_j,\nu_{x_j}}(x_j)\strictlyincluded \Omega,$ $\Gamma_{x_j},$ $\{(A_k,u_k-a_k^{h_j})\},$ $(A,u^{h_j}),$ $E:=Q_{r_j,\nu_j}(x_j)\setminus F^{h_j},$ $K:=Q_{r_j,\nu_j}(x_j)\cap\p^* F^{h_j}$ and $\phi(\cdot)=\varphi(x_j,\cdot)$ we find an open set $C_2^j\strictlyincluded Q_{r_j,\nu_{x_j}}(x_j)$ of finite perimeter (given by Lemma \ref{lem:creating_hole}) and $k_\delta^{2,j}>0$ such that  
\begin{align}\label{casasasasas}
\int_{C_2^j\cap\p^*A_k} \phi(\nu_{A_k})\d\cH^{n-1}  +   2\int_{C_2^j \cap A_k^{(1)} \cap J_{u_k}}   \phi(\nu_{J_{u_k}}) \d\cH^{n-1}  
\ge \int_{\p C_2^j} \phi(\nu_{C_2^j})\d\cH^{n-1} -c' \delta r_j^{n-1} 
\end{align}
for all $k>k_\delta^{2,j},$ where $c'$ depends only on $\bound_2.$ As in Substep 2.1,
by \eqref{lowe_denssos_K2} we have
\begin{equation}\label{rj2_esimates}
r_j^{n-1} \le \frac{1}{1-\delta}\cH^{n-1}(Q_{r_j,\nu_j}(x_j)\cap \p^*F^{l_j}\cap \p^*F^{h_j}) 
\end{equation}
while by \eqref{finsler_norm} and \eqref{two_dim_esimates}, we have
\begin{equation}\label{per_Cj2_opop}
\cH^{n-1}(\p^*C_2^j) \le \frac{2\bound_2}{\bound_1}\cH^{n-1}(Q_{r_j,\nu_j}(x_j)\cap \p^*F^{h_j}) + \frac{5\bound_2}{\bound_1}\,\delta r_j^{n-1} \le 
\frac{3\bound_2}{\bound_1} \cH^{n-1}(Q_{r_j,\nu_j}(x_j) \cap \p^*F^{h_j}) 
\end{equation}
and 
\begin{equation}\label{tot_volest_K3}
\Big|\bigcup_{j=1}^{N_2} C_2^j\Big| \le 2\delta \sum_{h=0}^{m}\cH^{n-1}(\p^*F^h).
\end{equation} 
Moreover, 
\begin{equation}\label{outside_Dj2}
\cH^{n-1}\Big(K_2\setminus \bigcup_{j=1}^{N_2}\cl{C_2^j}\Big)< \delta + 4\delta\sum_{h=0}^{m}\cH^{n-1}(\p^*F^h) 
\end{equation}

{\it Substep 2.3: construction of $D_k^j$.} Let $Q_{r_j,\nu_j}(x_j) \in\cG_3$ for some $j\in \{1,\ldots,N_3\}$ and let $x_j\in \p^*F^{h_j}\cap\p^*A$ for some $h_j\in\{0,\ldots,m\}.$ Using Proposition \ref{prop:estimate_red_boundary} applied with $U:=Q_{r_j,\nu_{x_j}}(x_j),$ $E_k:=Q_{r_j,\nu_{x_j}}(x_j) \cap A_k,$ $E:=Q_{r_j,\nu_{x_j}}(x_j)\cap A$ and $\phi(\cdot)=\varphi(x_j,\cdot)$ we find $k_\delta^{3,j} > 0$ such that for any $k > k_\delta^{3,j}$ there exists $t_{k,j}^\delta\in(\sqrt{\delta},2\sqrt{\delta})$ such that $\cH^{n-1}(\p^*A_k\cap T_{t_{k,j}^\delta,r_j}^j) = 0$ and
\begin{multline}\label{stupid_tdelta}
\cH^{n-1}( T_{t_{k,j}^\delta r_j}^j\cap A_k^{(1)}) + \cH^{n-1}( T_{-t_{k,j}^\delta r_j}^j\cap  [Q_{r_j,\nu_j}^-(x_j)\setminus  A^{(1)}]) \\
  + \cH^{n-1}( T_{-t_{k,j}^\delta r_j}^j\cap [A_k^{(1)}\Delta A^{(1)}]) <4\sqrt{\delta}r_j^{n-1},
\end{multline}
where 
$$
 T_t^j:=\{x\in Q_{r_j,\nu_j}(x_j):\,\, (x-x_j)\cdot \nu_j=t\},\quad t\in(-r_j,r_j),
$$
and the set 
$$
D_k^j:=\{x\in Q_{r_j,\nu_{x_j}}(x_j):\,\,|(x-x_j)\cdot\nu_{x_j}|<t_{k,j}^\delta r_j\}
$$
satisfy 
\begin{align}\label{yalililalala}
\int_{D_k^j\cap \p^*A_k} \phi(\nu_{A_k})\d\cH^{n-1} 
\ge \phi(\nu_j)\cH^{n-1}( T_{-t_{k,j}^\delta r_j}^j )
- c'\sqrt{\delta}r_j^{n-1}
\end{align}
for some $c'>0$ depending only on $\bound_2$ and $n.$ Note that by \eqref{lower_estos_K3},
\begin{equation}\label{rsdadad}
r_j^{n-1} \le \frac{1}{1-\delta}\cH^{n-1}(Q_{r_j,\nu_j}(x_j) \cap \p^*F^{h_j}\cap \p^*A) 
\end{equation}
and hence by the choice of $t_{k,j}^\delta$ and \eqref{rsdadad},
$$
|D_k^j|=2t_{k,j}^\delta r^n\le \frac{4\sqrt{\delta}}{1-\delta} \cH^{n-1}(Q_{r_j,\nu_j}(x_j)\cap \p^*F^{h_j}) 
$$
so that
\begin{equation}\label{tot_volest_K4}
\Big|\bigcup_{j=1}^{N_3} D_k^j\Big| \le 5\sqrt{\delta}\sum\limits_{h=1}^m\cH^{n-1}(\p^*F^h). 
\end{equation}
Moreover, by the definition of $D_k^j,$ \eqref{stupid_tdelta}, \eqref{rsdadad} and the equality $\cH^{n-1}( T_{\pm t_{k,j}^\delta r_j}^j ) = r_j^{n-1},$ we have
\begin{equation}\label{per_Dkj_opopa}
\cH^{n-1}(\p^*D_k^j) \le (2+4\sqrt{\delta})r_j^{n-1}\le 4\cH^{n-1}(Q_{r_j,\nu_j}(x_j) \cap \p^*F^{h_j}). 
\end{equation}

Let us estimate the error in covering $K_3$ with $\{D_k^j\}.$ Fox some $j\in\{1,\ldots,N_3\}.$ Recalling the definition of $\Gamma_{x_j}$ in Substep 1.3 and in view of \eqref{good_surfaces_K3}, we have $Q_{r_j,\nu_j}(x_j)\cap \Gamma_{x_j}\subset D_k^j$ and hence  by \eqref{error_estos_K3} and \eqref{rsdadad},
\begin{align*}
\cH^{n-1}([K_3\cap  &  Q_{r_j,\nu_j}(x_j)]   \setminus D_k^j)\\
\le & 
\cH^{n-1}\Big(\Big[Q_{r_j,\nu_j}(x_j) \cap \bigcup_{j=1}^{N_3}\Big]\setminus \Gamma_{x_j}\Big)
+ 
\cH^{n-1}([Q_{r_j,\nu_j}(x_j) \cap \Gamma_{x_j}]\setminus [\p^*A\cap \p^*F^{h_j}])\\
\le & \delta r_j^{n-1} \le 2\delta \cH^{n-1}(Q_{r_j,\nu_j}(x_j) \cap \p^*F^{h_j}).
\end{align*}
Hence, by \eqref{almost_cover_Ki},
\begin{align*}
\cH^{n-1}\Big(K_3\setminus\bigcup_{j=1}^{N_3} D_k^j\Big) = & \cH^{n-1}\Big(K_3\setminus\bigcup_{j>N_3} Q_{r_j,\nu_j}(x_j)\Big) + \sum\limits_{j=1}^{N_3} \cH^{n-1}\Big([K_3\cap Q_{r_j,\nu_j}(x_j)]\setminus D_k^j\Big) 
\end{align*}
so that
\begin{equation}\label{outside_Dkjdelta}
\cH^{n-1}\Big(K_3\setminus \bigcup_{j=1}^{N_3}\cl{D_k^j}\Big)< \delta + 2\delta \sum\limits_{h=1}^m \cH^{n-1}(\p^*F^h).
\end{equation}

{\it Step 3: Definition of $G_k^\delta$.}  
Let $k_\delta^i:=\max_{j=1,\ldots,N_i} k_\delta^{i,j},$ $i=1,2,3,$ and for each $k>k_\delta:=\max\{k_\delta^1,k_\delta^2,k_\delta^3\}$ let us define
$$
G_k^\delta:=\bigcup_{j=1}^{N_1} C_1^j \cup \bigcup_{j=1}^{N_2} C_2^j\cup \bigcup\limits_{j=1}^{N_3} D_k^j, 
$$
obviously, $G_k^\delta$ is open. 
By \eqref{tot_volest_K1},  \eqref{tot_volest_K3} and \eqref{tot_volest_K4} as well as the inclusion $\p^*A\subset \cup_j\p^*F^j,$ we get
$$
|G_k^\delta| \le \Big|\bigcup_{j=1}^{N_1} C_1^j\Big| + \Big|\bigcup_{j=1}^{N_2} C_2^j\Big| + \Big|\bigcup\limits_{j=1}^{N_3} D_k^j\Big|
\le 8\sqrt\delta \sum\limits_{h=0}^m \cH^{n-1}(\p^*F^h).
$$
Moreover, summing the estimates  
\eqref{per_C1j_ooopopo}, \eqref{per_Cj2_opop} and
\eqref{per_Dkj_opopa}
and using the disjointness of the closures of $C_1^j,$ $C_2^j$ and $D_k^j$ (because so are the containing cubes) we get
\begin{align*}
P(G_k^\delta) \le  \sum_{j=1}^{N_1} P(C_1^j) + \sum_{j=1}^{N_2} P(C_2^j) +  \sum\limits_{j=1}^{N_3} P(D_k^j)  
\le \Big(4 + \frac{3\bound_2}{\bound_1}\Big)\sum\limits_{h=0}^m \cH^{n-1}(\p^*F^h).
\end{align*}

{\it Step 4: Definition of $s_\delta.$} Since $A_k\to A$ in $L^1(\R^n),$ by the coarea formula applied with the 1-Lipschitz function $f(x)=\dist(x,S),$ we have
$$
0=\lim\limits_{k\to+\infty} |A_k\Delta A| = \int_0^\infty \cH^{n-1}(\{x\in A_k\Delta A:\,\,\dist(x,S) = s\}) \d s
$$
and thus, passing to a not relabelled subsequence if necessary, we may assume
\begin{equation*}
\lim\limits_{k\to+\infty} \cH^{n-1}(\{x\in A_k\Delta A:\,\,\dist(x,S) = s\}) = 0 \quad \text{for a.e.\  $s>0.$ }
\end{equation*}
In particular, there exists $s_\delta\in (0,\delta)$ such that
\begin{equation*}
\cH^{n-1}([A_k\Delta A] \cap \{\dist(\cdot, S) = s_\delta\}\}) < \delta \quad \text{and} \quad 
\cH^{n-1}(\{0 < \dist(\cdot, S) < s_\delta\}\} \cap  \p^*A) < \delta,
\end{equation*}

{\it Step 5: Proof of \eqref{error_estimate_Bv}.} Let $B_k^\delta$ and $v_k^\delta$ be given by \eqref{def_BBkdelta} and \eqref{def_vvkdelta}.
As in the proof of lower semicontinuity, given $(B,v)\in\admissible$ and a Borel set $D\subset\R^n,$ let us introduce
\begin{align*}
\mu_{B,v}(D) := & \int_{D\cap \p^*B}\varphi(x,\nu_{B})\d\cH^{n-1} + 2\int_{D\cap B^{(1)}\cap J_v} \varphi(x,\nu_{J_v})\d\cH^{n-1}\\
&+ 2 \int_{D\cap \Sigma\cap \p^*B\cap J_v} \varphi(x,\nu_{\Sigma})\d\cH^{n-1} + 
\int_{D\cap \Sigma\cap \p^*B\setminus J_v} [\beta + \varphi(x,\nu_{\Sigma})]\d\cH^{n-1}\\
& + \int_{D\cap \Sigma\setminus \p^*B} \varphi(x,\nu_{\Sigma})\d\cH^{n-1}.
\end{align*}
Since
$
\mu_{B,v}(\R^n) = \cS(B,v) + \int_{\Sigma}\varphi(x,\nu_\Sigma)\d\cH^{n-1}, 
$
we have
$$
\cS(A_k,u_k) - \cS(B_k^\delta,v_k^\delta) = 
\mu_{A_k,u_k}(\R^n) - \mu_{B_k^\delta,u_k^\delta}(\R^n).
$$
By construction,
\begin{align*}
& [\Omega\cap \p^*A_k] \setminus \cl{G_k^\delta} = [\Omega\cap (\p^*B_k^\delta] \setminus \cl{G_k^\delta},\quad
[\Sigma \cap \p^*A_k] \setminus \cl{G_k^\delta} = \Sigma \cap \p^*B_k^\delta,\\
& [\Sigma\cap \p^*A_k \cap J_{u_k}]\setminus \cl{G_k^\delta} = \Sigma\cap \p^*B_k^\delta \cap J_{v_k^\delta},
\quad
\Sigma\setminus \Big(\p^*A_k\cup \bigcup_{j=1}^{N_1} \p^*C_1^j) = \Sigma\setminus \p^*B_k^\delta,\\
& [A^{(1)}\cap A_k^{(1)}\cap J_{u_k} ]\setminus \cl{G_k^\delta} =  A^{(1)}\cap B_k^{(1)}\cap J_{v_k^\delta},\\
& [A^{(1)}\cap J_{v_k^\delta}]\setminus J_{u_k} = \bigcup_{j=1}^m \p^*F^j \setminus G_k^\delta,
\quad 
J_{v_k^\delta}\cap \p^*A\subset \bigcup_{j=0}^m \p^*F^j \setminus G_k^\delta,\\
& [A_k^{(1)}\setminus A^{(1)}]\cap J_{v_k^\delta} 
\subseteq ([R_\delta\setminus A]^{(1)}\cap J_{u_k}) \cup ([A_k \setminus A]^{(1)}\cap \p R_\delta)\cup (R_\delta \cap\p^*A),
\end{align*}
and hence
\begin{multline}\label{juhuggza}
\cS(A_k,u_k) - \cS(B_k^\delta,u_k^\delta) \ge
\mu_{A_k,u_k}(\cl{G_k^\delta}) - \mu_{B_k^\delta,u_k^\delta}(\cl{G_k^\delta})  
- 2\int_{R_\delta\cap \p^*A} \varphi(x,\nu_A)\d\cH^{n-1} \\
- 2
\int_{J_{v_k^\delta}\cap [B_k^\delta]^{(1)}\cap \bigcup_{i=0}^m\p^*F^j} \varphi(x,\nu_{J_{v_k^\delta}})\d\cH^{n-1} - 2 \int_{[A_k\setminus A]^{(1)}\cap \p R_\delta} \varphi(x,\nu_{R_\delta})\d \cH^{n-1}.
\end{multline}
By \eqref{finsler_norm}, the definition of $\tilde K_j$ and $U_j,$ the construction of $C_1^j,C_2^j,D_k^j$, the choice of $s_\delta$ and the error estimates \eqref{outside_Cj1}, \eqref{outside_Dj2}, \eqref{outside_Dkjdelta} and \eqref{four_disjoint_sets}, we have
\begin{multline}
\int_{R_\delta\cap \p^*A} \varphi(x,\nu_A)\d\cH^{n-1} +
\int_{J_{v_k^\delta}\cap [B_k^\delta]^{(1)}\cap \bigcup_{i=0}^m\p^*F^j} \varphi(x,\nu_{J_{v_k^\delta}})\d\cH^{n-1}\\
+ \int_{[A_k\setminus A]^{(1)}\cap \p R_\delta} \varphi(x,\nu_{R_\delta})\d \cH^{n-1} \le c_4^*\delta\Big(1 + \sum\limits_{h=0}^m\cH^{n-1}(\p^*F^h)\Big).
\label{first_error0901}
\end{multline}
Furthermore, from the additivity of the set-function $\alpha_{B,v}$ and disjointness of the closures of $C_1^j,$ $C_2^j$ and $D_k^j$ we obtain
\begin{align}
\mu_{A_k,u_k}(\cl{G_k^\delta}) - \mu_{B_k^\delta,u_k^\delta}(\cl{G_k^\delta})  
= & \sum\limits_{j=1}^{N_1} \Big[\mu_{A_k,u_k}(\cl{C_1^j}) - \mu_{B_k^\delta,u_k^\delta}(\cl{C_j^1})\Big]
+  
\sum\limits_{j=1}^{N_2} \Big[\mu_{A_k,u_k}(\cl{C_2^j}) - \mu_{B_k^\delta,u_k^\delta}(\cl{C_2^1})\Big] \nonumber \\
& + \sum\limits_{j=1}^{N_3} \Big[\mu_{A_k,u_k}(\cl{D_k^j}) - \mu_{B_k^\delta,u_k^\delta}(\cl{D_k^1})\Big]:=I_1+I_2+I_3.
\label{don_oomar_imala}
\end{align}

{\it Substep 5.1: A lower estimate for $I_1.$} 
Let 
\begin{multline*}
\alpha_k^{1,j}:= \int_{C_1^j\cap\p^*A_k} \varphi(x,\nu_{A_k})\d\cH^{n-1} +  2\int_{C_1^j \cap A_k^{(1)} \cap J_{u_k}}  \varphi(x,\nu_{J_{u_k}}) \d\cH^{n-1}\\
+ 2 \int_{\Sigma\cap \p^* C_1^j \cap \p^*A_k \cap J_{u_k}} \varphi(x,\nu_{J_{u_k}}) \d\cH^{n-1}.
\end{multline*}
By \eqref{jaaryon_phi_cont1818} and \eqref{estimate_Cj1},
\begin{equation}\label{saoiudd_K1}
\alpha_k^{1,j}\ge \int_{\p^* C_1^j} \varphi(x,\nu_{C_1^j})\d\cH^{n-1} - \delta \cH^{n-1}(\cl{C_1^j}\cap [J_{u_k}\cup \p^*A_k]) - \delta\cH^{n-1}(\p^* C_1^j) - c'\delta r_j^{n-1}.  
\end{equation}
Since $|\beta(x)|\le \phi(x,\nu_\Sigma)$ (see \eqref{hyp:bound_anis}), by the definition of $\mu_{A_k,u_k},$ $(B_k^\delta,v_k^\delta)$ and $\mu_{B_k^\delta,v_k^\delta},$ we have
$$
\mu_{A_k,u_k}(\cl{C_j^1}) \ge \alpha_k^{1,j}
\quad \text{and}\quad
\int_{\p^* C_1^j} \varphi(x,\nu_{C_1^j})\d\cH^{n-1} = \mu_{B_k^\delta}(\cl{C_1^j}).
$$
Therefore, from \eqref{saoiudd_K1} and \eqref{rjn_1_estimates} we get
\begin{multline*}
\mu_{A_k,u_k}(\cl{C_j^1})  - \mu_{B_k^\delta}(\cl{C_1^j}) 
\ge 
- \delta \cH^{n-1}(Q_{r_j,\nu_j}(x_j) \cap [J_{u_k}\cup \p^*A_k]) \\
- \delta\cH^{n-1}(\p^* C_1^j) - \frac{c'\delta}{1-\delta} \cH^{n-1}(Q_{r_j,\nu_j}\cap \Sigma\cap \p^*F^{h_j}).
\end{multline*}
Summing these estimates in $j$ and using the disjointness of $\{Q_{r_j,\nu_j}(x_j)\}$ and the perimeter estimate \eqref{per_C1j_ooopopo} for $C_1^j$ we obtain
\begin{align}\label{estimatos_I1osos}
I_1 \ge - c_1^*\delta \Big(\cH^{n-1}(J_{u_k}) + \cH^{n-1}(\p^*A_k) + \sum\limits_{h=0}^m \cH^{n-1}(\p^*F^h)\Big)
\end{align}
for all $k>k_\delta^1=\max_jk_\delta^{1,j}$ and for some $c_1^*$ depending only on $\bound_1$ and $\bound_2.$

{\it Substep 5.2: A lower estimate for $I_2.$} Let 
$$
\alpha_k^{2,j}:= \int_{C_2^j\cap\p^*A_k} \varphi(x,\nu_{A_k})\d\cH^{n-1}  +   2\int_{C_2^j \cap A_k^{(1)} \cap J_{u_k}}   \varphi(x,\nu_{J_{u_k}}) \d\cH^{n-1}.
$$
By \eqref{jaaryon_phi_cont1818} and \eqref{casasasasas},

\begin{align}\label{aser_saeras}
\alpha_k^{2,j}\ge & \int_{\p^* C_2^j} \varphi(x,\nu_{C_2^j})\d\cH^{n-1} - \delta\, \cH^{n-1}(Q_{r_j,\nu_{x_j}}(x_j) \cap [\p^*A_k\cup J_{u_k}] - \delta \cH^{n-1}(\p^*C_2^j) - c'\delta r_j^{n-1}
\end{align}
for all $k>k_\delta^{2,j}.$
Since $\cl{C_2^j}\cap\Sigma=\emptyset,$ from the definition of $\mu_{A_k,u_k},$ $(B_k^\delta,v_k^\delta)$ and $\mu_{B_k^\delta,v_k^\delta}$ we have 
$$
\mu_{A_k,u_k}(\cl{C_2^j}) = \alpha_k^{2,j}\quad\text{and}\quad 
\int_{\p^* C_2^j} \varphi(x,\nu_{C_2^j})\d\cH^{n-1} = \mu_{B_k^\delta,v_k^\delta}(\cl{C_2^j}),
$$
and thus, using \eqref{rj2_esimates} and \eqref{per_Cj2_opop} in \eqref{aser_saeras} we obtain
\begin{multline*}
\mu_{A_k,u_k}(\cl{C_2^j}) - 
 \mu_{B_k^\delta,v_k^\delta}(\cl{C_2^j}) \\
 \ge 
 -c_2^*\delta \Big( \cH^{n-1}(Q_{r_j,\nu_{x_j}}(x_j) \cap [\p^*A_k\cup J_{u_k}] + \cH^{n-1}(Q_{r_j,\nu_{x_j}}(x_j) \cap \p^*F^{l_j}\cap \p^*F^{h_j})\Big)
\end{multline*}
for some constant $c_2^*>0$ depending only on $\bound_1,\bound_2.$
Summing these estimates we get
\begin{align}\label{estimos_I2oplas}
I_2 \ge   - c_2^*\delta \Big(\cH^{n-1}(J_{u_k}) + \cH^{n-1}(\p^*A_k) + \sum_{h=0}^{m}\cH^{n-1}(\p^*F^h)\Big)
\end{align}
for all $k>k_\delta^2=\max_jk_\delta^{2,j}.$

{\it Substep 5.3: A lower estimate for $I_3.$} Let
\begin{align*}
\alpha_k^{3,j}:= \int_{D_k^j\cap \p^*A_k} \varphi(x,\nu_{A_k})\d\cH^{n-1}.
\end{align*}
Since 
$\cH^{n-1}( T_{-t_{k,j}^\delta r_j}) r_j^{n-1},$ using \eqref{jaaryon_phi_cont1818} and \eqref{yalililalala} we get 
\begin{equation}\label{caloasasa}
\alpha_k^{3,j} \ge \int_{Q_{r_j,\nu_{x_j}}(x_j)\cap T_{-t_{k,j}^\delta r}} \varphi(x,\nu_{x_j})\d\cH^{n-1}
- (\delta + c'\sqrt{\delta})r_j^{n-1} 
\end{equation}
for all $k>k_\delta^{3,j}.$ 
Moreover, by the choice of $t_{k,j}^\delta,$ \eqref{stupid_tdelta} and \eqref{finsler_norm},
\begin{align*}
\mu_{B_k^\delta,u_k^\delta} (\cl{D_k^j}) \le & \int_{Q_{r_j,\nu_{x_j}}(x_j)\cap T_{-t_{k,j}^\delta r_j}} \varphi(x,\nu_{x_j})\d\cH^{n-1}
 +
\int_{Q_{r_j,\nu_{x_j}}(x_j)\cap A_k^{(1)}\cap T_{t_{k,j}^\delta r_j}} \varphi(x,\nu_{x_j})\d\cH^{n-1} \\
& +
\int_{\p^*D_k^j\setminus [T_{-t_{k,j}^\delta r_j}\cup T_{t_{k,j}^\delta r_j}]} \varphi(x,\nu_{D_k^j})\d\cH^{n-1}\\
\le & \int_{Q_{r_j,\nu_{x_j}}(x_j)\cap T_{-t_{k,j}^\delta r_j}} \varphi(x,\nu_{x_j})\d\cH^{n-1}
+ 4\bound_2\sqrt{\delta}r_j^{n-1} + 
2\bound_2t_{k,j}^\delta r_j^{n-1}.
.
\end{align*}
Now using $t_{k,j}^\delta \le 2\sqrt\delta$ and \eqref{rsdadad} in this estimate and combining with \eqref{caloasasa} and obvious inequality  $
\mu_{A_k,u_k}(\cl{D_k^j}) \ge \alpha_k^{3,j}
$ (recall that $\cl{D_k^j}\cap\Sigma=\emptyset$) we get 
$$
\mu_{A_k,u_k}(\cl{D_k^j})  - \mu_{B_k^\delta,u_k^\delta} (\cl{D_k^j}) \ge - c_3^*\sqrt{\delta}\cH^{n-1}( Q_{r_j,\nu_j}(x_j)\cap \p^*F^{h_j}) 
$$
for some $c_3^*$ depending only on $n$ and $\bound_1,\bound_2.$
Summing these inequalities in $j$ gives
\begin{align}\label{estimaso_I3sla}
I_3\ge c_3^*\sqrt{\delta}\sum\limits_{h=0}^m \cH^{n-1}(  \p^*F^h)
\end{align}
for all $k>k_\delta^3=\max_jk_\delta^{3,j}.$

Including \eqref{estimatos_I1osos}, \eqref{estimos_I2oplas} and \eqref{estimaso_I3sla} in \eqref{don_oomar_imala} and using \eqref{first_error0901} in \eqref{juhuggza} we obtain
\begin{align*}
\cS(A_k,u_k) - \cS(B_k^\delta,u_k^\delta) \ge 
-c^* \sqrt\delta \Big(1 + \cH^{n-1}(J_{u_k}) + \cH^{n-1}(\p^*A_k) + \sum\limits_{i=0}^m\cH^{n-1}(\p^*F^i) \Big)
\end{align*}
for all $k>k_\delta=\max\{k_\delta^1,k_\delta^2,k_\delta^3\}$. 
Finally, since the elastic energy density is nonnegative  and invariant w.r.t. to additive piecewise rigid displacements,
$$
\cW(A_k,u_k) \ge \cW(B_k^\delta,v_k^\delta),
$$
and hence \eqref{error_estimate_Bv} follows.
\end{proof}

 Theorems \ref{teo:lower_semicontinuity} and \ref{teo:compactness} together with Proposition \ref{prop:fusco} imply that the minimum problem \eqref{min_prob_globals} is solvable.

\begin{proof}[\textbf{Proof of Theorem \ref{teo:global_existence}}]
Fix any $\lambda>0$ and let $\{(A_k,u_k)\}\subset\admissible$ be a minimizing sequence for $\cF^\lambda.$ Then $\sup_k \cF(A_k,u_k)<+\infty,$ and hence  by Theorem \ref{teo:compactness}, there exists a not relabelled subsequence $\{(A_k,u_k)\},$ a sequence $\{(B_k,v_k)\}\subset\admissible$ and $(A,u)\in\admissible$ such that $(B_k,v_k)\overset{\tau_\admissible}{\to} (A,u),$ $|A_k\Delta B_k|\to0$ and
\begin{equation}\label{shapatella1818}
\liminf\limits_{k\to+\infty}\,\cF(A_k,u_k) \ge \liminf\limits_{k\to+\infty}\cF(B_k,v_k) \ge \cF(A,u). 
\end{equation}
Since the map $E\mapsto ||E| - \fm|$ is $L^1(\R^n)$-continuous, from \eqref{shapatella1818} it follows that 
$$
\liminf\limits_{k\to+\infty}\,\cF^\lambda(A_k,u_k) \ge \liminf\limits_{k\to+\infty} \cF^\lambda(B_k,v_k) \ge \cF^\lambda(B,v). 
$$
Hence, $(B,v)$ is a minimizer of $\cF^\lambda.$
By Proposition \ref{prop:fusco}, there exists $\lambda_0>0$ such that for $\lambda>\lambda_0$ every minimizer $(A,u)$ of $\cF^\lambda$ satisfies the  volume constraint $|A|=\fm.$ Thus,
$(A,u)$ solves also the problem \eqref{min_prob_globals}.
Conversely, if $(A,u)$ solves \eqref{min_prob_globals}, then for $\lambda>\lambda_0,$
$$
\begin{aligned}
\min_{(B,v)\in\admissible,\,|B|=\fm}\cF(B,v) = & \cF(A,u) = \cF^\lambda(A,u) \ge 
\min_{(B,v)\in\admissible}\cF^\lambda(B,v)\\
= & \min_{(B,v)\in\admissible,\,|B|=\fm}\cF^\lambda(B,v) =\min_{(B,v)\in\admissible,\,|B|=\fm}\cF(B,v),
\end{aligned}
$$
and hence  $(A,u)$ is a minimizer of $\cF^\lambda.$
\end{proof}

\subsection{Compactness in $\admissible_p$ and $\admissible_{\rm Dir}$}\label{subsec:extension_CLS}

Recalling the definition of $\tau$-convergence in \eqref{tau_convergence_CACA} and using \eqref{lower_boundas_papap} and the compactness result \cite[Theorem 1.1]{ChC:2020_jems}, we can also solve the  $\tau$-compactness issue of energy-equibounded sequences in $\admissible_p$ and $\admissible_{\rm Dir}:$
\begin{itemize}[left=0pt]
 \item[--] if $\{(A_k,u_k)\}\subset \admissible_p$ is arbitrary sequence with $\sup_k \cF_p(A_k,u_k) <+\infty,$ then repeating the same arguments in the proof of Proposition  \ref{prop:pass_to_good_seq_compacte} we construct a not relabelled subsequence, the set $G_k^\delta,$ and numbers $s_\delta$ and $k_\delta$ satisfying \eqref{good_cut_sk}-\eqref{per_est_Gk} such that the configuration $(B_k^\delta,v_k^\delta)\in\admissible_p$ given by \eqref{def_BBkdelta} and \eqref{def_vvkdelta} satisfies
 $$
 \cS(A_k,u_k) - \cS(B_k^\delta,u_k^\delta) \ge 
-c^* \sqrt\delta \Big(1 + \cH^{n-1}(J_{u_k}) + \cH^{n-1}(\p^*A_k) + \sum\limits_{i=0}^m\cH^{n-1}(\p^*F^i) \Big).
 $$
 Then by \eqref{lower_boundas_papap},
 $$
 \cW(A_k,u_k) \ge \cW(B_k^\delta,u_k) + \int_{G_k^\delta}W_p(x,\str{v_k^\delta})\d x \ge 
 \cW(B_k^\delta,u_k) - \int_{G_k^\delta}|f|\d x.
 $$
 Since $f\in L^1(\Omega\cup S),$ by \eqref{vol_est_Gk} and the absolute continuity of the Lebesgue integral, we have
 \begin{equation}\label{gododosasas}
  \cW(A_k,u_k) \ge \cW(B_k^\delta,u_k) + o_\delta,  
 \end{equation}
 where $o_\delta\to0$ as $\delta\to0.$ Now the proof of the compactness in $\admissible_p$ runs exactly the same as Theorem \ref{teo:compactness} using \eqref{gododosasas} in place of  \eqref{elastic_bahooo};
 
 \item[--] if $\{(A_k,u_k)\}\subset \admissible_p$ is arbitrary sequence with $\sup_k \cF_{\rm Dir}(A_k,u_k) <+\infty,$ then  
 by \cite[Theorem 1.1]{ChC:2020_jems}, in the proof of Theorem \ref{teo:compactness} we will have only two sets $F^0$ and $F^1$ partitioning $A:$ the sequence $u_k$ converges a.e.\  in $F^1$  (up to a subsequence) and $|u_k|\to+\infty$ a.e.\  in $F^0.$ In particular, due to the Dirichlet condition for $u_k$ in $S,$ we do not need to add any rigid displacements, and then the proofs run as in $\admissible_p$.
\end{itemize}

The $\tau$-compactness in $\admissible_p$ (resp. $\admissible_{\rm Dir}$) and the $\tau$-lower semicontinuity of $\cF_p$  (resp. $\cF_{\rm Dir}$) imply that for any $\lambda>0$ there exists a minimizer of $\cF_p^\lambda$  (resp. $\cF_{\rm Dir}^\lambda$). Now obverving that the proof of Proposition \ref{prop:fusco} works also in $\admissible_p$ and $\admissible_{\rm Dir}$ (see Remark \ref{rem:extension_fuscortr}) we conclude that both minimum problems \eqref{min_prob_plow} and \eqref{min_prob_dir_plow}  admit  a solution.

\section{Decay estimates}\label{sec:decay_estimates} 

This section is devoted to the proof of the following density estimates for minimizers of $\cF$. 

\begin{theorem}[\textbf{Density estimates}]\label{teo:density_estimates}
There exist $\varsigma_*=\varsigma_*(\bound_3,\bound_4)\in(0,1)$ and $R_*=R_*(\bound_1,\bound_2,\bound_3,\bound_4)>0,$ where $\bound_i$ are given by \eqref{finsler_norm} and \eqref{hyp:elastic}, with the following property. Let $(A,u)\in\admissible$ be any minimizer of $\cF$ in $\admissible$ such that $\Omega\cap\p^*A\subset_{\cH^{n-1}} J_u$ and $\int_{\Omega\setminus A}|\str{u}|\d x=0,$ and let 
\begin{equation}\label{def_Ju_starrr}
J_u^* :=\{x\in J_u:\,\,\theta(J_u,x)=1\}.  
\end{equation}
Then for any $x\in\Omega$ and $r\in(0,\min\{1,\dist(x,\p \Omega)\})$
\begin{equation}\label{min_seq_density_up}
\frac{\cH^{n-1}(Q_r(x)\cap J_u)}{r^{n-1}} \le \frac{4n\bound_2 + \lambda_0}{\bound_1}. 
\end{equation}
Moreover, if $x\in \Omega\cap \cl{J_u^*}$ and $r\in(0,R_*)$ with $Q_r(x)\subset \Omega,$ then  
\begin{equation}\label{min_seq_density_low}
\frac{\cH^{n-1}(Q_r(x)\cap J_u)}{r^{n-1}} \ge \varsigma_*. 
\end{equation} 
In particular, 
\begin{equation}\label{essental_closed_jumpoas}
\cH^{n-1}(\Omega\cap [\cl{J_u^*}\setminus J_u^*])=0. 
\end{equation} 
\end{theorem}
 
Since $J_u$ is $\cH^{n-1}$-rectifiable, by the rectifiability criterion in \cite[Theorem 2.63]{AFP:2000}, $\cH^{n-1}(J_u\setminus J_u^*)=0.$ Thus, if we remove a $\cH^{n-1}$-negligible set from $J_u,$ then \eqref{essental_closed_jumpoas} implies that the jump set of $u$ is essentially closed in $\Omega$.
 
To prove Theorem \ref{teo:density_estimates} we follow the arguments of \cite[Section 3]{HP:2021_arxiv}.  First, we introduce the local version $\cF(\cdot;\openset):\admissible\to\R$ of $\cF$ in open sets  $\openset\subset\Omega$ as
\begin{equation}\label{local_cF}
\cF(A,u;\openset):=\cS(A,u;\openset) + 
\cW(A,u; \openset), 
\end{equation}
where $\cS(\cdot;\openset)$ and $\cW(\cdot; \openset)$ are the local versions of the surface and the elastic energies, i.e.,
$$
\cS(A,u;\openset):= \int_{\openset \cap \p^*A} \varphi(y,\nu_A)d\cH^{n-1} 
+ 2\int_{\openset \cap A^{(1)}\cap J_u} \varphi(y,\nu_A)\,\d\cH^{n-1}  
$$
and 
$$
\cW(A,u; \openset) = \int_{\openset\cap A} \C(y)\str{u}:\str{u}\, \d y. 
$$
Next, we introduce the notion of quasi-minimizers.
\begin{definition}[\textbf{$\Theta$-minimizers}]
Given $\Theta\ge 0,$ the configuration $(A,u)\in\admissible $
is a local {\it $\Theta$-minimizer} of $\cF:\admissible \to\R$ in $\openset$ if
\begin{equation*} 
\cF(A,u;\openset) \le \cF(B,v;\openset) + \Theta|A\Delta B|
\end{equation*}
whenever $(B,v)\in\admissible $ with $A\Delta B\strictlyincluded \openset$ and $\supp(u-v)\strictlyincluded \openset.$
\end{definition}

For any $(A,u)\in\admissible $ and any open set $\openset \strictlyincluded \Omega$ let 
\begin{equation}\label{minimal_with_dirixle}
\Phi(A,u;\openset):= 
\inf\Big\{ 
\cF(B,v;\openset):\,\, (B,v)\in\admissible,
B\Delta A\strictlyincluded \openset,\,\, 
\supp (u-v)\strictlyincluded \openset 
\Big\},
\end{equation}
 and let 
\begin{equation}\label{deviation}
\Psi(A,u;\openset): = \cF(A,u;\openset) - \Phi(A,u;\openset) 
\end{equation}
be the \emph{deviation of $(A,u)$ from minimality} in $\openset.$ 

The following proposition is a generalization to our setting of 
\cite[Theorem 4]{CCI:2017} established for the Griffith model. 

\begin{proposition}\label{prop:conti_prop_3.4}
Let $Q_R(x_0)\strictlyincluded\Omega.$ Consider sequences of  Finsler norms $\{\varphi_h\}$ and ellipticity tensors $\{\C_h\}$ such that $\{\C_h\}$ is equicontinuous in $\cl{Q_R(x_0)}$ and there exist $d_3,d_4,d_5>0$ with  
\begin{equation}\label{elastic_bound_1}
d_3M:M\le \C_h(x)M:M \le d_4M:M\quad\text{for all $(x,M)\in\overline{Q_R(x_0)}\times\mtwo$} 
\end{equation}
and 
\begin{equation}\label{anisotropic_bound_1}
\inf\limits_{(x,\nu)\in \cl{Q_R(x_0)}\times \S^{n-1}}\,\phi_h(x,\nu) \ge d_5 \sup\limits_{(x,\nu)\in \cl{Q_R(x_0)}\times \S^{n-1}}\,\phi_h(x,\nu).
\end{equation}
Let us define
$\cF_h$ and $\Psi_h$ in $\admissible$ as in \eqref{local_cF} and \eqref{deviation} respectively with  $\varphi_h$ and $\C_h,$ in places of $\varphi$ and $\C.$
Let also $\{(A_h,u_h)\}\subset\admissible$ be such that
\begin{subequations}
\begin{align}
& \int_{Q_R(x_0)\setminus A_h} |\str{u_h}|\d x=0, \label{zero_eelastic_energy}\\
& M:=\sup\limits_{h\ge1} \cF_h(A_h,u_h;Q_R(x_0))<\infty,  \label{energy_bound_uniform}\\
& \lim\limits_{h\to\infty} \Psi_h(A_h,u_h;Q_R(x_0)) =0,   \label{seq_almost_min}\\
& \lim\limits_{h\to\infty} \cH^{n-1} (Q_R(x_0)\cap J_{u_h}) = 0,   \label{set_has_no_boundary}\\
& Q_R(x_0)\cap \p^*A_h\subset_{\cH^{n-1}} J_{u_h}. \label{red_bound_jumpset} 
\end{align}
\end{subequations}
Then there exist $u\in H^1(Q_R(x_0) ),$ an elasticity tensor $\C\in C^{0}(\overline{Q_R(x_0)};\mtwo)$ and sequences $\{a_j\}$ of rigid displacements and subsequences  $\{(A_{h_j},u_{h_j})\}$, $\{\varphi_{h_j}\}$ and $\{\C_{h_j}\}$  such that 
\begin{itemize}
\item[\rm(i)] $\C_{h_j}\to \C$ uniformly in $\overline{Q_R(x_0)}$
and 
$$
w_j:=u_{h_j} - a_j \to u \,\,\text{a.e.\  in $Q_R(x_0)$}\quad \text{and}
\quad \str{w_j}\wk\str{u}\,\,\text{in $L^2(Q_R(x_0))$}
$$
as $j\to \infty;$

\item[\rm(ii)] for all $v \in u+H_0^1 (Q_R(x_0))$ 
\begin{equation}\label{local_minimal_ekanku}
\int_{Q_R(x_0)} \C(y)\str{u}:\str{u}\,\d y\le 
\int_{Q_R(x_0)} \C(y)\str{v}:\str{v}\,\d y; 
\end{equation}

\item[\rm(iii)]  for any $r\in(0,R]$ 
\begin{equation}\label{functional_conv_o}
\lim\limits_{j\to\infty} \cF_h(A_{h_j},u_{h_j};Q_r(x_0)) = \int_{Q_r(x_0)} \C(x)\str{u}:\str{u}\,\d x. 
\end{equation}
\end{itemize}
\end{proposition}

\begin{proof}
Without loss of generality, we assume $R=1$ and $x_0=0.$ 
Also by \eqref{set_has_no_boundary}, we may assume $\cH^{n-1} (Q_1\cap J_{u_h})<1/4$ for any $h.$
Let
$$
b_h':=\inf\limits_{(x,\nu)\in \cl{Q_R(x:0)}\times \S^{n-1}}\,\phi_h(x,\nu),\qquad b_h'':= \sup\limits_{(x,\nu)\in \cl{Q_R(x:0)}\times \S^{n-1}}\,\phi_h(x,\nu) 
$$
so that by \eqref{anisotropic_bound_1},
\begin{equation}\label{new_anis_b}
d_5 b_h'' \le b_h'\le b_h''\qquad \text{for any $h.$}
\end{equation}
By \cite[Proposition 2]{CCF:2016} and \eqref{elastic_bound_1}, there exist a constant $c_o$ (depending only on $n$ and $d_3$) and sequences $\{\omega_h\}$ of a measurable subsets of $Q_1$ with $|\omega_h|\le c_o\cH^{n-1} (Q_1\cap J_{u_h})$ and $\{a_h\}$ of  rigid displacements such that 
\begin{equation}\label{L2norm_chegarara}
\int_{Q_1\setminus \omega_h} |u_h - a_h|^2\,\d x \le c_o\int_{Q_1}\C_h(x)\str{u_h}:\str{u_h}\,\d x.  
\end{equation}
By \eqref{zero_eelastic_energy} and \eqref{energy_bound_uniform}, $\|(u_h - a_h) \chi_{Q_1\setminus \omega_h}\|_{L^2(Q_1)} \le Mc_o,$ and thus there exist  $u\in L^2(Q_1)$ and a not relablled subsequence such that
$
(u_h - a_h) \chi_{Q_1\setminus \omega_h} \wk  \tilde u 
$
in $L^2(Q_1).$  Since $|\omega_h|\to0,$ the set  
$$
F:=\{y\in Q_1:\,\,\limsup\limits_{h\to\infty}|u_h(y) - a_h(y)|=+\infty\}
$$
satisfies $|F|=0.$  Furthermore, by \eqref{zero_eelastic_energy}, \eqref{elastic_bound_1} and \eqref{energy_bound_uniform} as well as the equality $J_{u_h}=J_{u_h - a_h},$ %
$$
\sup\limits_{h\ge1} \int_{Q_1} |\str{(u_h-a_h)}|^2\,\d x +\cH^{n-1} (Q_1\cap J_{u_h-a_h}) < \frac{M}{d_3} + \frac{1}{4},
$$
and hence  by \cite[Theorem 1.1]{ChC:2020_jems}, there  exist  a not relabelled subsequence $\{u_h-a_h\}$
and $u\in GSBD^2(Q_1)$  such that 
\begin{align}
& u_h-a_h \to  u\qquad \text{a.e.\  in $Q_1$} \label{safafaess1}\\
& \str{(u_h-a_h)} \wk \str{u}\qquad \text{in $L^2(Q_1;\mtwo),$} \label{safafaess2}\\
& \cH^{n-1} (Q_1 \cap J_u) \le
\liminf\limits_{h\to+\infty} \cH^{n-1} (J_{u_h}) =0.\label{safafaess3} 
\end{align}
Since the weak limit and the pointwise limit coincide (see e.g., \cite[page 266]{DiBen:2002_book}), $\tilde u = u$ a.e.\  in $Q_1.$ Moreover,   \eqref{L2norm_chegarara}, \eqref{safafaess1} and the Fatou's Lemma imply $u\in L^2(Q_1)$ and also by \eqref{safafaess3}, one has $\cH^{n-1} (J_u)=0.$ Thus, by Lemma \ref{lem:gsbd_is_sobolev}, $u\in H^1(Q_1 )$.
Since our elastic energy is invariant under additive rigid displacements, without loss of generality, we may assume $a_h=0$ for any $h\ge1.$

Next we prove \eqref{local_minimal_ekanku}. Let $v\in H^1(Q_1 )$ be such that $\supp(u-v)\strictlyincluded Q_r$ for some $r\in(0,1).$ Fix $r''<r'<r$ and let $\psi\in C_c^1 (Q_r;[0,1])$ be a cut-off function with 
$\{0<\psi<1\}\subset \{u=v\}\cap Q_{r'}$ and $\supp(u-v)\subseteq \{\psi\equiv 1\}\subseteq Q_{r''}.$ 
By \eqref{set_has_no_boundary} and \cite[Theorem 3]{CCI:2017}, there exist  a positive constant $c>0$ (depending only on $n,$ $d_3$ and $d_4$), a function $\tilde v_h\in GSBD^2(Q_1),$ $r_h\in(r-\delta_h,r)$ with
\begin{equation}\label{defi_delta_h}
\delta_h:=\sqrt[2n]{\cH^{n-1} (J_{u_h})},
\end{equation}
and  a Lebesgue measurable set $\tilde \omega_h\subset Q_{r_h}$ such that 
\begin{itemize}
\item[($\rm a_1$)] $\tilde v_h\in C^\infty(Q_{r-\delta_h} ),$ $\tilde v_h = u_h$ in $Q_1\setminus Q_{r_h},$ and  
$$
\cH^{n-1} (J_{u_h}\cap \p Q_{r_h}) = \cH^{n-1} (J_{\tilde v_h}\cap \p Q_{r_h}) = 0;
$$

\item[($\rm a_2$)] $\cH^{n-1}(J_{\tilde v_h}\setminus J_{u_h})<c \delta_h \cH^{n-1}(J_{u_h} \cap (Q_r\setminus Q_{r-\delta_h}));$

\item[($\rm a_3$)]  $|\tilde \omega_h|\le c\delta_h^2\cH^{n-1} (Q_{r_h}\cap J_{u_h})$ and by \eqref{elastic_bound_1},
\begin{equation}\label{sdfsdgs}
\int_{Q_r\setminus \tilde \omega_h} |\tilde v_h - u_h|^2\d x \le c\delta_h^4 \int_{Q_r} \C_h(x)\str{u_h}:\str{u_h}\d x; 
\end{equation}

\item[($\rm a_4$)] if $\eta\in \Lip(Q_1;[0,1]),$ then 
\begin{align}\label{eta_constant}
\int_{Q_r} \eta \C_h(x)\str{\tilde v_h}:  \str{\tilde v_h}\d x \le
\int_{Q_r} \eta \C_h(x)\str{u_h}:\str{u_h} \d x
+
c\delta_h^s[1+\Lip(\eta)]\int_{Q_r} \C_h(x)\str{u_h}:\str{u_h}\d x 
\end{align}
for some $s\in(0,1)$ independent of $h.$ 
\end{itemize}
By ($\rm a_1$), $\tilde v_h\in H^1 (Q_{r'} )$ and $\supp(\tilde v_h - u_h)\strictlyincluded Q_r$ for all sufficiently  large $h.$
By \eqref{safafaess1}, \eqref{sdfsdgs} and the relation $\delta_h^{2n}=\cH^{n-1}(Q_1\cap J_{u_h})\to0,$ we have also
$\tilde v_h\to u$ a.e.\  in $Q_1.$  Let us define
\begin{equation}\label{wdsaw}
v_h:=(1-\psi) \tilde v_h+ \psi v. 
\end{equation}
Then $(A_h,v_h)$ is an admissible configuration for $\Phi_h(A_h,u_h;Q_1)$ in  \eqref{minimal_with_dirixle}. Therefore from \eqref{seq_almost_min} and the definition of deviation it follows that 
\begin{equation}\label{kdlfjakd}
\cF_h(A_h,u_h;Q_1) \le \cF_h(A_h,v_h;Q_1) + o(1), 
\end{equation}
where $o(1)\to0$ as $h\to\infty.$ Note that 
by ($\rm a_1$), ($\rm a_2$),  \eqref{new_anis_b} and \eqref{energy_bound_uniform},
\begin{align*}
\cS(A_h,v_h;Q_1) -\cS(A_h,u_h;Q_1)
= & \int_{A_h^{(1)}\cap J_{\tilde v_h}} \phi_h(x,\nu_{J_{\tilde v_h}}^{})\d\cH^{n-1} -  \int_{A_h^{(1)}\cap J_{u_h}}   \phi_h(x,\nu_{J_{u_h}}^{})\d\cH^{n-1}\\
\le & \int_{A_h^{(1)}\cap (J_{\tilde v_h}\setminus J_{u_h})\cap (Q_r\setminus Q_{r-\delta_h})} \phi_h(x,\nu_{J_{\tilde v_h}}^{})\d\cH^{n-1} \\
\le & b_h'' \cH^{n-1}(J_{\tilde v_h}\setminus J_{u_h}) \le  cb_h''\delta_h \cH^{n-1}(J_{u_h}\cap (Q_r\setminus Q_{r-\delta_h})) \\
\le & \frac{c\delta_h}{d_5} \cS(A_h,u_h;Q_1) \le \frac{Mc \delta_h}{d_5}.
\end{align*}
This estimate, \eqref{kdlfjakd} and the definition of localized elastic energy imply
\begin{equation}\label{almost_minimality_of_u_h}
\int_{A_h\cap Q_1} \C_h(x)\str{u_h}:\str{u_h}\d\cH^{n-1} \le 
\int_{A_h\cap Q_1} \C_h(x)\str{v_h}:\str{v_h}\d\cH^{n-1} + o(1)
\end{equation}
as $h\to+\infty.$

Next we estimate the integral in the right-hand-side of \eqref{almost_minimality_of_u_h}. By \eqref{wdsaw},
$$
\str{v_h}= (1-\psi)\str{\tilde v_h} + \psi \str{v} + \nabla \psi\odot (v- \tilde v_h),
$$
where $X\odot Y = (X\otimes Y +Y\otimes X)/2.$ Since $\tilde v_h\to u$ a.e.\  in $Q_r$ and $u=v$ in $Q_r\setminus Q_{r'},$ one has $v_h\to v$ a.e.\  in $Q_1.$

We claim that $\tilde v_h\to u$ strongly in $L_\loc^2(Q_r).$ Indeed, fix any $\rho\in(0,r).$ By ($\rm a_1$), $\tilde v_h\in H^1 (Q_\rho).$ Moreover, by \eqref{elastic_bound_1},  \eqref{energy_bound_uniform}  and  \eqref{eta_constant} (applied with $\eta=1$),
$$
d_3\int_{Q_\rho} |\str{\tilde v_h}|^2\d x\le d_3\int_{Q_r} |\str{\tilde v_h}|^2\d x \le C \int_{Q_r} \C_h(x) \str{u_h}:\str{u_h}\d x\le CM
$$
for some constant $C>0$ independent of $h.$  Furthermore, by the Poincar\'e-Korn inequality, for each $h$ there exist a rigid displacement $e_h$ (possibly depending also on $\rho$) such that
$$
\|\tilde v_h - e_h\|_{H^1(Q_\rho)} \le \int_{Q_\rho} |\str{\tilde v_h}|^2 \d x \le \frac{MCC'}{d_3},
$$
and hence the Rellich-Kondrachov Theorem implies the existence of $w\in H^1(Q_\rho )$ and not relabelled subsequence such that
$\tilde v_h - e_h \to w$ in $L^2(Q_\rho ).$ Since $\tilde v_h\to u$ a.e.\  in $Q_1,$ we have $e_h\to u-w,$ and thus $e:=u-w$ is also a rigid displacement. Then
$$
\begin{aligned}
\limsup\limits_{h\to\infty} \|\tilde v_h - u\|_{L^2(Q_\rho)}^{}
\le 
\limsup\limits_{h\to\infty} \|\tilde v_h -e_h - w\|_{L^2(Q_\rho)}^{} 
+
\limsup\limits_{h\to\infty} \|e_h + (w-u)\|_{L^2(Q_\rho)}^{} =0, 
\end{aligned}
$$
and the claim follows.

Since $u=v$ out of $\{\psi=1\},$ the claim implies 
$\tilde v_h\to v$  strongly in $L^2(\{0<\psi<1\}),$ and hence,
\begin{align}\label{sgsgsg}
\lim\limits_{h\to\infty} \int_{Q_r} |\nabla \psi\odot (v- \tilde v_h)\big|_{A_h}|^2 \d x \le & 
\liminf\limits_{h\to\infty}  \int_{\{0<\psi <1\}} |\nabla \psi\odot (v- \tilde v_h)|^2 \d x= 0. 
\end{align}
Thus, by definition \eqref{wdsaw} of $v_h,$
\begin{align}\label{rtetyeye}
\int_{Q_r\cap A_h}\C_h \str{v_h}:\str{v_h} & \d x =   \int_{Q_r\cap A_h} (1-\psi)^2\C_h \str{\tilde v_h}:\str{\tilde v_h}\d x
+ \int_{Q_r\cap A_h}\psi^2\C_h \str{v}:\str{v}\d x\nonumber \\
& + \int_{Q_r\cap A_h} \C_h (\nabla \psi\odot (v- \tilde v_h)) :( \nabla \psi\odot (v- \tilde v_h))\d x\nonumber \\
& + \int_{Q_r\cap A_h} (1-\psi) \C_h \str{\tilde v_h}:(\nabla \psi\odot (v- \tilde v_h))\d x\nonumber \\
& + 
\int_{Q_r} \psi \C_h\str{v}:
(\nabla \psi\odot (v- \tilde v_h))\d x\nonumber \\
=& \int_{Q_r\cap A_h} (1-\psi)^2\C_h \str{\tilde v_h}:\str{\tilde v_h}\d x
+ \int_{Q_r\cap A_h}\psi^2\C_h \str{v}:\str{v}\d x +
o(1)\nonumber\\
\le & 
\int_{Q_r\cap A_h} (1-\psi)^2\C_h \str{u_h}:\str{u_h}\d x 
+ \int_{Q_r\cap A_h}\psi^2\C_h \str{v}:\str{v}\d x + o(1), 
\end{align}
where in the second equality we use \eqref{energy_bound_uniform}, \eqref{eta_constant} with $\eta\equiv 1,$ \eqref{sgsgsg}, \eqref{elastic_bound_1} and   the H\"older inequality, while in the last inequality we use \eqref{eta_constant} with $\eta=(1-\psi)^2$  and \eqref{set_has_no_boundary}. 
Now  combining \eqref{rtetyeye} with \eqref{almost_minimality_of_u_h}   we get
\begin{align}\label{djksdfkj}
\int_{Q_r}(2\psi -\psi^2)\C_h \str{u_h}:\str{u_h}\d x \le 
\int_{Q_r}\psi^2\C_h \str{v}:\str{v}\d x + o(1). 
\end{align}
Since $\{\C_h\}$ is equibounded (see \eqref{elastic_bound_1}) and equicontinuous,  by the Arzela-Ascoli Theorem, there exist  a (not relabelled) subsequence and an elasticity tensor $\C\in C^0(Q_1;\mtwo)$ such that  $\C_h\to \C$ uniformly in $Q_1.$ Hence, 
letting $h\to\infty$ in \eqref{djksdfkj}, and using \eqref{safafaess2} and the convexity of the elastic energy, we obtain
\begin{equation}\label{shapatoyoqcha}
\int_{Q_r}(2\psi -\psi^2) \C(y)\str{u}:\str{u}\,\d y \le 
\int_{Q_r} \psi^2 \C(y)\str{v}:\str{v}\,\d y.
\end{equation}
By the choice of $\psi,$ \eqref{shapatoyoqcha} implies 
\begin{equation}\label{ajsdhjss}
\int_{Q_{r''}} \C(y)\str{u}:\str{u}\,\d y \le 
\int_{Q_r} \C(y)\str{v}:\str{v}\,\d y. 
\end{equation}
Since $r''$ is arbitrary, letting $r''\nearrow r$ we deduce that \eqref{ajsdhjss} holds also with $r''=r.$   Since $\supp(u-v)\strictlyincluded Q_r$, this implies \eqref{local_minimal_ekanku}. 

It remains to prove \eqref{functional_conv_o}.
Taking $v=u$ in \eqref{djksdfkj} and using $0\le \psi\le1$ and $\psi\big|_{Q_{r''}}=1$
we have
\begin{align*} 
\int_{Q_{r''}}\C \str{u}:\str{u}dx\le
\liminf\limits_{h\to\infty} \int_{Q_{r''}}\C_h \str{u_h}:\str{u_h}\d x
\le
\limsup\limits_{h\to\infty} \int_{Q_{r''}}\C_h \str{u_h}:\str{u_h}\d x \le
\int_{Q_r}\C \str{u}:\str{u}\d x. 
\end{align*}
Since $r''$ is arbitrary, letting $r''\nearrow r $ gives
\begin{equation}\label{elastic_energy_convergence} 
\lim\limits_{h\to\infty} \int_{Q_r}\C_h \str{u_h}:\str{u_h}\d x =
\int_{Q_r}\C \str{u}:\str{u}\d x.  
\end{equation}

In view of \eqref{elastic_energy_convergence},  to prove \eqref{functional_conv_o} it suffices to establish
\begin{equation}\label{surface_jugoldi}
\lim\limits_{h\to\infty}  \cS_h(A_h;Q_r) =0 \quad\text{for any $r\in(0,1).$ }
\end{equation}
By \eqref{red_bound_jumpset}, $Q_1\cap \p^*A_h\subset J_{u_h}$ up to an $\cH^{n-1}$-negligible set. Thus, by \eqref{set_has_no_boundary} and the relative isoperimetric inequality, up to a subsequence, either
\begin{equation}\label{set_disappear}
\lim\limits_{h\to\infty} |Q_1\cap A_h| =0 
\end{equation}
or 
\begin{equation}\label{complement_disappear}
\lim\limits_{h\to\infty} |Q_1\setminus A_h| =0. 
\end{equation}
We claim that there exists a not relabelled subsequence $\{A_h\}$ such that for a.e.\  $t\in(0,1)$ 
\begin{equation}\label{set_in_cube_bound0}
\lim\limits_{h\to\infty}   \int_{A_h \cap \p Q_t} \phi_h(x,\nu_{Q_t})\d \cH^{n-1} =0
\end{equation}
if \eqref{set_disappear} holds, and 
\begin{equation}\label{compl_in_cube_bound0}
\lim\limits_{h\to\infty}   \int_{(Q_1\setminus A_h) \cap \p Q_t} \phi_h(x,\nu_{Q_t})\d \cH^{n-1} =0
\end{equation}
if \eqref{complement_disappear} holds. 

We establish only \eqref{set_in_cube_bound0}, the proof of \eqref{compl_in_cube_bound0} being similar. By the coarea formula  (applied with $f(x)=\max\{|x_1|,\ldots,|x_n|\}$),
$$
\lim\limits_{h\to\infty} |Q_1\cap A_h| =\lim\limits_{h\to\infty} \int_0^1\cH^{n-1}(A_h\cap \p Q_t)\d t =0, 
$$
thus, passing to further not relabelled subsequence, $\lim\limits_{h\to\infty} \cH^{n-1}(A_h\cap \p Q_t)=0$ for a.e.\  $t\in(0,1).$  In particular, if $\sup_h b_h''<+\infty,$ then 
\begin{equation}
\limsup\limits_{h\to+\infty} \int_{A_h \cap \p Q_t} \phi_h(x,\nu_{Q_t})\d \cH^{n-1} \le \limsup\limits_{h\to+\infty} \,\, b_h''\cH^{n-1}(A_h \cap \p Q_t)=0.  
\label{surface_esimation} 
\end{equation}
On the other hand, if $b_h''\to+\infty$ (up to a subsequence), then by the coarea formula  and the relative isoperimetric inequality in $Q_1,$
\begin{equation}\label{area_plus_relisop}
b_h''\int_0^1 \cH^{n-1}(A_h\cap\p Q_t)\d t =b_h''|A_h\cap Q_1| \le b_h'' c_n P(A_h, Q_1)^{\frac{n}{n-1}},
\end{equation}
where $c_n>0$ is the relative isoperimetric constant for cubes.
By \eqref{anisotropic_bound_1},
$$
P(A_h, Q_1) \le \frac{1}{a_h}\,\cS(A_h,Q_1) \le \frac{1}{d_5 b_h''}\,\cS(A_h,Q_1) \le \frac{\cF_h(A_h,u_h,Q_1)}{d_5b_h''},
$$
hence by \eqref{area_plus_relisop},
$$
b_h''\int_0^1 \cH^{n-1}(A_h\cap\p Q_t)\d t  \le c_n\,
\Big[\frac{M}{d_5}\Big]^{\frac{n}{n-1}}\,[b_h'']^{-\frac{1}{n-1}}.
$$
This and \eqref{energy_bound_uniform} imply 
$$
\lim\limits_{h\to+\infty} b_h''\int_0^1 \cH^{n-1}(A_h\cap\p Q_t)\d t =0.
$$
In particular, 
$$
\lim\limits_{h\to+\infty} \,b_h'' \cH^{n-1}(A_h \cap \p Q_t)=0\quad\text{for a.e.\  $t\in(0,1).$ }
$$
Now the proof of \eqref{set_in_cube_bound0} follows as  in \eqref{surface_esimation}.

Now we prove \eqref{surface_jugoldi} assuming \eqref{set_disappear}. Given $t\in(r,1)$ for which \eqref{set_in_cube_bound0} holds, define $E_h:=A_h\setminus Q_t.$ Then $(E_h,u_h)$  is an admissible configuration in \eqref{minimal_with_dirixle},  and
thus, 
\begin{equation}\label{tsfde}
\cF_h(A_h,u_h; Q_1) \le \Phi_h(A_h,u_h;Q_1) +o(1) \le \cF_h(E_h,u_h; Q_1) + o(1), 
\end{equation}
where in the first inequality we use \eqref{seq_almost_min} and in the second the definition of $\Phi_h.$
From the definition of $E_h$ and \eqref{tsfde} it follows that
\begin{equation*}
\cS_h(A_h; Q_t) \le \int_{A_h \cap \p Q_t} \phi_h(x,\nu_{Q_t})\d \cH^{n-1} + o(1). 
\end{equation*}
This and \eqref{set_in_cube_bound0} imply \eqref{surface_jugoldi}.

Now suppose that \eqref{complement_disappear} holds. Let $\delta_h$ be defined as in \eqref{defi_delta_h}, and let $\psi,$ $r''< r'<r$ and $v_h$ be as in \eqref{wdsaw} with $v=u$.
Fix any $t\in(r,1)$ for which \eqref{compl_in_cube_bound0} holds, and set $E_h:=A_h\cup \overline{Q_t}.$ Then for sufficiently large $h$ that $(E_h,v_h)$ is an admissible configuration for $\Phi_h(A_h,u_h;Q_1)$ in \eqref{minimal_with_dirixle}. Thus, by \eqref{seq_almost_min},
\begin{equation*} 
\cF_h(A_h,u_h;Q_1) \le \cF_h(E_h,v_h;Q_1) + o(1). 
\end{equation*}
By the definition of $\cF_h$, as in the proof of \eqref{djksdfkj} we can establish
\begin{multline*} 
\cS_h(A_h;Q_t) + \int_{Q_r}(2\psi -\psi^2)\C_h \str{u_h}  :\str{u_h}\d x \\
\le \int_{Q_r}\psi^2\C_h \str{u}:\str{u}\d x + \int_{(Q_1\setminus A_h) \cap \p Q_t} \phi_h(x,\nu_{Q_t})\d \cH^{n-1} + o(1).  
\end{multline*}
Then as in \eqref{shapatoyoqcha}, letting $h\to\infty$  we obtain
\begin{align}\label{shakar_bola}
& \limsup\limits_{h\to\infty} \cS_h(A_h;Q_t) +  \int_{Q_r}(2\psi -\psi^2)\C \str{u}:\str{u}dx
\le 
\int_{Q_r}\psi^2\C \str{u}:\str{u}dx. 
\end{align}
Since $\psi=1$ in $Q_{r''}$ and $|\psi|\le1,$   from \eqref{shakar_bola} it follows that 
\begin{equation*} 
\limsup\limits_{h\to\infty} \cS_h(A_h;Q_t) +    \int_{Q_{r''}} \C \str{u}:\str{u}dx \le \int_{Q_r} \C \str{u}:\str{u}dx. 
\end{equation*}
Now letting $r''\to r$ we get \eqref{surface_jugoldi}.
\end{proof}

Recall that by \cite[Theorem 6.2.1]{Morrey:2008} 
if the elasticity tensor $\C$ is constant and satisfies \eqref{hyp:elastic}, then there exists $C_{\bound_3,\bound_4}>0$ such that every  local minimizer $u\in H^1(Q_1(x_0) )$ of the functional 
\begin{equation}\label{almost_laplacian}
v\in H^1(Q_1(x_0);\R^n)\mapsto \int_{Q_1(x_0)} \C\str{v}:\str{v}\d x 
\end{equation}
is analytic in $Q_1(x_0)$ and satisfies 
\begin{equation}\label{c_gamma}
\int_{Q_r(x_0)}\C\str{u}:\str{u}\,\d x \le 
C_{\bound_3,\bound_4}\, r^n \int_{Q_1(x)} \C\str{u}:\str{u}\,\d x \quad\text{for any $r\in(0,1/2).$}
\end{equation}

Using Proposition \ref{prop:conti_prop_3.4} and repeating similar arguments in \cite{ChC:2019_arxiv} we get the following  decay property of the functional $\cF$.

\begin{proposition}\label{prop:functional_decay}
Assume (H1)-(H3).  Let
\begin{equation}\label{def_tau0}
\tau_0:=\tau_0(\bound_3,\bound_4):= (1 + C_{\bound_3,\bound_4})^{-2},
\end{equation}
where $C_{\bound_3,\bound_4}>0$ is given in \eqref{c_gamma}.
For any $\tau\in(0,\tau_0)$ there exist
$\varsigma = \varsigma(\tau)\in(0,1),$ $\vartheta:=\vartheta(\tau)
>0$ and $R:=R(\tau)>0$ such that if $(A,u)\in\admissible $ satisfies
\begin{align*}
& Q_\rho(x)\cap \p^*A\subseteq J_u,\\
& \int_{Q_\rho(x)\setminus A} |\str{u}|\d x=0,\\
& \cH^{n-1} (Q_\rho(x)\cap J_u)<\varsigma  \rho^{n-1},\\
& \cF(A,u;Q_\rho(x)) \le (1+\vartheta)\Phi(A,u;Q_\rho(x)) 
\end{align*}
for some $Q_\rho(x)\strictlyincluded \Omega$ with $0<\rho<R,$ then 
$$
\cF(A,u;Q_{\tau \rho}(x)) \le \tau^{n-1/2}\,\cF(A,u;Q_\rho(x)).
$$
\end{proposition}

\begin{proof}
 Assume by contradiction that there exist $\tau\in(0,\tau_0)$, positive real numbers $\varsigma_h,\vartheta_h,\rho_h\to0,$ cubes $Q_{\rho_h}(x_h)\strictlyincluded \Omega$, and admissible configurations $(A_h,u_h)\in \admissible$ such that
 \begin{subequations}
\begin{align}
& Q_{\rho_h}(x_h)\cap \p^*A_h\subseteq J_{u_h},  \label{cntr_a0}\\
& \int_{Q_{\rho_h}(x_h)\setminus A_h} |\str{u_h}|\d x=0,\\
& \cH^{n-1} (Q_{\rho_h}(x_h)\cap J_{u_h}) \le \varsigma_h \rho_h^{n-1}, \\
& \cF(A_h,u_h;Q_{\rho_h}(x_h)) \le (1 + \vartheta_h)\Phi(A_h,u_h;Q_{\rho_h}(x_h)), \label{cntr_a1} 
\end{align}
\end{subequations}
but 
\begin{equation}
\cF(A_h,u_h;Q_{\tau\rho_h}(x_h))>\tau^{n-1/2} \cF(A_h,u_h;Q_{\rho_h}(x_h))  \label{contradiction_ass3} 
\end{equation}
for any $h.$ Note that $\cF(A_h,u_h;Q_{\rho_h}(x_h))>0.$

Let us define the rescaled energy $\cF_h(\cdot;Q_1)$ as in \eqref{local_cF} with
$$
\phi_h(y,\nu):= \frac{(\rho_h/2)^{n-1}\varphi(x_h+\frac12\rho_hy,\nu)}{ \cF(A_h,u_h;Q_{\rho_h}(x_h))}
$$
in place of $\varphi(y,\nu)$ and 
$$
 \C_h(y):=\C(x_h+\rho_hy)
$$
in place of  $\C(y)$, for $y\in Q_1$. In view of \eqref{cntr_a0}-\eqref{cntr_a1},   for
$$
E_h:= \sigma_{x_h,\rho_h}(A_h)
$$
(see definition of blow-up map $\sigma_{x,r}$ at \eqref{blow_ups}) and
$$
v_h(y):=\frac{(\rho_h/2)^{\frac{n-2}{2}}\,  u_h(x_h+\frac12\rho_hy)}{\sqrt{\cF(A_h,u_h;Q_{\rho_h}(x_h))}} 
$$
we have 
\begin{align*}
& \cF_h(E_h,v_h;Q_1)=1,\\
& Q_1\cap \p^*E_h \subset_{\cH^{n-1}} J_{v_h},\\
& \int_{Q_1\setminus E_h}|\str{v_h}|\d x=0,\\
& \cH^{n-1} (Q_1\cap \p J_{v_h}) < 2^{n-1} \varsigma_h,\\
& \Psi_h(E_h,v_h;Q_1) \le \vartheta_h \Phi_h(E_h,v_h;Q_1) \le \vartheta_h \cF_h(E_h,v_h;Q_1)=\vartheta_h,
\end{align*}
where $\Phi_h$ and $\Psi_h$ are defined as in \eqref{minimal_with_dirixle} and \eqref{deviation} (with $\varphi_h$ and $\C_h$ in places of $\varphi$ and $\C, $ respectively). By the boundednes of $\Omega,$ there exists $x_0\in \cl{\Omega}$ such that, up to extracting a subsequence, $x_h\to x_0$ as $h\to+\infty.$ In particular, $x_h+ \rho_h y\to x_0$ for every $y\in \cl{Q_1}.$ Then the uniform continuity of $\C$ implies that $\C_h\to \C_0:=\C(x_0)$ uniformly in $ \cl{Q_1}.$ 
Also by \eqref{finsler_norm}, $\phi_h$ satisfies \eqref{anisotropic_bound_1} with $d_5:=\bound_1/\bound_2.$ Thus,  by Proposition \ref{prop:conti_prop_3.4}, there exist $v\in H^1(Q_1 )$ and infinitesimal rigid displacements $a_h$ such that, up to a subsequence,
$$
w_h:=v_h - a_h \to v\quad\text{a.e. in $Q_1,$}
$$ 
$\str{w_h} \wk \str{v}$ in $L^2(Q_1)$ as $h\to+\infty,$
and  
\begin{equation}\label{decay_limit}
\lim\limits_{h\to+\infty} \cF_h(E_h,v_h;Q_r) =\lim\limits_{h\to+\infty} \cF_h(E_h,w_h;Q_r) = \int_{Q_r} \C_0 \str{v}:\str{v}\d x
\end{equation}
for any $r\in(0,1].$
In particular, 
from \eqref{contradiction_ass3} and \eqref{decay_limit} it follows that
\begin{align*}
\int_{Q_\tau} \C_0 \str{v}:\str{v}\d x =\lim\limits_{h\to+\infty}\cF(E_h,v_h;Q_{\tau})
\geq  \tau^{n-1/2} \lim\limits_{h\to+\infty}  \cF(E_h,v_h;Q_{1})
= \tau^{n-1/2} \int_{Q_1} \C_0 \str{v}:\str{v}\d x.
\end{align*}
Since $ \cF_h(E_h,v_h;Q_1)=1,$ by \eqref{decay_limit},  $\int_{Q_1} \C_0\str{v}:\str{v}dx =1.$ Moreover, as $\C_0$ is constant and $v$ is a local minimizer of \eqref{almost_laplacian},  applying \eqref{c_gamma} with $r:=\tau$ and $R:=1$ we get
\begin{align*}
C_{\bound_3,\bound_4}\,\tau^n =   C_{\bound_3,\bound_4}\,\tau^n\int_{Q_1} \C_0 \str{v}:\str{v}\d x \ge \int_{Q_{\tau}}  \C_0 \str{v}:\str{v}\d x
\ge   \tau^{n-1/2}  \int_{Q_1} \C_0 \str{v}:\str{v}\d x = \tau^{n-1/2},
\end{align*}
which contradicts to the assumption $\tau<\tau_0.$
\end{proof}

Proposition \ref{prop:functional_decay} together with the arguments of \cite[Section 4.3]{P:2012} imply the following lower bound for $\cF$.

\begin{proposition} 
\label{prop:lower_density_cG}
Given $\tau\in (0,\tau_0)$, let 
$\varsigma:=\varsigma(\tau)\in(0,1),$ $\vartheta=\vartheta(\tau)>0$ and $R:=R(\tau)>0$ 
be as in Proposition \ref{prop:functional_decay} and for $\Theta>0$ let
$$
R_0:=R_0(\Theta,\tau,\bound_1):= \min\left\{R(\tau),\frac{\bound_1 n \omega_n^{1/n} \vartheta}{\Theta(2+\vartheta)}\right\}.
$$
Let $(A,u)\in\admissible $ be a $\Theta$-minimizer of $\cF$ in $Q_{r_0}(x_0)$ such that 
$\Omega\cap\p^*A\subset_{\cH^{n-1}} J_u$ and $\int_{\Omega\setminus A}|\str{u}|\d x=0.$ Then for any $x\in Q_{r_0}(x_0)\cap \cl{J_u^*},$ where $J_u^*$ is given by \eqref{def_Ju_starrr}, and any cube $Q_\rho(x)\subset Q_{r_0}(x_0)$ with $\rho\in(0,R_0)$ one has 
\begin{equation}\label{lower_dens_functiona}
\cF(A,u;Q_\rho(x)) \ge \bound_1\varsigma \rho^{n-1}. 
\end{equation}
\end{proposition}

\begin{proof}
Let $(C,w),(D,v)\in\admissible $ and $\openset\subset\Omega$ be such that $C\Delta D\strictlyincluded \openset.$ 
By the isoperimetric inequality, the inclusion $\p^*(C\Delta D) \subset \openset\cap (\p^*C\cup\p^*D),$ \eqref{finsler_norm}, the definition of $\cS(\cdot;\openset)$ and the nonnegativity of $\cW(\cdot;\openset)$ one has
\begin{align}\label{relate_isopp}
n\omega_n^{1/n}\,|C\Delta D|^{\frac{n-1}{n}} \le  & P(C\Delta D) \le P(C,\openset) + P(D,\openset)\nonumber\\
\le &\frac{\cS(C,w,\openset) + \cS(D,v,\openset)}{\bound_1} \le \frac{\cF(C,w;\openset) + \cF(D,v;\openset)}{\bound_1},
\end{align}
From \eqref{relate_isopp} and the $\Theta$-minimality of $(A,u)$ in $Q_{r_0}(x_0)$ we deduce 
\begin{align}\label{almost_min_dan_ketamiz}
\!\cF(A,u;Q_r(x)) \le & \cF(B,v;Q_r(x)) + \Theta|A\Delta B|^\frac{1}{n}|A\Delta B|^\frac{n-1}{n}\nonumber \\
\le & \cF(B,v;Q_r(x)) + \frac{\Theta r}{ \bound_1 n\omega_n^{1/n}}\Big(\cF(A,u;Q_r(x)) + \cF(B,v;Q_r(x))\Big)  
\end{align}
for any $Q_r(x)\subset Q_{r_0}(x_0)$ and $(B,v)\in \admissible $ with $A\Delta B\strictlyincluded Q_r(x)$ and $\supp(u-v)\strictlyincluded Q_r(x)$, where in the last inequality we used the inequality $|A\Delta B| \le |Q_r|= r^n.$
By the choice of $R_0,$ if $r\in(0,R_0),$ then $\frac{\Theta r}{\bound_1 n\omega_n^{1/n}} \le \frac{\vartheta}{2+\vartheta},$ and thus, by \eqref{almost_min_dan_ketamiz} 
\begin{equation*}
\cF(A,u;Q_r(x)) \le (1+ \vartheta)\cF(B,v;Q_r(x)).
\end{equation*}
By the arbitrariness of $(B,v),$ this inequality is equivalent to
\begin{equation}\label{almost_minimal_mish}
\cF(A,u;Q_r(x)) \le (1+ \vartheta)\Phi(A,u;Q_r(x)). 
\end{equation}

Now we prove \eqref{lower_dens_functiona}. Fix any $x\in J_u^*;$ for simplicity we suppose that $x=0.$  By contradiction, assume that 
\begin{equation*} 
\cF(A,u;Q_\rho)<  \bound_1\varsigma \rho^{n-1}
\end{equation*}
for some $Q_\rho\strictlyincluded Q_{r_0}(x_0)$ with $\rho\in(0,R_0).$
Then by the nonnegativity of the elastic energy and 
\eqref{finsler_norm} one has
\begin{equation*} 
 \bound_1\varsigma \rho^{n-1} > \cS(A,u;Q_\rho) \ge \bound_1\cH^{n-1} (Q_\rho\cap J_u) 
\end{equation*}
so that 
\begin{equation*} 
\cH^{n-1} (Q_\rho\cap J_u)< \varsigma \rho^{n-1}.
\end{equation*}
By Proposition \ref{prop:functional_decay}  and the definition \eqref{def_tau0} of $\tau_0,$  we have also
$$
\cF(A,u;Q_{\tau \rho}) \le \tau^{n-1/2} \cF(A,u;Q_\rho) < \bound_1\varsigma (\tau\rho)^{n-1} 
$$
so that
$$
\cH^{n-1} (Q_{\tau\rho}\cap J_u)< \varsigma (\tau\rho)^{n-1}. 
$$
Then by induction,  
$$
\cH^{n-1} (Q_{\tau^m\rho}\cap J_u)< \varsigma (\tau^m\rho)^{n-1}  
\quad\text{for any $m\ge1.$}
$$
However, by the definition of $J_u^*,$
$$
1 = \lim\limits_{m\to+\infty} 
\frac{\cH^{n-1} (Q_{\tau^m\rho}\cap J_u)}{(\tau^m\rho)^{n-1}} 
\le \frac{2\bound_1\varsigma }{2\bound_1} = \varsigma<1, 
$$
a contradiction. Hence, \eqref{lower_dens_functiona} holds for any $x\in J_u^*.$  Note that the map $\cF(A,u;\cdot),$ defined for open sets $\openset\strictlyincluded Q_{r_0}(x_0)$ extends to a positive Borel measure in $Q_{r_0}(x_0),$ and therefore, by the continuity of Borel measures, \eqref{lower_dens_functiona} extends also for $x\in Q_{r_0}(x_0)\cap \overline{J_u^*}$.
\end{proof}

Now we are ready to prove \eqref{min_seq_density_up} and \eqref{min_seq_density_low}.

\begin{proof}[\textbf{Proof of Theorem \ref{teo:density_estimates}}]
Let $(A,u)$ be a minimizer of $\cF$  such that $\Omega\cap\p^*A\subset J_u$ and $\int_{\Omega\setminus A}|Eu|\d x=0$ and let $\lambda_0>0$ be given by Theorem \ref{teo:global_existence}.   Since $(A,u)$ is also a minimizer of $\cF^{\lambda_0},$ for any open set $\openset\subset\Omega$ and $(B,v)\in\admissible $ with $A\Delta B\strictlyincluded \openset$ and $\supp(u-v) \strictlyincluded \openset$ we have 
$$
\cF(A,u;\openset) \le \cF(B,v;\openset) + \lambda_0\big||A| - |B|\big| \le \cF(B,v;\openset) + \lambda_0\big|A\Delta B\big|. 
$$
Hence, $(A,u)$ is $\lambda_0$-minimizer of $\cF(\cdot;\Omega)$ in $\Omega.$  

Let us prove \eqref{min_seq_density_up}. Fix  $x\in \Omega$ and let $r_x:=\min\{1,\dist(x,\p \Omega)\}.$ Then by the $\lambda_0$-minimality of $(A,u),$
for any $r\in (0,r_x)$ and $\rho\in(r,r_x)$ 
\begin{equation}\label{trtrt00}
\cF(A,u;Q_\rho(x)) \le\cF(A\setminus \cl{Q_r},u;Q_\rho(x)) + \lambda_0|Q_r(x)\cap A|,   
\end{equation} 
where for shortness  $Q_r:=Q_r(x)$.  Since $\cF(A,u;Q_\rho(x)\setminus \cl{Q_r(x)}) = \cF(A\setminus \cl{Q_r(x)},u;Q_\rho(x)\setminus \cl{Q_r(x)}),$ from \eqref{trtrt00} and the  nonnegativity of $\cF$ we get
$$
\cF(A,u; Q_r(x)) \le \int_{\p Q_r(x)} \varphi(x,\nu_{Q_r(x)})d\cH^{n-1}  +  \lambda_0  r^n.
$$
By \eqref{finsler_norm},
$$
\int_{\p Q_r(x)} \varphi(x,\nu_{Q_r(x)})d\cH^{n-1} \le \bound_2\cH^{n-1}(\p Q_r(x)) = 2n\bound_2r^n  
$$
thus, using $r\le1$ we obtain
\begin{equation}\label{upper_bound_f}
\cF(A,u;\overline{Q_r(x)}) \le (2n\bound_2 + \lambda_0)r^{n-1}. 
\end{equation}
Since $\cW(A,u;Q_r(x))\ge0$, by \eqref{finsler_norm} and inserting the equality $Q_r(x)\cap J_u=(Q_r(x)\cap\p^*A)\cup (Q_r(x)\cap A^{(1)}\cap J_u)$ in \eqref{upper_bound_f}, we get also
$$
\cF(A,u;Q_r(x)) \ge \cS(A,u;Q_r(x)) \ge \bound_1\cH^{n-1}(Q_r(x)\cap J_u). 
$$
Therefore, 
$$
\cH^{n-1}(Q_r(x)\cap J_u) \le \frac{2n\bound_2 + \lambda_0}{\bound_1}\,r^{n-1}.
$$

Next we prove \eqref{min_seq_density_low}. Fix $x\in \cl{J_u^*}.$
For $\tau_0$ given in \eqref{def_tau0},  let $\varsigma_o=\varsigma(\tau_0/2)\in(0,1)$ and $R_o=R_0(\tau_0/2,\bound_1,\bound_2,\lambda_0)>0$ be as in Proposition \ref{prop:lower_density_cG}. By \eqref{lower_dens_functiona} ,
\begin{equation}\label{low_bound_f}
\cF(A,u; Q_{\gamma r}(x)) \ge \bound_1\varsigma_o  (\gamma r)^{n-1}  
\end{equation}
for any $\gamma\in(0,1)$ and $r\in(0,R_o)$ with $Q_r(x)\subset\Omega.$ Let 
$$
\varsigma_*:=\varsigma(\tau_*),\quad 
\vartheta_*:=\vartheta(\tau^*)\quad\text{and}\quad 
R_*:=\min\{R(\tau_*),R_o\}
$$ 
be given by Proposition \ref{prop:functional_decay} for
\begin{equation}\label{tau_startaa}
2\tau_*:=\min\Big\{\frac{\tau_0}{2},\Big(\frac{\bound_1\varsigma_o }{2n\bound_2+\lambda_0}\Big)^2\Big\} 
\end{equation}
By contradiction, if $\cH^{n-1} (Q_r(x)\cap J_u)<\varsigma_* r^{n-1},$ then applying \eqref{almost_minimal_mish} with $\tau=\tau_*$ we get
$$
\cF(A,u;Q_r(x)) \le (1 + \vartheta_*) \Phi(A,u;Q_r(x)).
$$
Hence, by Proposition \ref{prop:functional_decay},
$$
\cF(A,u;Q_{\tau_* r}(x)) \le \tau_*^{n-1/2} \cF(A,u;Q_r(x)) 
$$
so that by \eqref{low_bound_f} and \eqref{upper_bound_f},
$$
\tau_*^{1/2} \ge \frac{\bound_1\varsigma_o }{2n\bound_2 + \lambda_0}, 
$$
which contradicts to \eqref{tau_startaa}.  

Finally, \eqref{essental_closed_jumpoas} follows from the density estimates together with a covering argument.
\end{proof}

From Theorem \ref{teo:density_estimates} we get a partial regularity of minimizers of $\cF$.

\begin{proof}[\textbf{Proof of Theorem \ref{teo:regularity_of_minimizers}}]
(i)-(iii). Let $(\tilde A,\tilde u)\in\admissible$ be a minimizer of $\cF$ and let 
$$
A':=\tilde A^{(1)},\qquad u':=\tilde u\chi_{A'\cup S} + \xi'\chi_{\Omega\setminus A'},
$$
where $\xi'\in(0,1)^n$ is chosen such that $\Omega\cap \p^*A'\subset J_{u'}.$ By \cite[Chapter 15]{Maggi:2012},  $\p A' = \cl{\p^*A'}.$ Clearly, $(A',u')$ is a minimizer of $\cF,$ and by Theorem \ref{teo:density_estimates}, $\cH^{n-1}(\cl{J_{u'}^*}\setminus J_{u'}^*)=0.$ Since $J_{u'}$ is rectifiable,
by \cite[Theorem 2.63]{AFP:2000}, $\cH^{n-1}(J_{u'}\setminus J_{u'}^*)=0$  and hence observing $\Omega\cap \p A' = \Omega\cap \cl{\p^*A'} \subset \cl{J_{u'}}$ we deduce
$$
\cH^{n-1}(A'\setminus \Int{A'}) \le \cH^{n-1}(\p A') \le \cH^{n-1}(\p\Omega) + \cH^{n-1}(J_{u'}) <+\infty.
$$

Now let 
$$
A:=\Int{A'}\quad\text{and}\quad u:=\tilde u\chi_{A\cup S} + \xi'\chi_{\Omega\setminus A}.
$$
Since $|A\Delta A'| \le |\p A'| = 0,$ one has $u = u'$ a.e. in $\Omega\cup S$ and hence  $(A,u)$ is also a minimizer of $\cF.$ Moreover,
$$
\cH^{n-1}(\tilde A^{(1)} \setminus A) \le \cH^{n-1}(\p A') < +\infty,\qquad \cH^{n-1}(J_u\setminus J_{u}^*) = \cH^{n-1}(J_{u'}\setminus J_{u'}^*) =0,
$$
and 
$$
\cH^{n-1}(\cl{J_u^*}\setminus J_{u}^*) = \cH^{n-1}(\cl{J_{u'}^*}\setminus J_{u'}^*) =0.
$$
Thus, (i) follows. The assertions (ii) and (iii) directly follow from the minimality of $(A,u)$ and Theorem \ref{teo:density_estimates}.

(iv)  Finally, assuming that $E\subset A$ is a connected component of (the open set) $A$ with $\cH^{n-1}(\p^* E\cap \Sigma\setminus J_u)=0$ we prove
\begin{equation} \label{u87i32}
|E|\ge \omega_n \,\Big(\frac{\bound_1n}{\lambda_0}\Big)^n\qquad \text{and} \qquad
u = u_0 +a
\end{equation}
for some rigid displacement $a.$
Indeed, for $v:=u\chi_{A\cup\substrate\setminus E} + u_0\chi_E$  we have
$$
\cS(A,u) \ge \cS(A,v)
$$
and 
\begin{equation}\label{smaller_somssj}
\cW(A,u) \ge \cW(A,v).
\end{equation}
In \eqref{smaller_somssj} the equality holds if and only of $\cW(E,u)=0.$   Therefore, by the minimality of $(A,u),$ it follows that $u=u_0$ in $E$ (up to an additive rigid displacement).

To prove the first assertion in \eqref{u87i32} consider the competitor $(A\setminus E,u)\in\admissible.$ Since $(A,u)$ solves \eqref{min_prob_globals_uncons},
$\cF^{\lambda_0}(A,u) \le \cF^{\lambda_0}(A\setminus E,u)$ so that using $u=u_0+a$ in $E$ and the additivity of the surface energy, we get
$$
\int_{\p^*E} \varphi(x,\nu_E)\d\cH^{n-1} \le \lambda_0 |E|.
$$
Using \eqref{finsler_norm} and the isoperimetric inequality in this estimate we obtain
$$
\lambda_0 |E| \ge \bound_1 P(E) \ge \bound_1 n\omega_n^{1/n} |E|^{\frac{n-1}{n}}.
$$
Hence, the first assertion in \eqref{u87i32} follows.
\end{proof}

\appendix 
\section{}

\subsection{Equivalence of volume-constrained and uncontrained penalized minimum  problems}

The following proposition can be seen  an  extension of \cite[Theorem 1.1]{EF:2011}.

\begin{proposition}\label{prop:fusco}
Assume  (H1)-(H3). There exists $\lambda_0>0$ (possibly depending on $\bound_1,$ $\bound_2$ and $\Omega$) with the following property:  $(A,u)\in \admissible$ is a solution of \eqref{min_prob_globals} if and only if $(A,u)$ is also a solution to \eqref{min_prob_globals_uncons} for all $\lambda\ge \lambda_0.$ 
\end{proposition}

\begin{proof}
Note that any minimizer $(A,u)\in\admissible $ of $\cF^\lambda$ with $|A|=\fm$ is also minimizer of $\cF.$ Hence, it suffices to show that  there exists $\lambda_0>0$ such that any minimizer $(A,u)$ of $\cF^\lambda$ for $\lambda>\lambda_0$ satisfies  $|A|=\fm.$ 

Assume by contradiction that there exist a sequence $\lambda_h\to\infty$ and a sequence $(A_h,u_h)\in \admissible$ minimizing $\cF^{\lambda_h}$ such that  $|A_h|\ne \fm.$ Take any $A_0\in BV(\Omega;\{0,1\})$ with $|A| = \fm.$ Then by minimality, $\cF^{\lambda_h}(A_h,u_h) \le \cF^{\lambda_h}(A_0,u_0)=\cF(A_0,u_0)$ for all large $h,$ and hence by \eqref{finsler_norm} and \eqref{hyp:bound_anis},
\begin{equation}\label{gsrtaer}
\sup\limits_{h\ge 1} P(A_h) \le a:=\frac{\cF(A_0,u_0) + \bound_2\cH^{n-1} (\Sigma) + \cH^{n-1} (\p\Omega)}{\bound_1}
\end{equation}
and 
$$
\sup_{h\ge1} \,\lambda_h||A_h| - \fm|\le \cF(A_0,u_0) + \bound_2\cH^{n-1} (\Sigma).
$$ 
This implies $|A_h| \to \fm$ as $h\to\infty.$ By compactness, there exists a finite perimeter set $A\subset\Omega$ and a not relabelled subsequence such that $\chi_{A_h} \to \chi_A$ a.e.\ in $\R^n.$ In particular, $|A|=\fm.$  

Further we assume $|A_h|<\fm$ for all $h;$ the case $|A_h|>\fm$ can be treated analogously. As in the proof of \cite[Theorem 1.1]{EF:2011},  given $\epsilon\in(0,2\epsilon_n)$, where $\epsilon_n>0$ will be chosen later, there exist small  $r>0$ and $x_r\in\Omega$ such that  $B_r(x)\strictlyincluded\Omega$ and
$$
|A\cap B_{r/2}(x_r)|<\epsilon r^n,\qquad |A\cap B_r(x_r)|>\frac{\omega_n r^n}{2^{n+2}}.
$$
For shortness, we suppose that $x_r=0$ we write $B_r:=B_r(x_r).$  Since $A_h\to A$ in $L^1 (\R^n),$ for all large $h$
\begin{equation*} 
|A_h\cap B_{r/2}|<\epsilon r^n,\qquad |A_h\cap B_r|>\frac{\omega_n r^n}{2^{n+2}}. 
\end{equation*}
Let $\Phi:\R^n\to \R^n$ be the bi-Lipschitz homeomorphism, mapping $B_r$ into $B_r,$ defined as
$$
\Phi(x):= 
\begin{cases}
(1 - (2^n-1)\sigma)x, & |x|<\frac r2,\\
x+ \sigma\Big(1 - \frac{r^2}{|x|^2}\Big)x, & \frac{r}{2} \le x < r,\\
x, & |x|\ge r 
\end{cases}
$$
for some $\sigma\in(0,\frac{1}{2^n}).$ 
Recall from \cite[pp. 420-422]{EF:2011} that 
the Jacobian $J\Phi$ of $\Phi$ satisfies 
$$
J\Phi(y) \ge 1+ C_1(n)\,\sigma \qquad y\in B_r\setminus B_{r/2},
$$
for some $C_1(n)>0,$
and 
$$
J\Phi(y) \le 1+ 2^nn\sigma \qquad y\in B_r.
$$
Moreover, the tangential Jacobian $J_{n-1}T_x$ of $\Phi$ on the tangent space $T_x$ of $\p^* A_h$ satisfies 
\begin{equation}\label{tang_jacob}
J_{n-1}T_x \le 1+(1 + 2^n(n-1))\sigma,\qquad x\in B_r \cap \p^* A_h.  
\end{equation}
Set 
\begin{equation}
E_h:= \Phi(A_h),\qquad v_h: =u_h\chi_{A_h\setminus B_r}^{} + u_0\chi_{E_h\cap B_r}^{}.
\label{shaushaush}
\end{equation}
Note that $|E_h|<\fm$ and $E_h\Delta A_h\subset \overline{B_r}.$  Let us estimate
\begin{align}
\cF^{\lambda_h}(A_h,u_h) - \cF^{\lambda_h}(E_h,v_h)
= &\Big(\int_{\cl{B_r}\cap \p^*A_h} \varphi(x,\nu_{A_h})\d\cH^{n-1} -
\int_{\cl{B_r}\cap \p^*E_h} \varphi(x,\nu_{E_h})\d\cH^{n-1}\Big) \nonumber \\
+ & \Big(2\int_{\overline{B_r}\cap J_{u_h}} \varphi(x,\nu_{J_{u_h}})\d\cH^{n-1}
- 2\int_{\overline{B_r}\cap J_{v_h}} \varphi(x,\nu_{J_{v_h}})\Big)\d\cH^{n-1}\nonumber \\
+ & \Big( \int_{B_r\cap A_h} W(x,\str{u_h} - \bM_0)\d x - \int_{B_r\cap E_h} W(x,\str{v_h} - \bM_0)\d x \Big) \nonumber \\
+ & \lambda_h \Big(|E_h| - |A_h|\Big)
:= I_1 + I_2 + I_3+I_4. \label{saisnainsa}
\end{align}
By the definition of $v_h$ and the nonnegativity of $\cW,$ $I_3\ge0$  and since $B_r\cap J_{v_h}=\emptyset,$
$$
I_2\ge - 2\int_{\p B_r} \varphi(x,\nu_{J_{v_h}})\d\cH^{n-1} -2\bound_2\cH^{n-1}(\p B_r) = -2\bound_2 n\omega_n r^{n-1}.
$$
Moreover, by \eqref{tang_jacob} and the area formula as well as from \eqref{finsler_norm} and \eqref{gsrtaer},
\begin{align*}
\int_{B_r\cap \p^* E_h} \varphi(x,\nu_{E_h})\d\cH^{n-1} = &
\int_{B_r\cap \p^* A_h}  \varphi(\Phi(y),\nu_{A_h})J_{n-1} T_y\,d\cH^{n-1}(y)\\
\le & 2\bound_2(1+2^n(n-1)\sigma) \cH^{n-1}(B_r\cap \p^* A_h) 
  \le 2\bound_2(1+(1+2^n(n-1))\sigma)a.
\end{align*}
Moreover, by \eqref{finsler_norm},
$$
\int_{\p B_r\cap \p^* E_h} \varphi(x,\nu_{E_h})d\cH^{n-1} \le 2\bound_2\cH^{n-1}(\p B_r) \le 2n\omega_n \bound_2r^{n-1},
$$
and hence
$$
I_1\ge - 2\bound_2(1+(1+2^n(n-1))\sigma)a - 2n\omega_n \bound_2r^{n-1}.
$$
Finally, repeating the same arguments of Step 4  in the proof of \cite[Theorem 1.1]{EF:2011}, we obtain
\begin{align*}
I_4 \ge \lambda_h \sigma r^n \Big[C_1(n)\,\frac{\omega_n}{2^{n+1}} - C_1(n)\epsilon - (2^n-1)n\epsilon\Big],
\end{align*}
thus, 
\begin{multline}
\cF^{\lambda_h}(A_h,u_h) -  \cF^{\lambda_h}(E_h,v_h) \ge   \lambda_h \sigma r^n \Big[C_1(n)\,\frac{\omega_n}{2^{n+1}} - C_1(n)\epsilon - (2^n-1)n\epsilon\Big] \\
-  2\bound_2(1+(1+2^n(n-1))\sigma)a - 2n\omega_n \bound_2r^{n-1}.  \label{asfgsa}
\end{multline}
Now if we define 
$$
\epsilon_n:=\frac{C_1(n)\omega_n}{2^{n+2} [1+C_1(n) + (2^n-1)n]},
$$
then from \eqref{asfgsa} applied with $\epsilon=\epsilon_n$ we deduce 
\begin{equation*}
\cF^{\lambda_h}(A_h,u_h) -  \cF^{\lambda_h}(E_h,v_h) \ge \lambda_h \sigma \epsilon_n r^n - C  
\end{equation*}
for some $C$ independent of $h.$ Thus, $\cF^{\lambda_h}(A_h,u_h) >\cF^{\lambda_h}(E_h,v_h)$  for all sufficiently large $h,$   which contradicts to the minimality of $(A_h,u_h).$
\end{proof}

\begin{remark}\label{rem:extension_fuscortr}
The same proof of Proposition \ref{prop:fusco} works also with $\cF_p$ and $\cF_{\rm Dir}$ in Theorems \ref{teo:elastic_plow} and \ref{teo:dirichlet_plow}. Indeed, in case $\cF_p,$ for the configuration $(E_h,v_h)$ given by \eqref{shaushaush}, the  equality \eqref{saisnainsa} is written as
\begin{align*}
\cF_p^{\lambda_h}(A_h,u_h) - \cF_p^{\lambda_h}(E_h,v_h)
= & \int_{\cl{B_r}\cap \p^*A_h} \varphi(x,\nu_{A_h})\d\cH^{n-1} - 
\int_{\cl{B_r}\cap \p^*E_h} \varphi(x,\nu_{E_h})\d\cH^{n-1} \\
+ & 2\int_{\overline{B_r}\cap J_{u_h}} \varphi(x,\nu_{J_{u_h}})\d\cH^{n-1}
- 2\int_{\overline{B_r}\cap J_{v_h}} \varphi(x,\nu_{J_{v_h}})\d\cH^{n-1} \\
+ &  \int_{B_r\cap A_h} W_p(x,\str{u_h} - \bM_0)\d x - \int_{B_r\cap E_h} W_p(x,\str{v_h} - \bM_0)\d x \\
+ & \lambda_h \Big(|E_h| - |A_h|\Big)
:= I_1 + I_2 + I_3+I_4.  
\end{align*}
The estimates of $I_1,I_2$ and $I_4$ are the same as above, and for $I_3,$ by \eqref{lower_boundas_papap} we have the bound
$$
I_3 \ge \int_{A_h\cap B_r} W_p(x,\str{u_h} - \bM_0)\d x 
\ge -\int_{B_r}|f|\d x,
$$
which is independent of $h.$

Similarly, in case of $\cF_{\rm Dir}$ we define $v_h$ in \eqref{shaushaush} as 
$$
v_h = u_h\chi_{A_h\setminus B_r}
$$
and the proof runs as in the case of $\cF_p.$
\end{remark}

\subsection{Some properties of GSBD-functions}

\begin{lemma}\label{lem:changing_GSBD_functions}
Let $U$ be an open set and $A\subset BV(U;\{0,1\})$. Assume that $u,v\in GSBD^2(U).$ Then $u\chi_A + v\chi_{U\setminus A}\in GSBD^2(U).$  
\end{lemma}

\begin{proof}
Given $w\in GSBD(U),$ set
$$
\mu_w(B):= \cH^{n-1}(B\cap J_w) + \int_{B} |\str{w}|\d x,\quad B\subset U\quad\text{Borel}.
$$
By \cite[Remark 9.3]{D:2013_jems}, $\mu_w$ satisfies
$$
|D_\xi (\tau(w\cdot \xi))|(B) \le \mu_w(B),\quad B\subset U\quad \text{Borel},
$$
for any $\xi\in\R^n$ and smooth truncations $\tau.$ Now it is not difficult to see that $\omega:=u\chi_A+v\chi_{U\setminus A}$ satisfies
$$
|D_\xi (\tau(w\cdot \xi))|(B) \le \lambda(B),\quad B\subset U\quad\text{Borel}
$$
with the Radon measure
$$
\lambda(B):=\mu_u(B\cap A^{(1)}) + \mu_v(B\cap [U\setminus A]^{(1)})
+\cH^{n-1}(B\cap \p^*A)
$$
in $U.$ Thus, $\omega\in GBD(U).$
The fact that the one dimensional slices of $\hat\omega_y^\xi$ belongs to $SBV_\loc(U_y^\xi)$ follows from the observation that one dimensional slices $[\p^*A]_y^\xi$ of $\p^*A$ partitions the line  $U_y^\xi$ into connected components in each $\hat\omega_y^\xi$ is either $u_y^\xi$ or $v_y^\xi.$ In particular, $\omega\in GSBD(U).$ Finally,
since
$\str{\omega} = \str{u}\chi_A + \str{v}\chi_{U\setminus A},$ it follows that $\omega \in GSBD^2(U).$
\end{proof}

Note that this property does not hold for $GSBV$-functions, because the condition $u\chi_A + v\chi_{U\setminus A}\in GSBV(U)$ requires some regularity of the traces of   $u$ and $v$ along $U\cap \p^*A.$ From Lemma \ref{lem:changing_GSBD_functions} we get 
$
GSBD^2(\Ins{\Omega}) = GSBD^2(\Omega\cup S).
$

\begin{lemma}\label{lem:gsbd_is_sobolev}
Let $n\ge2$ and $D\subset\R^n$ be a connected bounded Lipschitz open set and let $u\in GSBD^2(D)$ be such that $\cH^{n-1}(J_u)=0.$ Then $u\in H^1(D)$ and there exists a rigid displacement $a$ such that 
\begin{equation*}
\|u-a\|_{H^1(D)} \le C_{n,D} \|\str{u}\|_{L^2(D)} 
\end{equation*} 
for some constant $C_{n,D}>0$ depending only on $n$ and $D.$
\end{lemma}

\begin{proof}
Recall that by the Poincar\'e-Korn inequality, for any connected Lipschitz set $U\subset\R^n$ there exists $C_{n,U}>0$ such that
\begin{equation}\label{poincare_korne}
\|v-a\|_{H^1(U)} \le C_{n,U}\|\str{v}\|_{L^2(U)}
\end{equation}
for any $v\in H^1(U)$ and for some rigid displacement $a:\R^n\to\R^n.$ Obviously, $C_{n,U}$ is independent of translation, and assuming $0\in U,$  let us show
\begin{equation}\label{scaled_coeffos}
C_{n,\lambda U} \le C_{n,U}\quad \text{for any $\lambda\in(0,1].$}
\end{equation}
Note that \eqref{poincare_korne} is equivalent to
\begin{equation}\label{poinc_korn1}
\min\limits_{\text{$a$ rigid}} \|v-a\|_{H^1(U)} \le C_{n,U}\|\str{v}\|_{L^2(U)},\quad  v\in H^1(U).
\end{equation}
Fix any $u\in H^1(\lambda U)$ and let $v_\lambda(x):=u(\lambda x).$ Then $v_\lambda\in H^1(U),$ 
$$
\int_U |v_\lambda(x)|^2\d x = \lambda^{-n} \int_{\lambda U} |u(y)|^2\d y
$$
and 
$$ 
\int_U |\nabla v_\lambda(x)|^2\d x = \lambda^{2-n} \int_{\lambda U} |\nabla u(y)|^2\d y,\quad 
\int_U |\str{ v_\lambda(x)}|^2\d x = \lambda^{2-n} \int_{\lambda U} |\str{u(y})|^2\d y.
$$
Then for any rigid displacement $a(x)=\bM x+ b$ we have 
\begin{align*}
\| u - a\|_{H^1(\lambda U)}^2 = &
\lambda^n \|v_\lambda - a_\lambda\|_{L^2(U)}^2 + 
\lambda^{n-2} \|\nabla v_\lambda - \bM\|_{L^2(U)}^2
\le \lambda^{n-2} \|v_\lambda - a_\lambda\|_{H^1(U)}^2,
\end{align*}
where  $a_\lambda(x) = \lambda \bM x + b.$ Now taking $a_\lambda,$ satisfying \eqref{poincare_korne} with $v=v_\lambda,$ we have 
\begin{align*}
\| u - a\|_{H^1(\lambda U)}^2 
\le \lambda^{n-2}\|v_\lambda - a_\lambda\|_{H^1(U)}^2 
\le C_{n,U}^2\lambda^{n-2}\|\str{v_\lambda}\|_{L^2(U)}^2 =
 C_{n,U}^2\|\str{u}\|_{L^2(\lambda U)}^2,
\end{align*}
and thus, from \eqref{poinc_korn1} we get \eqref{scaled_coeffos}.

Now we prove the lemma. By \cite[Proposition A.3]{HP:2021_arxiv}, $u\in H^1_\loc(D)$ and hence, in view of  \eqref{poincare_korne} we just need to show that $u\in H^1(D).$

{\it Step 1.} First assume additionally that $D$ is simply connected and $0$ is in the interior of $D.$  
Consider  the sequence
$$
D_i=(1-2^{-i})D,\quad i\in\N,
$$
of rescalings of $D.$ Since $D_i\strictlyincluded D$ and $u\in H^1(D_i),$ by \eqref{poincare_korne} and \eqref{scaled_coeffos}, there exists a rigid displacement $a_i:\R^n\to\R^n$ such that
\begin{equation}\label{poincar_korna_ianqq}
\|u-a_i\|_{H^1(D_i)} \le C_{n,D}\|\str{u}\|_{L^2(D_i)}. 
\end{equation}
Consider the sequence $\{a_i\}.$ Since $D_1\subset D_i\subset D,$ by \eqref{poincar_korna_ianqq},
$$
\|a_i-a_1\|_{H^1(D_1)} \le \|u-a_i\|_{H^1(D_i)} + \|u-a_1\|_{H^1(D_1)}  \le  C_{n,D}\|\str{u}\|_{L^2(D)}.
$$
Thus, $\{a_i\}$ is uniformly bounded in $H^1(D_1).$ Since $a_i$ are linear, up to a subsequence, $a_i\to a_0$ in $H_\loc^1(\R^n)$ and $a_i\to a_0$ a.e.\  in $\R^n$ for some rigid displacement $a_0.$ Hence, by \eqref{poincar_korna_ianqq},
\begin{align*}
\|u-a_0\|_{H^1(D_i)} = \lim\limits_{j\to+\infty} \|u-a_j\|_{H^1(D_i)} \le \limsup\limits_{j\to+\infty} \|u-a_j\|_{H^1(D_j)} 
\le & C_{n,D}\,\limsup\limits_{j\to+\infty} \|\str{u}\|_{L^2(D_j)}
\end{align*}
Since $D_j\nearrow D$ and $\str{u}\in L^2(D),$ by the monotone convergence theorem,
$$
\|u-a_0\|_{H^1(D_i)}  \le C_{n,D}\|\str{u}\|_{L^2(D)}.
$$
Letting $i\to+\infty$ in this inequality and using again the monotone convergence theorem we get 
$
u-a_0\in H^1(D),
$ and thus, $u\in H^1(D).$
\smallskip

{\it Step 2.} Now consider the general case.  Since $D$ is Lipschitz, for any $x\in \p D$ there exists a cylinder $R_x$ such that $D\cap R_x$ is a subgraph of a Lipchitz function. In particular, $D\cap R_x$ is Lipschitz and simply connected. For $x\in D$ let $R_x$ be largest cube centered at $x$ and contained in $D.$ Then 
$\cl{D} \subseteq \bigcup_x R_x,$ and hence by the compactness of $\cl{D}$, there exist  finitely many points $x_1,\ldots,x_m$ such that $\cl{D} \subset \bigcup_{j=1}^m R_{x_j}.$ Since $R_{x_j}\cap D$ is simply connected, by Step 1, $u\in H^1(R_{x_j}\cap D)$ and there exists a rigid displacement $a_j$ such that
$$
\|u - a_j\|_{H^1(R_{x_j}\cap D)} \le C_{n,R_{x_j}\cap D}
\|\str{u}\|_{L^2(R_{x_j}\cap D)}.
$$
Thus,
$$
\|u\|_{H^1(D)}^2 \le 
\sum\limits_{j=1}^m \|u\|_{H^1(D\cap R_{x_j})}^2
\le 
2\sum\limits_{j=1}^m \|u - a_j\|_{H^1(D\cap R_{x_j})}^2 
+
2\sum\limits_{j=1}^m \|a_j\|_{H^1(D\cap R_{x_j})}^2 <+\infty.
$$
\end{proof}

\subsection*{Acknowledgments} 

Sh. Kholmatov acknowledges support from the Austrian Science Fund (FWF) projects M2571-N32 and P33716. P. Piovano acknowledges support from the Austrian Science Fund (FWF) projects P 29681 and TAI 293, from the Vienna Science and Technology Fund (WWTF) together with the City of Vienna and Berndorf Privatstiftung through Project MA16-005, and from BMBWF through the OeAD-WTZ project HR 08/2020. Furthermore, P. Piovano is member of the Italian ``Gruppo Nazionale per l'Analisi Matematica, la Probabilit\`a e le loro Applicazioni'' (GNAMPA) and has received funding from the  INdAM - GNAMPA  2022 project CUP: E55F22000270001 and 2023 Project Codice CUP: E53C22001930001.
Finally, P. Piovano is grateful for the support received from Wolfgang Pauli Institute (WPI) in Vienna, Austria and as Visiting Professor and Excellence Chair at the Okinawa Institute of Science and Technology (OIST) in Japan.

\subsection*{Data availability}

The manuscript has no associated data.

\subsection*{Competing interests}

There are no conflict of interests related to the work.


\begin{thebibliography}{99}

\bibitem{ADT:2017} S. Almi, G. Dal Maso,  R. Toader: 
A lower semicontinuity result for a free discontinuity functional with a boundary term. J. Math. Pures Appl. {\bf 108} (2017), 952--990.

\bibitem{Ambrosio:1989} L. Ambrosio: A compactness theorem for a new class of functions of bounded variation. Boll. U. M. I. {\bf 3-B} (1989), 857--881.

\bibitem{AFP:2000} L. Ambrosio, N. Fusco, D. Pallara:  Functions of Bounded Variation and Free Discontinuity problems. Oxford University Press, New York 2000.

\bibitem{AT:1972} R. Asaro, W. Tiller:  Interface morphology development during stress corrosion cracking: Part I. Via surface diffusion. Metall. Trans. \textbf{3} (1972), 1789--1796.  
 
\bibitem{Babadjian:2016} J.-F. Babadjian, D. Henao: 
Reduced models for linearly elastic thin films allowing for fracture, debonding or delamination. Interface Free Bound.  \textbf{18} (2016), 545--578.

\bibitem{BGZ:2015} P. Bella, M. Goldman, B. Zwicknagl:  
Study of island formation in epitaxially strained films on unbounded domains. Arch. Rational Mech. Anal. {\bf 218} (2015), 163--217.
 
\bibitem{BNH:2017} G. Bellettini, M. Novaga, Sh. Kholmatov:  Minimizers of anisotropic perimeters with cylindrical norms. Comm. Pure Appl. Anal. {\bf 16} (2017), 1427--1454.  
 
\bibitem{Bonaci:2015} M. Bonacini: Stability of equilibrium configurations for elastic films in two and three dimensions. Adv. Calc. Var. {\bf 8} (2015), 117--153.
 
\bibitem{BCh:2002} E. Bonnetier, A. Chambolle:  Computing the equilibrium configuration of epitaxially strained crystalline films.  SIAM J. Appl. Math. {\bf 62} (2002), 1093--1121.

\bibitem{BFM:2008} B. Bourdin, G. Francfort, J.-J. Marigo:  The variational approach to fracture. Springer, Amsterdam, 2008.
 
\bibitem{CCF:2016} A. Chambolle, S. Conti, G. Francfort: Korn-Poincar\'e inequalities for functions with a small jump set. Indiana Univ. Math. J. {\bf 65} (2016), 1373--1399. 

\bibitem{CCI:2017} A. Chambolle, S. Conti, F. Iurlano:  Approximation of functions with small jump sets  and existence of strong minimizers of Griffith's energy. J. Math. Pures Appl. {\bf 128} (2019), 119--139.

\bibitem{ChC:2019_arxiv} A. Chambolle, V. Crismale:  Existence of strong solutions to the Dirichlet problem for the Griffith energy.  Calc. Var. Partial Differential Equations {\bf 58} (2019), 136. 

\bibitem{ChC:2020_jems}  A. Chambolle, V. Crismale:  Compactness and lower semicontinuity in GSBD. J. Eur. Math. Soc. (JEMS) {\bf 23} (2021), 701--719.

\bibitem{ChC:2020_arxiv} A. Chambolle, V. Crismale:  Equilibrium configurations for nonhomogeneous linearly elastic materials with surface discontinuities. To appear in Ann. Sc. Norm. Super. Pisa Cl. Sci. (2022).

\bibitem{ChS:2007} A. Chambolle, M. Solci: Interaction of a bulk and a surface energy with a geometrical constraint. SIAM J. Math. Anal. {\bf 39}(2007), 77--102.

\bibitem{CF:2020_arma} V. Crismale, M. Friedrich:  Equilibrium configurations for epitaxially strained films and material voids in three-dimensional linear elasticity. Arch. Rational Mech. Anal. {\bf 237} (2020), 1041--1098.

\bibitem{D:2013_jems} {\sc  G. Dal Maso:} Generalised functions of bounded deformation. J. Eur. Math. Soc. {\bf 15} (2013), 1943--1997.

\bibitem{DFT:2005} G. Dal Maso, G.A. Francfort, R. Toader: Quasistatic crack growth in nonlinear elasticity. Arch. Ration. Mech. Anal. {\bf176} (2005), 165--225.


\bibitem{DT:2002} G. Dal Maso, R. Toader: A model for the quasi-static growth of brittle fractures: existence and approximation results. Arch. Ration. Mech. Anal. {\bf162} (2002), 101--135.


\bibitem{D:2001} A. Danescu: The Asaro-Tiller-Grinfeld instability revisited. Int. J. Solids Struct. {\bf 38} (2001), 4671--4684.

\bibitem{DP:2018_1}  E. Davoli,  P. Piovano:  Derivation of a heteroepitaxial thin-film model. Interface Free Bound. {\bf 22} (2020), 1--26.

 
\bibitem{DP:2018_2} E. Davoli E., P. Piovano: Analytical validation of the Young-Dupr\'e law for epitaxially-strained thin films.  Math.\ Models Methods Appl.\ Sci., \textbf{29} (2019), 2183--2223.

 

\bibitem{DiBen:2002_book} E. DiBenedetto:  Real Analysis. Birkh\"auser, Basel, 2002.
 
\bibitem{DBD:1983} E. De Giorgi, G. Buttazzo, G. Dal Maso: On the lower semicontinuity of certain integral functionals. Atti Acc. Naz. Lincei. Cl. Sc. Fis. Mat. Nat. Rend. {\bf 74} (1983), 274--282.

\bibitem{DPhM:2015} G. De Philippis, F.  Maggi: Regularity of free boundaries in anisotropic capillarity problems and the validity of Young's law. Arch. Rational Mech. Anal. {\bf 216} (2015), 473--568. 

\bibitem{Deng:1995}  X. Deng:  Mechanics of debonding and delamination in composites: Asymptotic studies. Compos. Eng. {\bf 5} (1995),  1299--1315.

\bibitem{EF:2011}  L. Esposito,  N. Fusco:  A remark on a free interface problem with volume constraint. J. Convex Anal. {\bf18} (2011), 417--426.

\bibitem{FFLM:2007} I. Fonseca, N. Fusco, G. Leoni, 
M. Morini: Equilibrium configurations of epitaxially strained crystalline films: existence and regularity results. Arch. Rational Mech. Anal.  \textbf{186} (2007), 477--537.

\bibitem{FFLM:2011}  I. Fonseca, N.  Fusco, G.Leoni, V. Millot: Material voids in elastic solids with anisotropic surface energies. J. Math. Pures Appl. \textbf{96} (2011), 591--639.        

\bibitem{FGL:2019} G. Francfort, A. Giacomini, O. Lopez-Pamies: Fracture with healing: a first step towards a new view of cavitation. Anal. PDE {\bf 12} (2019), 417--447.

\bibitem{FL:2003} G.A. Francfort, C.J. Larsen: Existence and convergence for quasi-static evolution in brittle fracture. Commun. Pure Appl. Math. {\bf56} (2003), 1465--1500.

\bibitem{FG:2004} E. Fried, M. Gurtin: A unified treatment of evolving interfaces accounting for small deformations and atomic transport with emphasis on grain-boundaries and epitaxy. Adv. Appl. Mech. {\bf 40} (2004), 1--177.

\bibitem{FPS:2021} M. Friedrich, M. Perugini, F. Solombrino: Lower semicontinuity for functionals defined on piecewise rigid functions and on GSBD. J. Funct. Anal. {\bf 280} (2021), 108929.

\bibitem{FS:2018} M. Friedrich, F. Solombrino:  Quasistatic crack growth in 2d-linearized elasticity. Ann. I. H. Poincar\'e -- AN {\bf35} (2018), 27--64.

\bibitem{Gi:1984} E. Giusti:  Minimal Surfaces and Functions of Bounded Variation.  Birkh\"auser, Boston, 1984. 

\bibitem{GZ:2014} M. Goldman, B. Zwicknagl:  Scaling law and reduced models for epitaxially  strained crystalline films. SIAM J. Math. Anal. {\bf 46} (2014), 1--24.

\bibitem{Gr:1993} M.A. Grinfeld:  The stress driven instabilities in crystals: mathematical models and physical manifestations.  J. Nonlinear Sci. \textbf{3} (1993), 35--83. 


\bibitem{HS:1991} J. Hutchinson, Z. Suo:  Mixed mode cracking in layered materials.  Adv. Appl. Mech. {\bf 29} (1991), 63--191.

\bibitem{HP:2020_arma} Sh. Kholmatov, P. Piovano: A unified model for stress-driven rearrangement instabilities. Arch. Rational Mech. Anal. {\bf 238} (2020), 415--488. 


\bibitem{HP:2021_arxiv} Sh. Kholmatov, P. Piovano:  Existence of minimizers for the SDRI model in 2d: Wetting and dewetting regime with mismatch strain. Adv. Calc. Var. {\bf17} (2024), 673--725.

\bibitem{KGT:2009} D. Kim, A.L. Giermann,
C.V. Thompson: Solid-state dewetting of patterned thin films. Appl. Phys. Lett. {\bf 95} (2009), 251903.

\bibitem{KP:2021} L. Kreutz,  P. Piovano: Microscopic validation of a variational model of epitaxially strained crystalline films. SIAM J. Math. Anal. \textbf{53} (2021), 453--490. 

\bibitem{Baldelli:2013}  A.A. Le\'on Baldelli, B. Bourdin,  J.-J. Marigo, C. Maurini:  Fracture and debonding of a thin film on a stiff substrate: analytical and  numerical solutions of a one-dimensional variational model.  Cont. Mech. Thermodyn. \textbf{25} (2013), 243--268.

\bibitem{Baldelli:2014}  A.A. Le\'on Baldelli, J.-F. Babadjian, B. Bourdin, D. Henao, C. Maurini: A variational model for fracture and debonding of thin  films under in-plane loadings. J. Mech. Phys. Solids \textbf{70} (2014), 320--348. 

\bibitem{LP:2025} R. Llerena,  P. Piovano:
Variational modeling of multilayer films with coherent and incoherent interlayer interfaces. Continuum Mech. Thermodyn. {\bf 37} (2025), 32.



\bibitem{Maggi:2012} F. Maggi: Sets of Finite Perimeter and Geometric Variational Problems: An Introduction to Geometric Measure Theory.  Cambridge University Press, Cambridge, 2012.

\bibitem{Morrey:2008} C.B. Morrey Jr.: Multiple Integrals in Calculus of Variations. Classics in Mathematics, Springer, Berlin, 2008.

\bibitem{P:2012}   P. Piovano: Evolution and Regularity Results for Epitaxially Strained Thin Films and Material Voids. Ph.D. Thesis - Carnegie Mellon University, 2012. 

\bibitem{PV:2022} P. Piovano, I. Vel\v{c}i\'c: Microscopical Justification of Solid-State Wetting and Dewetting. J. Nonlinear Sci. \textbf{32} (2022), 32.


\bibitem{PV:2023} P. Piovano, I. Vel\v{c}i\'c: Microscopical justification of the Winterbottom problem for well-separated lattices. Nonlinear Analysis. \textbf{231} (2023), 113113.

\bibitem{R:1992} G. Reiter: Dewetting of thin polymer films. Phys. Rev. Lett. {\bf68} (1992), 75--78.

\bibitem{R:1994} G. Reiter: Dewetting as a probe of polymer mobility in thin films. Macromolecules {\bf27} (1994), 3046--3052.


\bibitem{RG:2021} F. Ruffino, M.G. Grimaldi: Nano-shaping of gold particles on silicon carbide substrate from solid-state to liquid-state dewetting. Surfaces and Interfaces {\bf24}
(2021), 101041.

\bibitem{SMV:2004} M. Siegel, M. Miksis,  P. Voorhees:  Evolution of material voids for highly anisotropic surface energy. J. Mech. Phys. Solids \textbf{52} (2004), 1319--1353.

\bibitem{Spencer:1999}  B. Spencer:  Asymptotic derivation of the guled-wetting-layer model and the contact-angle condition for Stranski-Krastanov islands. J. Mech. Phys. Solids {\bf 52} (2004), 1319--1353.

\bibitem{Srolovitz:1989} D.J. Srolovitz: On the stability of surfaces of stressed solids. Acta Metallurgica. {37-2} (1989), 621--625.

\bibitem{SS:1986} D.J. Srolovitz, S.A. Safran: Capillary instabilities in thin films. II. Kinetics. J. Appl. Phys. {\bf60} (1986), 255.


\bibitem{WL:2003} H. Wang,  Z. Li:  The instability of the diffusion-controlled grain-boundary void in stressed solid. Acta Mech. Sinica  \textbf{19} (2003),  330--339.	

\bibitem{Winterbottom}   W.L. Winterbottom: Equilibrium shape of a small particle in contact with a foreign substrate. Acta Metallurgica \textbf{15} (1967), 303--310.

\bibitem{Wulff01} G. Wulff: Zur Frage der Geschwindigkeit des Wastums und der Aufl\"osung der Kristallflachen. Krystallographie und Mineralogie. Z. Kristallner. \textbf{34} (1901), 449--530.
	
\bibitem{Xia:2000} Z.C. Xia,  J.W. Hutchinson: Crack patterns in thin films. J.  Mech. Phys. Solids \textbf{48} (2000), 1107--1131. 


\end{thebibliography}
\end{document}